\documentclass[11pt,letterpaper]{amsart}
\usepackage[foot]{amsaddr}

\usepackage{mathtools}
\usepackage{amsmath}
\usepackage[T1]{fontenc}
\usepackage[latin9]{inputenc}
\usepackage{geometry}
\usepackage{wrapfig}
\usepackage{multicol}
\usepackage{graphicx}
\usepackage{soul}
\usepackage{xcolor}
\usepackage{amssymb}
\usepackage{placeins}
\usepackage{bbm}
\usepackage{cite}
\usepackage{caption}
\usepackage{enumerate}
\usepackage{afterpage}
\usepackage{enumitem}
\usepackage{bmpsize}
\usepackage{hyperref}
\usepackage{tabu}
\usepackage{enumitem}
\numberwithin{equation}{section}
\usepackage{stmaryrd}
\usepackage{tikz}
\usetikzlibrary{matrix,graphs,arrows,positioning,calc,decorations.markings,shapes.symbols}
\usepackage{ifthen}
\usetikzlibrary{math}
\usetikzlibrary{matrix,graphs,arrows,positioning,calc,decorations.markings,shapes.symbols}

\makeatletter
\renewcommand{\email}[2][]{%
  \ifx\emails\@empty\relax\else{\g@addto@macro\emails{,\space}}\fi%
  \@ifnotempty{#1}{\g@addto@macro\emails{\textrm{(#1)}\space}}%
  \g@addto@macro\emails{#2}%
}
\makeatother

\newtheorem{theorem}{Theorem}[section]
\newtheorem{lemma}[theorem]{Lemma}
\newtheorem{proposition}[theorem]{Proposition}

{ \theoremstyle{definition}
\newtheorem{definition}[theorem]{Definition}}
{ \theoremstyle{remark}
\newtheorem{remark}[theorem]{Remark}}

\newcommand{\weyl}{W^\circ}

\newcommand{\im}{\mathsf{i}}
\newcommand{\Real}{\mathsf{Re}\hspace{0.5mm}}
\newcommand{\Imag}{\mathsf{Im}\hspace{0.5mm}}

\newcommand{\cev}[1]{\reflectbox{\ensuremath{\vec{\reflectbox{\ensuremath{#1}}}}}}

\newcommand{\parF}{\mathcal{P}_{\mathrm{fin}}}
\newcommand{\parP}{\mathcal{P}_{\mathrm{pos}}}

\title{Structural properties of the Airy wanderer line ensembles}
\date{\today}
\author{Evgeni Dimitrov} 
\begin{document}

\begin{abstract}
The Airy wanderer line ensembles are infinite-parameter generalizations of the classical Airy line ensemble that arise naturally as scaling limits of inhomogeneous (spiked) models in the Kardar-Parisi-Zhang universality class. In this paper, we establish several structural properties of these ensembles. Our results show their laws depend continuously on the parameters, which encode the asymptotic slopes of the ensemble's curves near positive and negative infinity. We further prove that these ensembles admit multiple monotone couplings with respect to their parameters. Finally, we show that the Airy wanderer line ensembles are extreme points in the space of all Brownian Gibbsian line ensembles on the real line.
\end{abstract}

\maketitle

\tableofcontents

%
%
\section{Introduction and main results}\label{Section1}

%
%
\subsection{Preface}\label{Section1.1} The {\em Airy line ensemble} $\mathcal{A}$ is a sequence $\{\mathcal{A}_{i}: \mathbb{R} \to \mathbb{R}\}_{i \geq 1}$ of random continuous functions, which are defined on a common probability space and are almost surely strictly ordered in the sense that $\mathcal{A}_{i}(t) > \mathcal{A}_{i+1}(t)$ for all $i \geq 1$ and $t \in \mathbb{R}$. The Airy line ensemble arises as the edge-scaling limit of time-dependent Wigner matrices (including {\em Dyson Brownian motion}) \cite{Sod15}, lozenge tilings \cite{AH21}, as well as various integrable models of non-intersecting random walkers and last passage percolation \cite{DNV19}. A formal construction of the Airy line ensemble was given in \cite{CorHamA}, where $\mathcal{A}$ was defined as the weak edge-scaling limit of the {\em Brownian watermelon}, although various projections of the ensemble had been identified earlier. For example, the top curve $\mathcal{A}_1$ is a stationary process, called the {\em Airy process}, which was constructed as the weak limit of paths in the {\em polynuclear growth model} \cite{J03, Spohn}. In addition, the one-point marginal of $\mathcal{A}_1$ is the {\em (GUE) Tracy--Widom distribution} from random matrix theory \cite{TWPaper}. Due to its appearance as a universal scaling limit, and its role in the construction of the {\em Airy sheet} in \cite{DOV22}, the Airy line ensemble has become a central object of interest in the {\em Kardar--Parisi--Zhang (KPZ)} universality class \cite{CU2}.

In \cite{dimitrov2024airy, dimitrov2024tightness} we constructed infinite-parameter generalizations $\mathcal{A}^{a,b,c} = \{\mathcal{A}^{a,b,c}_{i}: \mathbb{R} \to \mathbb{R}\}_{i \geq 1}$ of the Airy line ensemble, which we call the {\em Airy wanderer line ensembles}, as weak limits of inhomogeneous {\em Schur processes}, or equivalently {\em geometric last passage percolation (LPP)}; see \cite{DY25} for a brief explanation of the distributional equivalence between the two. This construction generalizes the finite-parameter ensembles previously obtained in \cite{CorHamA}, based on correlation kernel formulas from \cite{AFM}, and relies on a careful interpretation of the correlation kernel formulas from \cite{BP08}. The Airy wanderer line ensembles are believed to arise as universal scaling limits of models in the KPZ universality class, which have inhomogeneities or spikes. For example, the one-point marginals of the top curves of these ensembles are given by the {\em Baik--Ben Arous--P{\'e}ch{\'e} (BBP) distributions}, which appear in the asymptotics of sample covariance matrices \cite{BBP05}, finite-rank perturbations of random matrices \cite{Peche06}, asymmetric exclusion processes \cite{BFS09,IS07}, directed random polymers \cite{BCD21, TV20}, and exponential LPP \cite{BP08}.  

The goal of this paper is to initiate a detailed study of the Airy wanderer line ensembles $\mathcal{A}^{a,b,c}$ and investigate some of their remarkable properties. Specifically, we prove the following statements.
\begin{enumerate}
\item[I.] {\bf Symmetries.} We show that the family of Airy wanderer line ensembles is closed under vertical shifts, and horizontal reflections across the origin. We further explain how the parameters of the model change under these transformations. See Proposition \ref{S13P1} for the precise statements.
\item[II.] {\bf Continuity.} We show that the laws of the Airy wanderer line ensembles vary continuously with their parameters. See Proposition \ref{S13P2}.
\item[III.] {\bf Monotonicity.} We show that the Airy wanderer line ensembles admit {\em multiple} monotone couplings in their parameters. See Theorem \ref{Thm.MonCoupling}.
\item[IV.] {\bf Asymptotic slopes.} We show that under a parabolic shift the curves of the Airy wanderer line ensembles have asymptotic slopes near positive and negative infinity, which are encoded by the parameters of the model. See Theorem \ref{Thm.Slopes}.
\item[V.] {\bf Extremality.} We show that the Airy wanderer line ensembles are extreme points in the space of all line ensembles on $\mathbb{R}$ that satisfy the Brownian Gibbs property. See Theorem \ref{Thm.Extreme} for the precise statement.
\end{enumerate} 
The rest of the introduction is structured as follows. In Section \ref{Section1.2} we formally define the Airy wanderer line ensembles, and recall some results from \cite{dimitrov2024airy}. In Section \ref{Section1.3} we state some of the more basic properties of these ensembles, including the aforementioned symmetries and continuity. The monotonicity, asymptotic description near positive/negative infinity, and extremality are discussed in Section \ref{Section1.4}, and form the main results of the paper. Section \ref{Section1.5} contains an outline of the paper, and discusses some of the key ideas behind our arguments.

%
%
\subsection{The Airy wanderer line ensembles}\label{Section1.2} The goal of this section is to give a formal definition of the Airy wanderer line ensembles constructed in \cite{CorHamA,dimitrov2024airy}. Our exposition closely follows that of \cite[Section 1.2]{dimitrov2024airy}. We begin by formally introducing the notion of a line ensemble.
\begin{definition}\label{Def.LE} Let $\Sigma \subseteq \mathbb{Z}$ and $\Lambda \subseteq \mathbb{R}$ be an interval. We let $C (\Sigma \times \Lambda)$ denote the space of continuous functions $f: \Sigma \times \Lambda \rightarrow \mathbb{R}$ with the topology of uniform convergence over compact sets, and corresponding Borel $\sigma$-algebra.

A {\em $\Sigma$-indexed line ensemble $\mathcal{L}$ on $\Lambda$} is a random variable defined on a probability space $(\Omega, \mathcal{F}, \mathbb{P})$ that takes values in $C (\Sigma \times \Lambda)$. Setting $\mathcal{L}_i(t) = \mathcal{L}(i, t)$ for $i \in \Sigma$, $t \in \Lambda$, we have that $\mathcal{L}_i$ are random continuous functions, sometimes called {\em lines}, and we write $\mathcal{L} = \{\mathcal{L}_i\}_{i \in \Sigma}$.
\end{definition}
\begin{remark}\label{Rem.LE} A comprehensive treatment of line ensembles can be found in \cite[Section 2]{DEA21}.
\end{remark}

We next introduce the parameter space of the Airy wanderer line ensembles.
\begin{definition}\label{DLP} 
We assume we are given four sequences of non-negative real numbers $\{a_i^+\}_{ i \geq 1}$, $\{a_i^-\}_{ i \geq 1}$, $\{b_i^+\}_{ i \geq 1}$, $\{b_i^-\}_{ i \geq 1}$, such that 
\begin{equation}\label{ParProp}
\sum_{i = 1}^{\infty} (a_i^+ + a_i^- + b_i^+ + b_i^-) < \infty \mbox{ and } a_{i}^{\pm} \geq a_{i+1}^{\pm},  b_{i}^{\pm} \geq b_{i+1}^{\pm} \mbox{ for all } i \geq 1,
\end{equation}
as well as two real parameters $c^+, c^-$. We let $J_a^{\pm} = \inf \{ k \geq 1: a_{k}^{\pm} = 0\} - 1$ and $J_b^{\pm} = \inf \{ k \geq 1: b_{k}^{\pm} = 0\} - 1$. In words, $J_a^{\pm}$ is the largest index $k$ such that $a_{k}^{\pm} > 0$, with the convention that $J_a^{\pm} = 0$ if all $a_k^{\pm} = 0$ and $J_a^{\pm} = \infty$ if all $a_k^{\pm} > 0$, and analogously for $J_b^{\pm}$. For future reference, we denote the set of parameters satisfying the above conditions such that $c^- = 0$ and $J_a^- + J_b^- < \infty$ by $\parF$, and the subset of $\parF$ such that $c^+ = J_a^- = J_b^- = 0$ by $\parP$.

Lastly, we define 
$$\underline{a} = \begin{cases} 0  &\hspace{-3.5mm}  \mbox{ if } a_1^- + b_1^- > 0 \mbox{ or } c^- \neq 0, \\   \infty & \hspace{-3.5mm} \mbox{ if }   a_1^- + b_1^- =  c^- = 0\mbox{ and }  a_1^+ = 0, \\ 1/a_1^+ & \hspace{-3.5mm}  \mbox{ if }a_1^- + b_1^- = c^- =  0 \mbox{ and } a_1^+ > 0, \end{cases} \hspace{1mm} \mbox{ and }\hspace{1mm} \underline{b} = \begin{cases} 0 &\hspace{-3.5mm}  \mbox{ if }  a_1^- + b_1^- > 0 \mbox{ or } c^- \neq 0, \\  -\infty & \hspace{-3.5mm} \mbox{ if } a_1^- + b_1^- = c^- =  0 \mbox{ and } b_1^+ = 0, \\ -1/b_1^+ & \hspace{-3.5mm} \mbox{ if } a_1^- + b_1^- = c^- =  0 \mbox{ and } b_1^+ > 0. \end{cases}$$
Observe that $\underline{a} \in [0, \infty]$ and $\underline{b} \in [- \infty, 0]$. 
\end{definition}

For $z \in \mathbb{C} \setminus \{0\}$, we define the function
\begin{equation}\label{DefPhi}
\Phi_{a,b,c}(z) = e^{c^+z + c^-/ z} \cdot \prod_{i = 1}^{\infty} \frac{(1 + b_i^+ z) (1 + b_i^- /z)}{(1 - a_i^+ z) ( 1 - a_i^{-}/z)}.
\end{equation}
From (\ref{ParProp}) and \cite[Chapter 5, Proposition 3.2]{Stein}, we have that the above defines a meromorphic function on $\mathbb{C} \setminus \{0\}$ whose zeros are at $\{-(b_i^+)^{-1}\}_{i =1}^{J_b^+}$ and $\{- b_i^-\}_{i =1}^{J_b^-}$, while its poles are at $\{(a_i^+)^{-1}\}_{i =1}^{J_a^+}$ and $\{ a_i^-\}_{i =1}^{J_a^-}$. We also observe that $\Phi_{a,b,c}(z) $ is analytic in $\mathbb{C} \setminus [\underline{a}, \infty)$, and its inverse is analytic in $\mathbb{C} \setminus (-\infty, \underline{b}]$, where $\underline{a}, \underline{b}$ are as in Definition \ref{DLP}.\\

The following definitions present the Airy wanderer kernel, originally introduced in \cite{BP08}, starting with the contours that appear in it. 
\begin{definition}\label{DefContInf}  Fix $a \in \mathbb{R}$. We let $\Gamma_a^+$ denote the union of the contours $\{a + y e^{\pi \im /4}\}_{y \in \mathbb{R}_+}$ and $\{a + y e^{-\pi \im/4}\}_{y \in \mathbb{R}_+}$, and $\Gamma_a^-$ the union of the contours $\{a + y e^{ 3\pi \im /4}\}_{y \in \mathbb{R}_+}$ and $\{a + y e^{-3 \pi \im/4}\}_{y \in \mathbb{R}_+}$. Both contours are oriented in the direction of increasing imaginary part.
\end{definition}

\begin{definition}\label{3BPKernelDef} Assume the same notation as in Definition \ref{DLP}. For $t_1, t_2,x_1,x_2 \in \mathbb{R}$, we define 
\begin{equation}\label{3BPKer}
\begin{split}
&K_{a,b,c} (t_1, x_1; t_2, x_2) = K^1_{a,b,c} (t_1, x_1; t_2, x_2) +  K^2_{a,b,c} (t_1, x_1; t_2, x_2) +  K^3_{a,b,c} (t_1, x_1; t_2, x_2), \mbox{ with } \\
& K^1_{a,b,c} (t_1, x_1; t_2, x_2) = \frac{1}{2\pi \im} \int_{\gamma}dw \cdot  e^{(t_2 - t_1)w^2 + (t_1^2 - t_2^2) w + w (x_2-x_1) + x_1 t_1 - x_2 t_2 - t_1^3/3 + t_2^3/3},    \\
&K^2_{a,b,c} (t_1, x_1; t_2, x_2)  = -  \frac{{\bf 1}\{ t_2 > t_1\} }{\sqrt{4\pi (t_2 - t_1)}} \cdot e^{ - \frac{(x_2 - x_1)^2}{4(t_2 - t_1)} - \frac{(t_2 - t_1)(x_2 + x_1)}{2} + \frac{(t_2 - t_1)^3}{12} }, \\
& K^3_{a,b,c} (t_1, x_1; t_2, x_2) = \frac{1}{(2\pi \im)^2} \int_{\Gamma_{\alpha }^+} d z \int_{\Gamma_{\beta}^-} dw \frac{e^{z^3/3 -x_1z - w^3/3 + x_2w}}{z + t_1 - w - t_2} \cdot \frac{\Phi_{a,b,c}(z + t_1) }{\Phi_{a,b,c}(w + t_2)}.
\end{split}
\end{equation}
In (\ref{3BPKer}) $\alpha, \beta  \in \mathbb{R}$ are such that $\alpha + t_1 < \underline{a}$ and $\beta + t_2 > \underline{b}$, the function $\Phi_{a,b,c}$ is as in (\ref{DefPhi}) and the contours of integration in $K^3_{a,b,c}$ are as in Definition \ref{DefContInf}. If $\Gamma^+_{\alpha + t_1} (=t_1 + \Gamma^+_{\alpha})$ and $\Gamma^-_{\beta + t_2} (= t_2 + \Gamma^-_{\beta} )$ have zero or one intersection point, we take $\gamma = \emptyset$ and then $K^1_{a,b,c} \equiv 0$. Otherwise, $\Gamma^+_{\alpha + t_1} $ and $\Gamma^-_{\beta + t_2}$ have exactly two intersection points, which are complex conjugates, and $\gamma$ is the straight vertical segment that connects them with the orientation of increasing imaginary part. See Figure \ref{S12}.
\end{definition}
\begin{figure}[h]
    \centering
     \begin{tikzpicture}[scale=2.7]

        \def\tra{3} 
        \draw[->, thick, gray] (-1.2,0)--(1.2,0) node[right]{$\Real$};
        \draw[->, thick, gray] (0,-1.2)--(0,1.2) node[above]{$\Imag$};

        \draw[-,thick][black] (0.6,0) -- (0.2,-0.4);
        \draw[->,thick][black] (-0.4,-1) -- (0.2,-0.4);
        \draw[black, fill = black] (0.6,0) circle (0.02);
        \draw (0.6,-0.3) node{$\beta + t_2$};
        \draw[->,very thin][black] (0.6,-0.2) -- (0.6, -0.05);
        \draw[->,thick][black] (0.6,0) -- (0.2,0.4);
        \draw[-,thick][black]  (-0.4,1) -- (0.2,0.4);       

        \draw[-,thick][black] (-0.75, 0) -- (-0.25,-0.5);
        \draw[->,thick][black] (0.25, -1) -- (-0.25, -0.5);
        \draw[black, fill = black] (-0.75,0) circle (0.02);
        \draw (-0.75,-0.275) node{$\alpha + t_1$};
        \draw[->,very thin][black] (-0.75,-0.2) -- (-0.75, -0.05);
        \draw[->,thick][black] (-0.75,0) -- (-0.25,0.5);
        \draw[-,thick][black]  (-0.25,0.5) -- (0.25,1);

        \draw[->,thick][black] (-0.075, -0.675) -- (-0.075,0.2);
        \draw[-,thick][black] (-0.075, 0.2) -- (-0.075,0.675);

        \draw[black, fill = black] (-0.075,0.675) circle (0.02);
        \draw[black, fill = black] (-0.075,-0.675) circle (0.02);
        \draw (-0.075,0.8) node{$u_+$};
        \draw (-0.075,-0.825) node{$u_-$};

        \draw (0.35,0.825) node{$\Gamma^+_{\alpha + t_1}$};
        \draw (-0.4,0.825) node{$\Gamma^-_{\beta + t_2}$};
        \draw (-0.15,0.225) node{$\gamma$};

    \end{tikzpicture} 
    \caption{The figure depicts the contours $\Gamma_{\alpha + t_1}^+, \Gamma_{\beta + t_2}^-$ when they have two intersection points, denoted by $u_-$ and $u_+$. The contour $\gamma$ is the segment from $u_-$ to $u_+$.}
    \label{S12}
\end{figure}
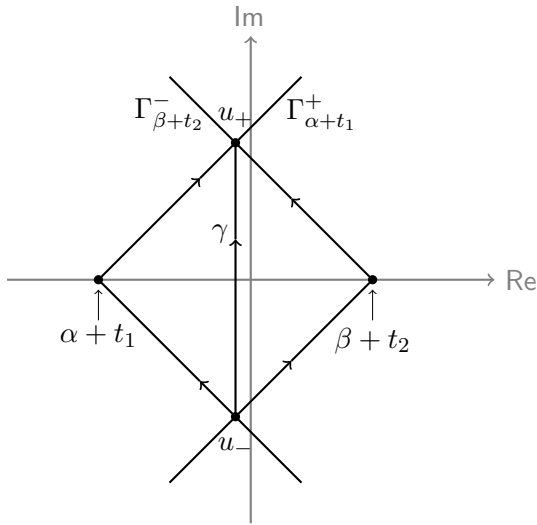

The following proposition summarizes some basic properties of the kernel $K_{a,b,c}$ in Definition \ref{3BPKernelDef}. It was proved as \cite[Lemma 1.4]{dimitrov2024airy} under the assumption that $c^+, c^- \geq 0$; however, the proof remains unchanged for all $c^+, c^- \in \mathbb{R}$, which is how we state the result below.
\begin{proposition}\label{WellDefKer} Assume the same notation as in Definition \ref{DLP}. For each $t_1, t_2,x_1,x_2 \in \mathbb{R}$ we have that the double integral in the definition of $K^3_{a,b,c}$ in (\ref{3BPKer}) is convergent. The value of $K_{a,b,c}(t_1, x_1; t_2, x_2) $ does not depend on the choice of $\alpha$ and $\beta$ as long as $\alpha + t_1 < \underline{a}$ and $\beta + t_2 > \underline{b}$. Moreover, for each fixed $t_1, t_2 \in \mathbb{R}$ we have that $K_{a,b,c}(t_1, \cdot; t_2, \cdot)$ is continuous in $(x_1, x_2) \in \mathbb{R}^2$.
\end{proposition}

The following result, which follows from \cite[Theorems 1.8 and 1.10]{dimitrov2024airy}, introduces the Airy wanderer line ensembles. 
\begin{proposition}\label{S12AWD} Assume the same notation as in Definition \ref{DLP} and fix $(a,b,c) \in \parF$. Then, there exists a unique line ensemble $\mathcal{A}^{a,b,c} = \{\mathcal{A}_i^{a,b,c}\}_{i \geq 1}$ on $\mathbb{R}$ such that all the following hold. For each $m \in \mathbb{N}$ and $s_1, \dots, s_m \in \mathbb{R}$ with $s_1 < s_2 < \cdots < s_m$ we have that the random measure
\begin{equation}\label{T2E1}
M(\omega, A) = \sum_{i \geq 1} \sum_{j = 1}^m {\bf 1}\left\{ (s_j, \mathcal{A}^{a,b,c}_i(\omega, s_j)) \in A \right\}
\end{equation}
is a determinantal point process on $\mathbb{R}^2$, with correlation kernel $K_{a,b,c}$ as in (\ref{3BPKer}), and reference measure $\mu_{\mathsf{S}} \times \mathrm{Leb}$. Here, $\mu_{\mathsf{S}}$ is the counting measure on $\mathsf{S} = \{s_1, \dots, s_m\}$ and $\mathrm{Leb}$ is the Lebesgue measure on $\mathbb{R}$. In addition, if we define the line ensemble $\mathcal{L}^{a,b,c}$ via
\begin{equation}\label{S1FDE}
\left(\sqrt{2} \cdot \mathcal{L}_i^{a, b, c}(t) + t^2: i \geq 1, t \in \mathbb{R} \right)  = \left(\mathcal{A}^{a,b,c}_i(t): i \geq 1 , t \in \mathbb{R} \right),
\end{equation}
then $\mathcal{L}^{a,b,c}$ satisfies the Brownian Gibbs property, see Definition \ref{Def.BGP}.
\end{proposition}
\begin{remark}\label{AWDRem1} We mention that \cite[Theorems 1.8 and 1.10]{dimitrov2024airy} only define $\mathcal{A}^{a,b,c}$ when $c^+ \geq 0$. For general $c^+ \in \mathbb{R}$, one can construct such an ensemble by starting with $\mathcal{A}^{a,b,0}$ (i.e. an ensemble with the same $a_i^{\pm}, b_i^{\pm}$ parameters but with $c^+ = 0$), and then set $\mathcal{A}_i^{a,b,c}(t):= \mathcal{A}_i^{a,b,0}(t) + c^+$. The fact that the latter satisfies the conditions of the proposition follows by a simple change of variables -- this change of variables is explained in the proof of Proposition \ref{S13P1}(a) below.
\end{remark}
\begin{remark}\label{Rem.AWLE2} When all parameters $a_i^{\pm}, b_i^{\pm}$ and $c^{\pm}$ are equal to zero, the ensemble $\mathcal{A}^{0,0,0}$ is just the usual Airy line ensemble from \cite{CorHamA}, and $\mathcal{L}^{0,0,0}$ is the parabolic Airy line ensemble.
\end{remark}

%
%
\subsection{Some basic properties}\label{Section1.3} In this section we summarize some simple properties of the Airy wanderer line ensembles. Some of the statements in this section were previously established in \cite{dimitrov2024airy} for special cases. The results in this section are proved in Section \ref{Section2}.

\begin{proposition}\label{S13P1} Assume the same notation as in Definition \ref{DLP}, fix $(a,b,c) \in \parF$ and let $\mathcal{A}^{a,b,c}$ be as in Proposition \ref{S12AWD}. Then, the following statements all hold.
\begin{enumerate}
    \item[(a)] (Translation) Suppose that $v \in \mathbb{R}$. Then, the line ensemble 
    $$\bar{\mathcal{A}} :=\left(\mathcal{A}^{a,b,c}_i(t) + v : i \geq 1, t \in \mathbb{R}\right).$$ 
    has the same law as $\mathcal{A}^{\bar{a},\bar{b},\bar{c}}$, where $\bar{a}^{\pm}_i = a_i^{\pm}$, $\bar{b}^{\pm}_i = b^{\pm}_i$ for $i \geq 1$ and $\bar{c}^+ = c^+ + v$, $c^- = 0$.
    \item[(b)] (Reflection) The line ensemble $\cev{\mathcal{A}} :=\left(\mathcal{A}^{a,b,c}_i(-t) : i \geq 1, t \in \mathbb{R}\right)$ has the same law as $\mathcal{A}^{b,a,c}$, i.e. the Airy wanderer line ensemble whose ``$a$'' and ``$b$'' parameters have been swapped, but with the same parameters $c^{\pm}$. 
    \item[(c)] There exist $(\tilde{a}, \tilde{b}, \tilde{c}) \in \parP$ and $\Delta \in \mathbb{R}$, such that we have the following equality in law
    $$\left( \mathcal{A}^{\tilde{a},\tilde{b},\tilde{c}}_i(t-\Delta) + c^+: i \geq 1, t \in \mathbb{R} \right) \overset{d}{=}  \left( \mathcal{A}^{a,b,c}_i(t): i \geq 1, t \in \mathbb{R} \right).$$
\end{enumerate}
\end{proposition}
\begin{remark}\label{S13Rem1}
Let us briefly explain the content of Proposition \ref{S13P1}. The first two parts show that the family of Airy wanderer line ensembles is closed under vertical translations, as well as horizontal reflections across the origin. The third part of the proposition states that we can realize any $\mathcal{A}^{a,b,c}$ from Proposition \ref{S12AWD} as an appropriately translated Airy wanderer line ensemble with parameters in $\parP$. In order to simplify our exposition later in the paper, we will state some subsequent results for Airy wanderer ensembles under the additional assumption that their parameters lie in $\parP$. However, in view of the third part of Proposition \ref{S13P1}, one can transfer properties of such ensembles to all Airy wanderer line ensembles by merely translating them. We mention that the parameters $(\tilde{a}, \tilde{b}, \tilde{c})$, and $\Delta$ depend in a non-trivial way on $(a,b,c)$, and we forego stating it until the proof of the proposition in Section \ref{Section2.1}. 
\end{remark}

The next result shows that the laws of the Airy wanderer line ensembles are continuous in their parameters.
\begin{proposition}\label{S13P2} Assume the same notation as in Definition \ref{DLP} and Proposition \ref{S12AWD}. Fix a sequence $ (a_n, b_n, c_n) \in \parP$ such that $(a_n, b_n, c_n) \rightarrow (a,b,c) \in \parP$ in the sense that 
\begin{equation}\label{S13E1}
\lim_{n \rightarrow \infty} \sum_{i \geq 1} \left|a^+_i - a^{+}_{n,i}\right| = 0, \hspace{2mm} \lim_{n \rightarrow \infty} \sum_{i \geq 1}\left|b^+_i - b^{+}_{n,i}\right| = 0.
\end{equation}
We recall that $a_i^{-} = b_i^{-} = a_{n,i}^- = b_{n,i}^- = c^+ = c_n^+ = 0$ here. Then, we have $\mathcal{A}^{a_n,b_n,c_n} \Rightarrow \mathcal{A}^{a,b,c}$. 
\end{proposition}

%
%
\subsection{Main results}\label{Section1.4} In this section, we present the main results of the paper.  

Our first result shows that there are {\em multiple} ways (indexed by $A, B \in \mathbb{Z}_{\geq 0}$) to couple two Airy wanderer line ensembles with different sets of parameters $(a,b,c)$ and $(\tilde{a}, \tilde{b}, \tilde{c})$, that ensure that their appropriately reindexed curves are stochastically ordered. 
\begin{theorem}\label{Thm.MonCoupling} Assume the same notation as in Definition \ref{DLP} and Proposition \ref{S12AWD}. Fix $A, B \in \mathbb{Z}_{\geq 0}$ and two sets of parameters $(a,b,c), (\tilde{a}, \tilde{b}, \tilde{c}) \in \parP$, such that $a_{i + A}^+ \leq \tilde{a}_i^+$ and $b_{i+B}^+ \leq \tilde{b}_i^+$ for $i \geq 1$. Then, we can couple $\mathcal{A}^{a,b,c}$ and $\mathcal{A}^{\tilde{a},\tilde{b},\tilde{c}}$ on the same probability space, so that almost surely
\begin{equation}\label{Eq.MonCoupling1}
\mathcal{A}^{a,b,c}_{k + \max(A,B)}(t) \leq \mathcal{A}_k^{\tilde{a},\tilde{b},\tilde{c}}(t) \mbox{ for all } k \geq 1, t \in \mathbb{R}.
\end{equation}
\end{theorem}
\begin{remark}\label{Rem.MonCoupling} In the simplest case when $A = B = 0$, Theorem \ref{Thm.MonCoupling} shows that the laws of $\mathcal{A}^{a,b,c}$ are monotone in the parameters $(a,b,c)$.
\end{remark}

Let us illustrate the content of Theorem \ref{Thm.MonCoupling} through a simple example.\\
{\em \raggedleft Example.} Suppose that the parameters $(a,b,c)$ satisfy $a_1^+ = 2$, $a_2^+ = 1$, $b_1^+ = 1$, and all other parameters are equal to zero. In addition, suppose that the parameters $(\tilde{a},\tilde{b},\tilde{c})$ satisfy $\tilde{a}^+_1 = 1$, and all other parameters are equal to zero. Since $a_i^+ \geq \tilde{a}_i^+$ and $b_i^+ \geq \tilde{b}_i^+$ for all $i \geq 1$, we have by Theorem \ref{Thm.MonCoupling} with $A = B = 0$, that we can couple $\mathcal{A}^{a,b,c}$ and $\mathcal{A}^{\tilde{a},\tilde{b},\tilde{c}}$, so that almost surely
$$\mathcal{A}^{a,b,c}_i(t) \geq \mathcal{A}^{\tilde{a},\tilde{b},\tilde{c}}_i(t) \mbox{ for all } i \geq 1, t \in \mathbb{R}.$$
On the other hand, since $a_{i+1}^+ \leq \tilde{a}_i^+$ and $b_{i+1}^+ \leq \tilde{b}_i^+$ for all $i \geq 1$, we have by Theorem \ref{Thm.MonCoupling} with $A = B = 1$, that we can couple $\mathcal{A}^{a,b,c}$ and $\mathcal{A}^{\tilde{a},\tilde{b},\tilde{c}}$, so that almost surely
$$\mathcal{A}^{a,b,c}_{i+1}(t) \leq \mathcal{A}^{\tilde{a},\tilde{b},\tilde{c}}_i(t) \mbox{ for all } i \geq 1, t \in \mathbb{R}.$$
We mention that the above two couplings are {\em different}. It would be interesting to see if one can construct a coupling where both inequalities simultaneously hold, although our methods presently do not allow us to do it.\\

The next result describes the global behavior of the points $\{\mathcal{A}_k^{a,b,c}(t)\}_{k \geq 1}$ as $t \rightarrow \pm \infty$. 
\begin{theorem}\label{Thm.Slopes} Assume the same notation as in Definition \ref{DLP} and Proposition \ref{S12AWD}. Fix $(a,b,c) \in \parP$ and any sequence $t_N > 0$ with $t_N \uparrow \infty$. Then, the following statements all hold.
\begin{enumerate}
\item[(a)] If $k \in \mathbb{N}$, $k \leq J_a^+$, then $t_N^{-1} \cdot (\mathcal{A}^{a,b,c}_k(t_N) - t_N^2) \Rightarrow -2/a_k^+$.
\item[(b)] If $k \in \mathbb{N}$, $k \leq J_b^+$, then $t_N^{-1} \cdot (\mathcal{A}^{a,b,c}_k(-t_N) - t_N^2) \Rightarrow -2/b_k^+$.
\item[(c)] If $J_a^+ < k <\infty$, then $\{\mathcal{A}^{a,b,c}_k(t_N) - t_N^2\}_{N \geq 1}$ is tight.
\item[(d)] If $J_b^+ < k <\infty$, then $\{\mathcal{A}^{a,b,c}_k(-t_N) - t_N^2\}_{N \geq 1}$ is tight.
\end{enumerate}
\end{theorem}
\begin{remark}\label{Rem.Slopes} In plain words, Theorem \ref{Thm.Slopes}(a) shows that if $a_k^+ > 0$, then as $t \rightarrow \infty$ the points $\mathcal{A}_k^{a,b,c}(t)$ approximately follow the parabola $t^2 - 2t/a_k^+$. On the other hand, Theorem \ref{Thm.Slopes}(c) shows that if $a_k^+ = 0$, then as $t \rightarrow \infty$ the points $\mathcal{A}_k^{a,b,c}(t)$ are of unit order. Parts (b) and (d) follow from parts (a) and (c) upon reflecting the ensembles across the origin, see Proposition \ref{S13P1}(b).
\end{remark}
\begin{remark}
If $\mathcal{L}^{a,b,c}$ is the parabolic Airy wanderer line ensemble from (\ref{S1FDE}), then Theorem \ref{Thm.MonCoupling}(a,b) shows that the curves $\mathcal{L}^{a,b,c}_k$ have slopes $-\sqrt{2}/a_k^+$ near $\infty$ for $k = 1, \dots, J_a^+$, and slopes $\sqrt{2}/b_k^+$ near $-\infty$ for $k = 1, \dots, J_b^+$. This asymptotic behavior for $\mathcal{L}^{a,b,c}$ was predicted in \cite{dimitrov2024airy}.
\end{remark}
\begin{remark}\label{Rem.Flat} Let us give a simple reason one might expect Theorem \ref{Thm.Slopes}(c,d). Suppose $J_a^+ = J_b^+ = K < \infty$, i.e. exactly $K$ parameters $a_i^+, b_i^+$ are non-zero. Then, by Theorem \ref{Thm.MonCoupling} we can create two couplings of $\mathcal{A}^{a,b,c}$ and $\mathcal{A}^{0,0,0}$ (the usual Airy line ensemble), so that for the first coupling
$$ \mathcal{A}^{a,b,c}_{K+i}(t) \leq \mathcal{A}^{0,0,0}_i(t) \mbox{ for } i \geq 1, t \in \mathbb{R},$$
while for the second coupling
$$ \mathcal{A}^{a,b,c}_{K+i}(t) \geq \mathcal{A}^{0,0,0}_{K+i}(t) \mbox{ for } i \geq 1, t \in \mathbb{R}.$$
Since $\mathcal{A}^{0,0,0}$ is stationary, the latter allows us to stochastically ``squeeze'' $\mathcal{A}^{a,b,c}_{K+i}(t)$ for all $t$ (in particular for large or small ones) between $\mathcal{A}^{0,0,0}_{K+i}(0)$ and $\mathcal{A}^{0,0,0}_{i}(0)$, ensuring tightness.
\end{remark}

We now turn to our last main result, which shows that the ensembles $\mathcal{L}^{a,b,c}$ from Proposition \ref{S12AWD} are extreme points in the space of line ensembles that satisfy the Brownian Gibbs property. Before we state it, we introduce a bit of notation, in particular explaining the content of the last sentence. We start with a simple formulation of the Brownian Gibbs property, introduced in \cite{CorHamA}. We refer the reader to Definition \ref{Def.BGP} for a precise statement, and to \cite[Section 2]{DEA21} and \cite[Section 2]{DimMat} for comprehensive treatments.

Suppose that $\mathcal{L}$ is a $\mathbb{N}$-indexed line ensemble on $\mathbb{R}$ as in Definition \ref{Def.LE}. We say that $\mathcal{L}$ satisfies the {\em Brownian Gibbs property}, if the ensemble is {\em non-intersecting}, i.e.
$$\mathcal{L}_i(t) > \mathcal{L}_{i+1}(t) \mbox{ for all } i \geq 1, \mbox{ and } t \in \mathbb{R},$$
and for each $m \in \mathbb{N}$ and $b > a$, the law of $\{\mathcal{L}_{i}\}_{i = 1}^{m}$ on the interval $[a,b]$, conditioned on $\{\mathcal{L}_{i}(a)\}_{i = 1}^{m}$, $\{\mathcal{L}_{i}(b)\}_{i = 1}^{m}$ and $\{\mathcal{L}_{m+1}(t): t \in [a,b]\}$, is that of $m$ independent Brownian bridges $\{{B}_{i} \}_{i=1}^{m}$ from ${B}_i(a) = \mathcal{L}_i(a)$ to ${B}_i(b) = \mathcal{L}_i(b)$ that are conditioned to not intersect:
$$B_1(t) > B_2(t) > \cdots > B_{m}(t) > \mathcal{L}_{m+1}(t)  \mbox{ for } t\in [a,b].$$

Let $\mathcal{P}$ be the space of probability measures on $C(\mathbb{N} \times \mathbb{R})$ with the weak topology. Sitting inside this space is the (closed and convex) space $\mathcal{G}$ of laws, which satisfy the Brownian Gibbs property. An {\em extreme point} in $\mathcal{G}$ is a measure $\mu$, such that if $\mu = \alpha \cdot \mu_1 + (1-\alpha) \cdot \mu_2$ for some $\mu_1, \mu_2 \in \mathcal{G}$ and $\alpha \in (0,1)$, then $\mu_1 = \mu_2 = \mu$. In other words, $\mu$ is extreme if it cannot be decomposed non-trivially into two measures in $\mathcal{G}$. Denote the set of extreme points of $\mathcal{G}$ by $\mathrm{Ext}(\mathcal{G})$. 

With the above notation in place, we are ready to state our final main result.
\begin{theorem}\label{Thm.Extreme} Assume the same notation as in Definition \ref{DLP}, Proposition \ref{S12AWD} and the preceding paragraphs. For any $(a,b,c) \in \parP$, we have that $\text{Law}\left(\mathcal{L}^{a,b,c}\right) \in \mathrm{Ext}(\mathcal{G})$. 
\end{theorem}

A major problem in the study of any family of Gibbs measures $\mathcal{M}$ is to understand the set of extreme points $\mathrm{Ext}(\mathcal{M})$. Part of the reason is that under general conditions, one can express any Gibbs measure as a {\em Choquet integral} over $\mathrm{Ext}(\mathcal{M})$, which sets up a one-to-one correspondence between $\mathcal{M}$ and probability measures on $\mathrm{Ext}(\mathcal{M})$. In other words, the elements of $\mathrm{Ext}(\mathcal{M})$ are the {\em simplest} Gibbs measures, and any measure in $\mathcal{M}$ is a convex combination of these. When $\mathcal{M} = \mathcal{G}$, Theorem \ref{Thm.Extreme} provides some insight about the nature of the set $\mathrm{Ext}(\mathcal{G})$, as it identifies a large (infinite parameter) family of elements in this set. 

There are still many questions we would like to address about the set $\mathrm{Ext}(\mathcal{G})$. It is certainly not the case that $\{\text{Law}\left(\mathcal{L}^{a,b,c}\right): (a,b,c) \in \parP\}$ exhausts $\mathrm{Ext}(\mathcal{G})$. Indeed, the Brownian Gibbs property is preserved under translations, affine shifts and Brownian scalings ($f(t) \mapsto \lambda^{-1} f(\lambda^2t)$ for $\lambda > 0$), see e.g. \cite[Lemma 2.1]{DS25} with $\vec{A} = 0$. Consequently, by Proposition \ref{S13P1}(c), at the very least $\{\text{Law}\left(\mathcal{L}^{a,b,c}\right): (a,b,c) \in \parF\}$ and all their affine shifts and Brownian scalings are contained in $\mathrm{Ext}(\mathcal{G})$. Whether the last containment is in fact an equality would be interesting to show, although we suspect the answer to be negative. Beyond this, we expect $\mathrm{Ext}(\mathcal{G})$ to be non-compact, which prevents the application of the usual Choquet theory to this set. Nevertheless, we expect some version of the correspondence between probability measures on $\mathrm{Ext}(\mathcal{G})$ and $\mathcal{G}$ to hold. We hope to address the various open-ended questions in this paragraph in the future.

%
%
\subsection{Key ideas and paper outline}\label{Section1.5} The rest of the paper is split into five sections. In Section \ref{Section2} we establish the basic properties from Section \ref{Section1.3}, with the proofs of Propositions \ref{S13P1} and \ref{S13P2} occupying Sections \ref{Section2.1} and \ref{Section2.2}, respectively. 

In Section \ref{Section3}, we establish the monotone coupling Theorem \ref{Thm.MonCoupling}. As mentioned in Section \ref{Section1.1}, in \cite{dimitrov2024airy, dimitrov2024tightness} we showed that the Airy wanderer line ensembles are weak limits of Schur processes --- the precise statement is recalled as Proposition \ref{Prop.SchurToAiry} in Section \ref{Section3.5}. The key observation here is that there are multiple ways to couple two Schur processes monotonically in their parameters --- see Proposition \ref{MonCoup2} in Section \ref{Section3.4}. These monotone couplings persist in the limit to the Airy line ensembles, which is how Theorem \ref{Thm.MonCoupling} is proved in Section \ref{Section3.6}. 

The way we prove our monotone coupling in Proposition \ref{MonCoup2} is by utilizing an exact sampling algorithm for Schur processes introduced by Borodin in \cite{B11}, which is recalled in Section \ref{Section3.2}. This algorithm is based on certain push-block Markovian dynamics on sequences of partitions that involve sequential update rules. The key idea behind the proof of Proposition \ref{MonCoup2} is that we can couple these dynamics monotonically in the parameters of the model by feeding the same background noise that drives the dynamics through certain parameter-dependent quantile functions. These quantile functions correspond to truncated and shifted geometric variables, and the core computation is that the action of these functions is monotone in their parameters --- see Lemma \ref{S33Lem1} in Section \ref{Section3.3}.

In Section \ref{Section4}, we introduce two random measures corresponding to the two scalings in parts (a) and (c) of Theorem \ref{Thm.Slopes}. We show that the first and second moments of these measures converge in the large-time limit; see Lemmas \ref{Lem.FirstMoment}, \ref{Lem.SecondMoment} and \ref{Lem.FirstMomentFlat}. The proofs of these lemmas rely on the determinantal point process structure of the Airy wanderer line ensembles described in Proposition \ref{S12AWD}, together with a detailed asymptotic analysis of the correlation kernel $K_{a,b,c}$ from Definition \ref{3BPKernelDef} via the method of steepest descent. 

Lemmas \ref{Lem.FirstMoment}, \ref{Lem.SecondMoment} and \ref{Lem.FirstMomentFlat} are then used in Section \ref{Section5} to establish the asymptotic slopes theorem, Theorem \ref{Thm.Slopes}. A key simplifying strategy we use is that these lemmas are proved under the assumption that the parameters of the model are all distinct. This assumption leads to a simpler asymptotic analysis, since it ensures that all the poles of the integrands in $K_{a,b,c}$ are simple, but it restricts subsequent arguments to a subset of the parameter space. Nevertheless, this subset is sufficiently large that we can ``fill in the gaps'' using the monotone coupling Theorem \ref{Thm.MonCoupling}, and thereby prove Theorem \ref{Thm.Slopes} for all parameters.

In the last section, Section \ref{Section6}, we establish the extremality of the Airy wanderer line ensembles from Theorem \ref{Thm.Extreme}. The proof is split into two parts, occupying Sections \ref{Section6.2} and \ref{Section6.3}. Below we outline some of the key ideas behind the proof. Additional details can be found in the outline at the end of Section \ref{Section6.2}, which contains a high-level explanation of the second part of the proof.\\  

The setup of the theorem provides a parabolic Airy wanderer line ensemble $\mathcal{L}^{a,b,c}$ as in (\ref{S1FDE}), whose law we denote by $\mu$, a parameter $\alpha \in (0,1)$, and two Brownian Gibbsian ensembles $\mathcal{L}^1$ and $\mathcal{L}^2$ with respective laws $\mu_1$ and $\mu_2$, satisfying 
$$\mu = \alpha \mu_1 + (1-\alpha) \mu_2.$$ 
Our goal is to show $\mu = \mu_1 = \mu_2$. 

The main idea is to construct a sequence of appropriately rescaled Airy wanderer ensembles, denoted $\hat{\mathcal{L}}^{\delta, N}$, which are typically below $\mathcal{L}^{a,b,c}$, and hence below both $\mathcal{L}^1$ and $\mathcal{L}^2$. These ensembles are related to $\mathcal{L}^{a,b,c}$ as follows:
\begin{itemize}
    \item One defines $(\hat{a},\hat{b},\hat{c})$ by starting with $(a,b,c) \in \parP$, setting $a_i^+ = b_j^+ = 0$ for $i \geq L+1, j \geq R+1$, and then slightly decreasing the remaining non-zero parameters using a small parameter $\delta >0$. 
    \item The curves $\hat{\mathcal{L}}^{\delta, N}_i$ are then defined by 
    $$\hat{\mathcal{L}}^{\delta, N}_i(t) = \lambda_N^{-1}\mathcal{L}^{\hat{a}, \hat{b}, \hat{c}}_i(\lambda_N^2 t) - r_N$$ 
    for suitable sequences $\lambda_N \rightarrow 1^+$ and $r_N \rightarrow 0^+$.
\end{itemize}

Using Theorem \ref{Thm.Slopes}(a,b), we compare the top few curves of $\mathcal{L}^{j}$ at times $\pm N$ with those of $\hat{\mathcal{L}}^{\delta, N}$, and conclude that the former are typically higher. Moreover, by carefully choosing $\lambda_N$, we alter the curvature of $\hat{\mathcal{L}}^{\delta, N}$ sufficiently to show that the endpoints $\{\mathcal{L}_i^{j}(\pm N)\}_{i = 1}^{N-1}$ are likely higher than $\{\hat{\mathcal{L}}^{\delta, N}_i(\pm N)\}_{i = 1}^{N-1}$. Here we also exploit Theorem \ref{Thm.Slopes}(c,d).

Next, using the monotone coupling from Theorem \ref{Thm.MonCoupling} together with rigidity estimates for the parabolic Airy line ensemble, we show that the $N$-th curve $\mathcal{L}^j_N$ is likely higher than a certain inverted parabola $g_N$. Using the same theorem and adjusting the $r_N$, we can ensure that the $N$-th curve $\hat{\mathcal{L}}^{\delta, N}_N$ is below $g_N$, and hence $\mathcal{L}^j_N$.

Combining these steps and invoking the Brownian Gibbs property, we conclude that $\hat{\mathcal{L}}^{\delta, N}$ can be coupled with $\mathcal{L}^j_N$ so that the latter is likely higher than the former. Taking the limits $N \rightarrow \infty$, $\delta \rightarrow 0^+$, and $L,R \rightarrow \infty$, we deduce that $\mathcal{L}^{a,b,c}$ is stochastically dominated by both $\mathcal{L}^1$ and $\mathcal{L}^2$. As the law $\mu$ of $\mathcal{L}^{a,b,c}$ is a convex combination of $\mu_1$ and $\mu_2$, this is only possible if $\mu = \mu_1 = \mu_2$.

%
%
\subsection*{Acknowledgments} The author would like to thank Amol Aggarwal, Alexei Borodin, and Ivan Corwin for many fruitful discussions. This work was partially supported by Simons Foundation International through Simons Award TSM-00014004.

%
%
\section{Basic properties of Airy wanderer line ensembles}\label{Section2} The goal of this section is to establish the two propositions in Section \ref{Section1.3}. Throughout this section we freely use the terminology and notation for determinantal point processes from \cite[Section 2]{dimitrov2024airy}.

%
%
\subsection{Proof of Proposition \ref{S13P1}}\label{Section2.1} We prove each of the parts separately.

%
%
\subsubsection{Proof of part (a)}\label{Section2.1.1} In \cite[Lemma 3.1]{DimMat} it was shown that the law of a line ensemble is uniquely determined by its finite-dimensional distributions. Consequently, it suffices to show that $\bar{\mathcal{A}}$ has the same finite-dimensional distribution as $\mathcal{A}^{\bar{a},\bar{b},\bar{c}}$. From \cite[Corollary 2.20]{dimitrov2024airy} and \cite[Proposition 2.13(3)]{dimitrov2024airy}, we see that it suffices to show that for any $s_1, \dots, s_m \in \mathbb{R}$ with $s_1 < s_2 < \cdots < s_m$, we have that the random measure
\begin{equation}\label{S211E1}
\bar{M}(A) = \sum_{i \geq 1} \sum_{j = 1}^m {\bf 1}\left\{ (s_j, \bar{\mathcal{A}}_i(s_j)) \in A \right\}
\end{equation}
is a determinantal point process on $\mathbb{R}^2$, with correlation kernel $K_{\bar{a},\bar{b},\bar{c}}$ as in (\ref{3BPKer}), and reference measure $\mu_{\mathsf{S}} \times \mathrm{Leb}$ as in Proposition \ref{S12AWD}. We next observe that $\bar{M} = M \phi^{-1}$, where 
\begin{equation}\label{S211E2}
M(A) = \sum_{i \geq 1} \sum_{j = 1}^m {\bf 1}\left\{ (s_j , \mathcal{A}^{a,b,c}_i(s_j)) \in A \right\}, \mbox{ and } \phi(s,x) = (s, x + v).
\end{equation}
From Proposition \ref{S12AWD}, we know that $M$ is a determinantal point process with correlation kernel $K_{a,b,c}$ and reference measure $\mu_{\mathsf{S}} \times \mathrm{Leb}$. Using the change of variables, cf. \cite[Proposition 2.13(5)]{dimitrov2024airy}, we conclude that $\bar{M}$ is a determinantal point process with correlation kernel
\begin{equation}\label{S211E3}
\bar{K}(t_1, x_1; t_2,x_2) = K_{a,b,c}(t_1, x_1 - v; t_2 , x_2 - v),
\end{equation}
and reference measure $(\mu_{\mathsf{S}} \times \mathrm{Leb})\phi^{-1} = \mu_{\mathsf{S} } \times \mathrm{Leb}$. We finally check directly from the kernel definition in Definition \ref{3BPKernelDef} that
\begin{equation}\label{S211E4}
 K_{a,b,c}(t_1, x_1 - v; t_2, x_2 - v) = e^{vt_2 - v t_1} \cdot K_{\bar{a},\bar{b},\bar{c}}(t_1, x_1; t_2, x_2).
\end{equation}
Equations (\ref{S211E3}) and (\ref{S211E4}) imply that $\bar{M}$ is determinantal with correlation kernel as the right side of (\ref{S211E4}), but then by a gauge transformation it is also determinantal with correlation kernel $K_{\bar{a},\bar{b},\bar{c}}$, cf. \cite[Proposition 2.13(4)]{dimitrov2024airy}. This suffices for the proof.

%
%
\subsubsection{Proof of part (b)}\label{Section2.1.2} Arguing as in the proof of part (a) above, we see that it suffices to show that for any $s_1, \dots, s_m \in \mathbb{R}$ with $s_1 < s_2 < \cdots < s_m$, we have that the random measure
\begin{equation}\label{S212E1}
\cev{M}(A) = \sum_{i \geq 1} \sum_{j = 1}^m {\bf 1}\left\{ (s_j, \mathcal{A}^{a,b,c}_i(-s_j)) \in A \right\}
\end{equation}
is a determinantal point process on $\mathbb{R}^2$, with correlation kernel $K_{b,a,c}$ as in (\ref{3BPKer}), and reference measure $\mu_{\mathsf{S}} \times \mathrm{Leb}$ as in Proposition \ref{S12AWD}. We next observe that $\cev{M} = M \phi^{-1}$, where 
\begin{equation}\label{S212E2}
M(A) = \sum_{i \geq 1} \sum_{j = 1}^m {\bf 1}\left\{ (-s_j , \mathcal{A}^{a,b,c}_i(-s_j)) \in A \right\}, \mbox{ and } \phi(s,x) = (-s, x).
\end{equation}
From Proposition \ref{S12AWD}, a change of variables, cf. \cite[Proposition 2.13(5)]{dimitrov2024airy}, and a kernel transposition, cf. \cite[Proposition 2.13(4)]{dimitrov2024airy}, we conclude that $\cev{M}$ is a determinantal point process on $\mathbb{R}^2$, with correlation kernel
\begin{equation}\label{S212E3}
\cev{K}(t_1, x_1; t_2,x_2) = K_{a,b,c}(-t_2 , x_2;-t_1, x_1),
\end{equation}
and reference measure $(\mu_{-\mathsf{S}} \times \mathrm{Leb})\phi^{-1} = \mu_{\mathsf{S}} \times \mathrm{Leb}$, where $- \mathsf{S} = \{-s_m, -s_{m-1}, \dots, -s_1\}$. To finish the proof, it suffices to check that 
\begin{equation}\label{S212E4}
 K_{a,b,c}(-t_2 , x_2;-t_1, x_1) = K_{b,a,c}(t_1, x_1; t_2, x_2).
\end{equation}

In the remainder we briefly explain why (\ref{S212E4}) holds. Fix $A < 0$ sufficiently small, and $B > 0$ sufficiently large, so that $A + |s| < 0$ and $B - |s| > 0$ for all $s \in \mathsf{S}$. On the left of (\ref{S212E4}) we use Definition \ref{3BPKernelDef} with $\alpha = A$, $\beta = B$, while on the right side we use Definition \ref{3BPKernelDef} with $\alpha = -B$, $\beta = -A$. With this choice, one directly verifies that 
$$K^i_{a,b,c}(-t_2 , x_2;-t_1, x_1) = K^i_{b,a,c}(t_1, x_1; t_2, x_2) \mbox{ for } i = 1,2,3.$$
We mention that when $i = 2$, the equality is directly verified using the third line in (\ref{3BPKer}). When $i = 1$ one needs to change variables $w \rightarrow -w$ in the left side to match it with the right. When $i = 3$, one needs to change variables $w \rightarrow -w$ and $z \rightarrow -z$ on the left side, and use the identity
$$\Phi_{a,b,c}(v+t) = \frac{1}{\Phi_{b,a,c}(-v-t)},$$
which is a direct consequence of the definition of $\Phi_{a,b,c}$ in (\ref{DefPhi}).

%
%
\subsubsection{Proof of part (c)}\label{Section2.1.3} As the proof of this statement essentially already appeared in the proof of \cite[Theorem 1.8]{dimitrov2024airy}, we will be brief. Let $\{\hat{a}^+_i\}_{i \geq 1}$ be the decreasing sequence formed by the terms in $\{a_i^+\}_{i \geq 1}$ (counted with multiplicities) as well as the terms $1/a_1^-, \dots, 1/ a_{J_a^-}^-$. Similarly, we let $\{\hat{b}^+_i\}_{i \geq 1}$ be the decreasing sequence formed by the terms in $\{b_i^+\}_{i \geq 1}$ (counted with multiplicities) as well as the terms $1/b_1^-, \dots, 1/ b_{J_b^-}^-$. 

We next fix $\Delta \in \mathbb{R}$ as follows. If $J_a^- = J_b^-$, we set $\Delta = 0$. If $J_a^- > J_b^-$, we let $\Delta > 0$ be close enough to zero so that $1 - \Delta a_1^+ > 0$. If $J_a^- < J_b^-$, we let $\Delta < 0$ be close enough to zero so that $1 + \Delta \hat{b}_1^+ >0$. We now define the parameters $\{\tilde{a}^{+}_i\}_{i \geq 1}$, $\{\tilde{b}^{+}_i\}_{i \geq 1}$, using $\{\hat{a}_i^+\}_{i \geq 1}$, $\{\hat{b}_i^+\}_{i \geq 1}$, and $\Delta$.  If $J_a^- = J_b^-$, we set $\tilde{a}_i^+ = \hat{a}_i^+$ and $\tilde{b}_i^+ = \hat{b}_i^+$ for $i \geq 1$. If $J_a^- > J_b^-$, we set $h = J_a^- - J_b^-$ and let $\{\tilde{b}_i^+\}_{i \geq 1}$ be the decreasing sequence formed by the terms in $\{ \hat{b}_i^+/ (1 + \Delta \hat{b}_i^+) \}_{i \geq 1}$ (counted with multiplicities) and $h$ copies of $\Delta^{-1}$, while $\tilde{a}^+_i = \hat{a}_i^+/ (1 - \Delta \hat{a}_i^+)$ for $i \geq 1$. If $J_b^- > J_a^-$, we set $h = J_b^- - J_a^-$ and let $\{\tilde{a}_i^+\}_{i \geq 1}$ be the decreasing sequence formed by the terms in $\{ \hat{a}_i^+/ (1 - \Delta \hat{a}_i^+) \}_{i \geq 1}$ (counted with multiplicities) and $h$ copies of $-\Delta^{-1}$, while $\tilde{b}^+_i = \hat{b}_i^+/ (1 + \Delta \hat{b}_i^+)$ for $i \geq 1$. 

Suppose that $\mathcal{A}^{\tilde{a}, \tilde{b}, \tilde{c}}$ is as in Proposition \ref{S12AWD} for the above choice of parameters and with $c^+ = 0$, in particular $a_i^- = b_i^- = 0$ for $i \geq 1$, i.e. $(\tilde{a}, \tilde{b}, \tilde{c}) \in \parP$. From Step 2 in the proof of \cite[Theorem 1.8]{dimitrov2024airy}, we have the following equality in law
$$\left( \mathcal{A}^{\tilde{a},\tilde{b},\tilde{c}}_i(t-\Delta) : i \geq 1, t \in \mathbb{R} \right) =  \left( \mathcal{A}^{a,b,0}_i(t): i \geq 1, t \in \mathbb{R} \right),$$
where $\mathcal{A}^{a,b,0}$ has the same parameters $(a,b,c)$ in the statement of the proposition but with $c^+ = 0$. The third statement of the proposition follows from the last equality and part (a) of the proposition.

%
%
\subsection{Proof of Proposition \ref{S13P2}}\label{Section2.2} Let us define $\mathcal{L}^n$ via 
$$\left(\sqrt{2} \cdot \mathcal{L}_i^{n}(t) + t^2: i \geq 1, t \in \mathbb{R} \right)  = \left(\mathcal{A}^{a_n,b_n,c_n}_i(t): i \geq 1 , t \in \mathbb{R} \right),$$
We seek to show that $\mathcal{L}^n$ satisfy the conditions of \cite[Lemma 8.1]{dimitrov2024airy}. If true, then we would conclude that $\mathcal{L}^n \Rightarrow \mathcal{L}^{a,b,c}$, where the latter is as in (\ref{S1FDE}) for the ensemble $\mathcal{A}^{a,b,c}$ as in the statement of the proposition. In particular, by the continuous mapping theorem we would get the result we desire.

The fact that $\mathcal{L}^n$ satisfy the Brownian Gibbs property follows from Proposition \ref{S12AWD}. Consequently, we just need to show that for fixed $s, t \in \mathbb{R}$ and sequences $x_n, y_n \in \mathbb{R}$ such that $\lim_n x_n = x$ and $\lim_n y_n = y$ we have
\begin{equation}\label{S81E1}
\lim_{n \rightarrow \infty} K_{a_n, b_n, c_n}(s, x_n; t, y_n) = K_{a,b,c}(s,x;t,y),
\end{equation}
and also that for each $t \in \mathbb{R}$ we have
\begin{equation}\label{S81E2}
\lim_{ a \rightarrow \infty} \limsup_{n \rightarrow \infty} \mathbb{P}\left(\mathcal{A}^{a_n,b_n,c_n}_1(t) \geq a\right) = 0.
\end{equation}
The proofs of (\ref{S81E1}) and (\ref{S81E2}) we present below are quite similar to that in Step 1 of the proof of \cite[Theorem 1.10]{dimitrov2024airy}, and so we will be brief.\\

We first fix $s,t \in \mathbb{R}$ and proceed to show (\ref{S81E1}). Below, we make use of the following inequalities:
\begin{equation}\label{RatBound}
|1 + z| \leq e^{|z|} \mbox{ for $z \in \mathbb{C}$, and } \frac{1}{|1-z|} \leq 1 + \frac{|z|}{|1-z|} \leq e^{|z|/d} \mbox{ for } z \in \mathbb{C}, |1- z| \geq d > 0.
\end{equation}
Fix $\alpha, \beta \in \mathbb{R}$ so that $s + \alpha > 0$, $t + \beta < 0$, and $a^+_{n,1} (s+ \alpha) < 1/2$, $-b^+_{n,1} (t + \beta) < 1/2$ for all $n \geq 1$. Note that if $z \in \Gamma_{\alpha}^+$, then $|1 - a_{n,i}^+(z+s)| \geq 1/4$, and if $w \in \Gamma_{\beta}^-$, then $|1 - b_{n,i}^+(w +t)| \geq 1/4$ from our choice of $\alpha, \beta$. In addition, from (\ref{S13E1}) we can find $A \in [0, \infty)$, such that $\sum_{i \geq 1} (a^+_{n,i} + b^+_{n,i}) \in [0, A]$ for all $n \geq 1$. Combining the last few observations and the definition of $\Phi_{a,b,c}$ in (\ref{DefPhi}), we conclude that we can find a constant $c_A$, depending on $A$ alone, such that for $z \in \Gamma_{\alpha}^+$, $w \in \Gamma_{\beta}^-$, we have
\begin{equation}\label{S22E1}
\left|\frac{\Phi_{a_n,b_n, c_n}(z + s)}{\Phi_{a_n,b_n, c_n}(w + t)} \right|  = \left|\prod_{i = 1}^{\infty} \frac{(1 + b_{n,i}^+ (z+s)) (1 - a_{n,i}^+ (w+t))}{(1 - a_{n,i}^+ (z+s)) (1 + b_{n,i}^+ (w+t)) } \right| \leq \exp\left(c_A|z+s| + c_A|w+t| \right). 
\end{equation}
We mention that in deriving the latter we used (\ref{RatBound}) with $d = 1/4$. 

In addition, by analyzing the real part of $z^3/3$ and $w^3/3$, we see that we can find a constant $D_1 > 0$ (depending on $\alpha, \beta$), such that for $z \in \Gamma_{\alpha}^+$, $w \in \Gamma_{\beta}^-$, we have
\begin{equation}\label{S22E2}
\begin{split}
\left|e^{z^3/3 - w^3/3} \right| \leq D_1 \cdot \exp \left(-|z|^3/12 - |w|^3/12 \right).
\end{split}
\end{equation}
Since by construction $\Gamma_{\alpha + s}^+$ is well-separated from $\Gamma_{\beta + t}^-$, we also have for some $D_2 > 0$ (depending on $\alpha, \beta,s ,t$)
\begin{equation}\label{S81Q5}
\left|\frac{1}{z + s - w - t} \right| \leq D_2.
\end{equation}
Finally, for a sequence $x_n$, and $y_n$ converging to $x$ and $y$ we have that for all large $n$
\begin{equation}\label{S81Q6}
\left|e^{-x_n z  + y_nw}\right| \leq \exp \left( (|x|+1)|z| + (|y|+1) |w| \right).
\end{equation}

We now observe that for $i = 1,2,3$ we have
\begin{equation}\label{S22E3}
\lim_{n \rightarrow \infty} K^i_{a_n, b_n, c_n}(s, x_n; t, y_n) = K^i_{a,b,c}(s,x;t,y).
\end{equation}
When $i = 1$, the latter is trivial as both sides are equal to zero (as $\gamma = \emptyset$ by construction). When $i = 2$, equation (\ref{S22E3}) holds by the continuity of the exponential functions on the third line of (\ref{3BPKer}) and the fact that $\lim_n x_n = x$, $\lim_n y_n = y$. When $i = 3$, we note that the integrands in the definition of $K^3_{a_n, b_n, c_n}(s, x_n; t, y_n)$ in the fourth line of (\ref{3BPKer}) converge to those of $ K^3_{a,b,c}(s,x;t,y)$ in view of the assumed convergence $(a_n, b_n, c_n) \rightarrow (a,b,c)$ in (\ref{S13E1}). The fact that the integrals themselves converge follows from the dominated convergence theorem, with dominating function given by the product of the right sides of (\ref{S22E1}), (\ref{S22E2}), (\ref{S81Q5}) and (\ref{S81Q6}). Equation (\ref{S22E3}) clearly implies (\ref{S81E1}).\\

In the remainder of this section we establish (\ref{S81E2}). Using that $\Real(z+t) > 0$ for all $z \in \Gamma_{\alpha}^+$ and $\Real(w+t) < 0$ for all $w \in \Gamma_{\beta}^-$, we see that there is a constant $d_1 > 0$ (depending on $\alpha, \beta, t$) such that for $z\in \Gamma_{\alpha}^+$ and $w \in \Gamma_{\beta}^-$ 
\begin{equation}\label{S81Q8}
\begin{split}
\left|e^{-x (z + t)} \right| \leq \begin{cases} e^{-d_1 x} &\mbox{ if } x \geq 0, \\ e^{|L||z + t|} &\mbox{ if } x \in [-L, 0], \end{cases} \hspace{3mm} \left|e^{(w + t)y} \right| \leq \begin{cases} e^{-d_1 y}, &\mbox{ if } y \geq 0, \\ e^{|L||w + t|} &\mbox{ if } y \in [-L, 0]. \end{cases}
\end{split}
\end{equation}
Combining (\ref{S22E1}), (\ref{S22E2}), (\ref{S81Q5}) and (\ref{S81Q8}), we see that we can find a constant $B_L$ (depending on $\alpha$, $\beta$, $t$, $A$, and $L$), such that for all large $n$ and $x, y \geq -L$
\begin{equation}\label{S81Q9}
\begin{split}
&\left|K_{a_n,b_n,c_n}(t,x;,t,y) \right| \leq B_L \cdot \exp\left( -d_1 |x| - d_1 |y| \right).
\end{split}
\end{equation}
In deriving the last identity we used that $K_{a_n,b_n,c_n} = K_{a_n,b_n,c_n}^3$ by our choice of $\alpha, \beta$.

If $V$ is an $n \times n$ matrix with complex entries and columns $v_1, \dots, v_n$, we have by Hadamard's inequality, see \cite[Corollary 33.2.1.1.]{Prasolov}, that 
\begin{equation}\label{Hadamard}
 |\det V| \leq \prod_{i = 1}^n \|v_i\|
\end{equation}
where $\|v\| = (|v_1|^2 + \cdots + |v_n|^2)^{1/2}$ for $v = (v_1, \dots, v_n) \in \mathbb{C}^n$. Consequently, for $x_1, \dots, x_n \geq -L$ 
\begin{equation*}
\left| \det\left[ K_{a_n,b_n,c_n}(t,x_i;t,x_j) \right]_{i,j = 1}^n \right| \leq B_L^n \cdot n^{n/2} \cdot \prod_{i = 1}^n \exp \left( - d_1 |x_i| \right),
\end{equation*}
which implies for any $a \in \mathbb{R}$
\begin{equation*}
1 + \sum_{n \geq 1} \frac{1}{n!} \int_{(a, \infty)^n} \left| \det\left[ K_{a_n,b_n,c_n}(t,x_i;t,x_j) \right]_{i,j = 1}^n \right| \lambda^n(dx) < \infty.
\end{equation*}

From \cite[Proposition 2.9 and Lemma 2.17]{dimitrov2024airy}, we conclude for all $a \in \mathbb{R}$
\begin{equation*}
\mathbb{P}\left(\mathcal{A}^{a_n,b_n,c_n}_1(t) \geq a\right)  = \sum_{n = 1}^{\infty} \frac{(-1)^{n-1}}{n!}  \int_{(a, \infty)^n}\det\left[ K_{a_n,b_n,c_n}(t,x_i;t,x_j) \right]_{i,j = 1}^n\lambda^n(dx). 
\end{equation*}
Combining the estimates after and including (\ref{S81Q9}), we get for all $a \geq 0$
\begin{equation*}
\mathbb{P}\left(\mathcal{A}^{a_n,b_n,c_n}_1(t) \geq a\right) \leq \sum_{n = 1}^{\infty} \frac{n^{n/2} B_0^n }{n!} \cdot \left(\int_{a}^{\infty} e^{-d_1 x} dx\right)^n \leq e^{-d_1a} \cdot \sum_{n = 1}^{\infty} \frac{n^{n/2} B_0^n }{n! \cdot d_1^n},
\end{equation*}
which proves (\ref{S81E2}).\\

%
%
\section{Monotone couplings}\label{Section3} The goal of this section is to establish Theorem \ref{Thm.MonCoupling}. In Section \ref{Section3.1} we introduce Schur processes, and in Section \ref{Section3.2} we describe a sampling algorithm for them based on \cite{B11}. Section \ref{Section3.3} develops several properties of quantile functions, which we then use in Section \ref{Section3.4} to show that certain Schur processes admit monotone couplings in their parameters, see Proposition \ref{MonCoup2}. In Section \ref{Section3.5} we recall from \cite{dimitrov2024tightness} that the Airy wanderer line ensembles can be obtained as weak limits of Schur processes, see Proposition \ref{Prop.SchurToAiry} for the precise statement. Finally, in Section \ref{Section3.6} we prove Theorem \ref{Thm.MonCoupling} by showing that the monotone couplings from Proposition \ref{MonCoup2} persist in the limit to the Airy wanderer line ensembles established in Proposition \ref{Prop.SchurToAiry}.

%
%
\subsection{Schur processes}\label{Section3.1} In this section we define the Schur processes that we work with in the present paper. Our exposition follows \cite[Section 3.1]{dimitrov2024airy}, which in turn goes back to \cite{BR05, OR03}. 

A {\em partition} is a sequence $\lambda = (\lambda_1, \lambda_2,\dots)$ of non-negative integers, called {\em parts}, such that $\lambda_1 \geq \lambda_2 \geq \cdots$ and all but finitely many elements are zero. We denote the set of all partitions by $\mathbb{Y}$. The {\em weight} of a partition $\lambda$ is given by $|\lambda| = \lambda_1 + \lambda_2 + \cdots$ . There is a single partition of weight $0$, which we denote by $\emptyset$. We say that two partitions $\lambda, \mu$ {\em interlace}, denoted $\lambda \succeq \mu$ or $\mu \preceq \lambda$, if $\lambda_1 \geq \mu_1 \geq \lambda_2 \geq \mu_2 \geq \cdots$. For two partitions $\lambda, \mu$, we define the (skew) Schur polynomial in a single variable by
\begin{equation}\label{S31E1}
s_{\lambda/\mu}(x) = {\bf 1}\{\lambda \succeq \mu\} \cdot x^{|\lambda| - |\mu|},
\end{equation}
and in several variables by
\begin{equation}\label{S31E2}
s_{\lambda/\mu}(x_1,\dots, x_n) = \sum_{\lambda^1, \dots, \lambda^{n-1} \in \mathbb{Y}} \prod_{i = 1}^n s_{\lambda^i/\lambda^{i-1}}(x_i),
\end{equation}
where $\lambda^0 = \mu$ and $\lambda^n = \lambda$. If $\mu = \emptyset$ in (\ref{S31E2}), we drop it from the notation and simply write $s_\lambda$ in place of $s_{\lambda/\emptyset}$. With the above notation in place we can define our measures.
\begin{definition}\label{SPD} Fix $M, N \in \mathbb{N}$ and suppose that $X = (x_1,\dots, x_M)$, $Y = (y_1, \dots, y_N)$ are such that $x_i, y_j \geq 0$ and $x_i y_j < 1$ for all $i = 1, \dots, M$, and $j = 1, \dots, N$. With this data we define the measure
\begin{equation}\label{SP}
\mathbb{P}_{X,Y} (\lambda^1, \dots, \lambda^{M}) = \prod_{i = 1}^M \prod_{j = 1}^N (1 - x_i y_j) \cdot \prod_{i = 1}^M s_{\lambda^i/ \lambda^{i-1}}(x_i) \cdot s_{\lambda^M}(y_1, \dots, y_M),
\end{equation}
where $\lambda^0 = \emptyset$ and $\lambda^i \in \mathbb{Y}$. 
We call the measure $\mathbb{P}_{X,Y}$ the {\em Schur process} with parameters $X,Y$. 
\end{definition}
\begin{remark}\label{S31Rem} Note that, in view of (\ref{S31E1}), $\mathbb{P}_{X,Y}$ is supported on sequences $(\lambda^1, \dots, \lambda^M)$ such that
$$\emptyset \preceq \lambda^1 \preceq \lambda^2 \preceq \cdots \preceq \lambda^{M-1} \preceq \lambda^M.$$ 
\end{remark}

%
%
\subsection{Schur dynamics}\label{Section3.2} In this section we describe a sampling algorithm for the Schur processes in Definition \ref{SPD}. This algorithm was introduced in \cite{B11}, and it is based on certain Markov dynamics on the Schur processes that involve sequential update. At a high level, the sampling algorithm allows us to sequentially build $(\lambda(n,1), \dots, \lambda(n,M))$, distributed according to $\mathbb{P}_{X,Y_n}$ for $n = 1, \dots, N$, where $Y_n = (y_1, \dots, y_n)$, by evolving the parts of $\lambda(n,j)$. This evolution of the parts is based on certain truncated geometric random variables, which we introduce next. 

\begin{definition}\label{GeomDistr} Fix $q \in [0,1)$, $a \in \mathbb{Z}$, and $b \in \mathbb{Z} \cup \{\infty\}$ such that $a \leq b$. With this data we define the probability mass function
\begin{equation}\label{S32DefP}
p(x|a,b,q) = \frac{q^x}{\sum_{y = a}^b q^y} \mbox{ for } x \in \mathbb{Z}, a \leq x \leq b, \mbox{ and } p(x|a,b,q) = 0 \mbox{ otherwise.}
\end{equation}
We also define the corresponding cumulative distribution function (cdf) by
\begin{equation}\label{S32DefF}
F(x|a,b,q) = \sum_{y \leq x} p(y|a,b,q) = \begin{cases} 0 &\mbox{ if } x < a \\ 1 & \mbox{ if } x \geq b \\ \frac{1 - q^{\lfloor x \rfloor - a + 1}}{1 - q^{b - a + 1}} &\mbox{ if } x \in (a,b). \end{cases}
\end{equation}
\end{definition}
\begin{remark}\label{S32Rem1} In words, the distribution $p(x|a,b,q)$ in Definition \ref{GeomDistr} is just a geometric distribution with parameter $q$, conditioned to be inside $[0,b -a]$, and translated by $a$. 
\end{remark}

We next introduce the the sampling algorithm.

{\bf \raggedleft Sampling algorithm.} Assume $M, N, X, Y$ are as in Definition \ref{SPD}. We proceed to sample $\{\lambda(n,m): m = 0, \dots, M, n = 0, \dots, N\}$ as follows. 
\begin{enumerate}
    \item[1.] {\em Initialization}. We set $\lambda(0,m) = \emptyset = \lambda(n, 0)$ for $m = 0, \dots, M$ and $n = 0, \dots, N$.
    \item[2.] {\em Update $n \rightarrow n+1$}. Assume that $(\lambda(n,1), \dots, \lambda(n,M))$ have already been sampled. We proceed to sequentially sample $\lambda(n+1,m)$ for $m = 1,\dots, M$ as follows.
    \item[3.] {\em Sampling $\lambda(n+1,m)$}. For $i \in \mathbb{N}$, we define $a_{i} = \max(\lambda_i(n+1, m-1), \lambda_i(n,m))$ and $b_i = \min(\lambda_{i-1}(n+1, m-1), \lambda_{i-1}(n,m))$, with the convention $b_1 = \infty$. Let $X_i(n+1,m)$ be independent random variables with distributions $p(\cdot|a_i, b_i, x_my_{n+1})$ as in Definition \ref{GeomDistr}, and set $\lambda_i(n+1,m) = X_i(n+1,m)$. 
\end{enumerate}
\begin{remark}\label{S32Rem2} By induction on $m + n$, one directly shows that $\lambda_{i}(n,m) = 0$ for all $i > \min(m,n)$. In particular, in the third step of the sampling algorithm we only need to generate finitely many variables $X_i(n+1,m)$ for $i = 1,\dots, \min(m,n)$, making the algorithm implementable on a computer. 
\end{remark}
\begin{remark}\label{S32Rem3} One directly observes from the above algorithm that if $0 \leq m_1 \leq m_2 \leq M$, and $0 \leq n_1 \leq n_2 \leq N$, then $\lambda(n_1,m_1) \subseteq \lambda (n_2, m_2)$ in the sense that $\lambda_i(m_1,n_1) \leq \lambda_i(m_2, n_2)$ for $i \geq 1$.  
\end{remark}

The key result we require about the above sampling algorithm is contained in the following proposition, which is a special case of \cite[Theorem 10 and Remark 12]{B11} (applied to the specializations in \cite[Example 9]{B11}). 
\begin{proposition}\label{DSP} Assume the same notation as in Definition \ref{SPD}. Suppose that $\{\lambda(n,m): m = 0, \dots, M, n = 0, \dots, N\}$ are as in the above sampling algorithm. Then, $(\lambda(N,1), \dots, \lambda(N,M))$ is distributed according to $\mathbb{P}_{X,Y}$. 
\end{proposition}

%
%
\subsection{Quantile functions}\label{Section3.3} Suppose that $F$ is a cdf on $\mathbb{R}$. We define the quantile function $Q$ associated with $F$ via
\begin{equation}\label{S3DefQ}
Q(u) = \inf \{ t : u \leq F(t) \} \mbox{ for all } u \in (0,1).
\end{equation}
The following proposition summarizes the properties we require about $Q$. Each statement below is either proved in \cite[Sections 11.4, 11a]{Pfeiffer} or follows immediately from statements there.

\begin{proposition}\label{QuantileP} Suppose that $F$ is a cdf on $\mathbb{R}$ and $Q$ is as in (\ref{S3DefQ}). Then, the following all hold.
\begin{enumerate}
    \item[P1.] The function $Q$ is increasing on $(0,1)$.
    \item[P2.] $Q(u) \leq t$ if and only if $u \leq F(t)$ for each $u \in (0,1)$ and $t \in \mathbb{R}$.
    \item[P3.] If $U$ is a uniform random variable on $(0,1)$, then $X = Q(U)$ has cdf $F$.
    \item[P4.] Let $F_1, F_2$ be two cdfs on $\mathbb{R}$, with associated quantile functions $Q_1, Q_2$. If $F_2(t) \leq F_1(t)$ for all $t \in \mathbb{R}$, then $Q_1(u) \leq Q_2(u)$ for all $u \in (0,1)$. 
\end{enumerate}
\end{proposition}

The key result we show in this section is that for certain choices of parameters, we can monotonically couple two distributions as in Definition \ref{GeomDistr}.
\begin{lemma}\label{S33Lem1} Suppose that $q_1, q_2 \in [0,1)$,  $a_1, a_2 \in \mathbb{Z}$, $b_1, b_2 \in \mathbb{Z} \cup \{\infty\}$ satisfy $a_1 \leq b_1$, $a_2 \leq b_2$, $q_1 \leq q_2$, $a_1 \leq a_2$, and $b_1 \leq b_2$. For $i = 1,2$ we let $F_i(x) = F(x|a_i,b_i, q_i)$ be the cdfs from (\ref{S32DefF}), and let $Q_i$ denote the corresponding quantile functions. Let $U$ be a uniform $(0,1)$ random variable and set $X_i = Q_i(U)$. Then, $X_i$ has distribution $F_i$ for $i =1,2$ and $X_1 \leq X_2$.
\end{lemma}
\begin{proof} The fact that $X_i$ has distribution $F_i$ for $i = 1,2$ follows from property P3 in Proposition \ref{QuantileP}. From property P4 in Proposition \ref{QuantileP} and the fact that $F_i$ are constant on each interval $(a,a+1)$ for $a \in \mathbb{Z}$, we see that to show that $X_1 \leq X_2$, it suffices to prove that for each $x \in \mathbb{Z}$
\begin{equation}\label{S33E1}
F_2(x) \leq F_1(x).
\end{equation}
If $x < a_2$, we see that (\ref{S33E1}) trivially holds as the left side is $0$, and if $x \geq b_1$, then it trivially holds as the right side is $1$. We may thus assume that $ a_2 \leq x \leq b_1$, in which case from (\ref{S32DefF}) we see that (\ref{S33E1}) is equivalent to 
$$\frac{1 - q_2^{x - a_2 + 1}}{1 - q_2^{b_2 - a_2 + 1}} \leq   \frac{1 - q_1^{x - a_1 + 1}}{1 - q_1^{b_1 - a_1 + 1}}.$$ 
The latter is clear if $q_1 = 0$ (as then the right side is equal to $1$), and so we may assume $q_1 > 0$. Clearing denominators and multiplying both sides by $q_1^{a_1} q_2^{a_2}$, the above becomes equivalent to 
\begin{equation}\label{S33E2}
0 \leq G(x| a_1,b_1,q_1, a_2, b_2, q_2) \mbox{ for } a_2 \leq x \leq b_1 ,
\end{equation}
where 
\begin{equation*}
\begin{split}
G(x| a_1,b_1,q_1, a_2, b_2, q_2) = q_1^{a_1} q_2^{x+1} - q_1^{b_1+1} q_2^{x + 1} + q_1^{b_1 + 1} q_2^{a_2} - q_1^{x+1} q_2^{a_2} - q_1^{a_1}q_2^{b_2 +1} + q_1^{x+1} q_2^{b_2 + 1} .  
\end{split}
\end{equation*}

By a direct computation we have
\begin{equation*}
    \begin{split}
     &G(x| a_1,b_1,q_1, a_2, b_2, q_2) - G(x| a_2,b_1,q_1, a_2, b_1, q_2) ,\\
     & = (q_1^{a_1} - q_1^{a_2}) ( q_2^{x+1} - q_2^{b_2+1}) + (q_2^{b_1 +1} - q_2^{b_2 + 1}) (q_1^{a_2} - q_1^{x+1})  \geq 0,
    \end{split}
\end{equation*}
where in the last inequality we used that $q_1, q_2 \in (0,1)$, $a_1 \leq a_2 \leq x \leq b_1 \leq b_2$. In particular, we conclude that to show (\ref{S33E2}), it suffices to prove that for $a \leq x \leq b$ and $0 < q_1 \leq q_2 <1$, we have 
\begin{equation}\label{S33E3}
0 \leq G(x| a,b,q_1, a, b, q_2). 
\end{equation}
If $b = \infty$, we note that 
$$G(x| a,\infty,q_1, a, \infty, q_2) =  q_1^{a} q_2^{x+1} - q_1^{x+1} q_2^{a} = q_1^a q_2^{x+1} ( 1 - (q_1/q_2)^{x+1-a}),$$
which clearly implies (\ref{S33E3}), as $q_2 \geq q_1$, and so we may assume that $b < \infty$.

We finally observe that
$$\Delta(x) = G(x| a,b,q_1, a, b, q_2) - G(x-1|a,b,q_1,a,b,q_2) = (1-q_2)q_2^x (q_1^{b + 1} - q_1^{a}) + (1-q_1)q_1^x(q_2^{a} - q_2^{b+1}).  $$
The latter shows that $\Delta(x) \leq 0$ if and only if 
$$(1-q_2)q_2^x (q_1^{a} - q_1^{b + 1}) \geq  (1-q_1)q_1^x(q_2^{a} - q_2^{b+1}) \iff (q_2/q_1)^x \geq \frac{(1-q_1)(q_2^{a} - q_2^{b+1})}{(1-q_2)(q_1^{a} - q_1^{b + 1}) }.$$
As $q_2 \geq q_1$, we see that $(q_2/q_1)^x$ is increasing in $x$, and so we either have 
$$\Delta(x) \leq 0 \mbox{ for all } x = a, a+1, \dots, b \mbox{, or }$$
$$\Delta(x) > 0 \mbox{ for } x = a, a+1, \dots, x^* \mbox{ and } \Delta(x) \leq 0 \mbox{ for } x = x^* + 1, \dots, b,$$
for some  $x^* \in [a,b] \cap \mathbb{Z}$. In both cases, we see that $G(x| a,b,q_1, a, b, q_2)$ is minimal on $[a-1,b]$ when $x = a-1$, or $x = b$. Since by direct computation 
\begin{equation*}
G(a-1| a,b,q_1, a, b, q_2) = 0 = G(b| a,b,q_1, a, b, q_2) ,
\end{equation*}
we conclude (\ref{S33E3}) and hence the lemma.
\end{proof}

%
%
\subsection{Monotone couplings of Schur processes}\label{Section3.4} In this section we establish monotone couplings for the Schur processes in Definition \ref{SPD}. They are established by appropriately coupling the Schur dynamics in Section \ref{Section3.2} using Lemma \ref{S33Lem1}. The main result of the section is as follows.

\begin{proposition}\label{MonCoup2} Fix $M, N \in \mathbb{N}$, $A,B \in \mathbb{Z}_{\geq 0}$, and let 
\begin{equation}\label{S34E1}
\mathcal{P}_{M,N} = \{(\vec{x}, \vec{y}) \in [0,\infty)^M \times [0,\infty)^N: x_i y_j < 1 \mbox{ for } i = 1, \dots, M, j = 1, \dots, N\}.
\end{equation}
There is a probability space $(\Omega, \mathcal{F}, \mathbb{P})$ and two families of random sequences 
$$(\lambda^1[X,Y], \dots, \lambda^{M}[X,Y]), (\tilde{\lambda}^1[X,Y], \dots, \tilde{\lambda}^{M}[X,Y])  \in \mathbb{Y}^{M},$$ 
indexed by $(X,Y) \in \mathcal{P}_{M,N}$, so that the following hold. Under $\mathbb{P}$ the distribution of the sequences $(\lambda^1[X,Y], \dots, \lambda^{M}[X,Y])$ and $(\tilde{\lambda}^1[X,Y], \dots, \tilde{\lambda}^{M}[X,Y])$ is $\mathbb{P}_{X,Y}$ as in Definition \ref{SPD}. In addition, we have for each $(X,Y) \in \mathcal{P}_{M,N}, (\tilde{X}, \tilde{Y}) \in \mathcal{P}_{M,N}$, $\omega \in \Omega$, $k \in \mathbb{N}$, and $j \in \{1, \dots, M\}$ that
\begin{equation}\label{MonIneq2}
 \lambda_{k+\max(A,B)}^j[X,Y](\omega) \leq \tilde{\lambda}_k^j[\tilde{X}, \tilde{Y}](\omega),
\end{equation}
provided that $x_{i+B}y_{j+A} \leq \tilde{x}_i\tilde{y}_j$ for all $i = 1, \dots, M-B$ and $j = 1, \dots, N-A$.
\end{proposition}
\begin{remark}\label{Rem.MonCoup} Observe that from (\ref{MonIneq2}) with $A = B= 0$, we have in particular
\begin{equation}\label{Eq.MonIneqRem}
\sum_{i = 1}^k \lambda_i^j[X,Y](\omega) \leq \sum_{i = 1}^k\tilde{\lambda}_i^j[\tilde{X}, \tilde{Y}](\omega),
\end{equation}
when $x_iy_j \leq \tilde{x}_i\tilde{y}_j$. In \cite[Proposition 3.7]{dimitrov2024airy} we established a monotone coupling of Schur processes, which ensures (\ref{Eq.MonIneqRem}), and is based on the {\em Robinson-Schensted-Knuth (RSK)} correspondence and {\em Greene's theorem}. It is worth pointing out that the coupling in Proposition \ref{MonCoup2} is different from the one in \cite[Proposition 3.7]{dimitrov2024airy} even when $A= B = 0$, for which (\ref{MonIneq2}) may fail for general $k$, see \cite[Remark 3.9]{dimitrov2024airy}.
\end{remark}
\begin{proof} Let $(\Omega, \mathcal{F}, \mathbb{P})$ be any probability space that supports $\{U(n,m,k): m, n, k \in \mathbb{N}\}$, where the latter are i.i.d. uniform $(0,1)$ random variables. Define further the random variables
\begin{equation}\label{S34L2E1}
\tilde{U}(n,m,k) = U(n+A,m +B,k + \max(A,B)),
\end{equation}
which we observe are again i.i.d. uniform $(0,1)$. We construct $\{\lambda[X,Y](n,m): m = 0,\dots, M, n = 0, \dots, N\}$ and $\{\tilde{\lambda}[X,Y](n,m): m = 0,\dots, M, n = 0, \dots, N\}$  using the sampling algorithm from Section \ref{Section3.2}. We start by setting 
$$\lambda[X,Y](0,m) = \lambda[X,Y](n, 0) = \tilde{\lambda}[X,Y](0,m) = \tilde{\lambda}[X,Y](n, 0)= \emptyset $$
for $m = 0, \dots, M$ and $n = 0, \dots, N$. 

If $\lambda[X,Y](n-1,m)$, $\lambda[X,Y](n, m-1)$, $\tilde{\lambda}[X,Y](n-1,m)$, $\tilde{\lambda}[X,Y](n, m-1)$ have been constructed, we set for $k \geq 1$
\begin{equation}\label{S34L2E2}
\lambda_k[X,Y](n,m) = Q^{n,m}_k[X,Y](U(n,m,k))\mbox{ and } \tilde{\lambda}_k[X,Y](n,m) = \tilde{Q}^{n,m}_k[X,Y](\tilde{U}(n,m,k)),
\end{equation}
where $Q^{n,m}_k[X,Y]$ is the quantile function of the distribution $F(\cdot | a_k,b_k,q)$ from (\ref{S32DefF}) with $a_k, b_k, q$ given by
\begin{equation}\label{S34L1E2}
\begin{split}
&a_k = \max(\lambda_{k}[X,Y](n,m-1), \lambda_{k}[X,Y](n-1,m)), \hspace{2mm} \\
&b_k = \min(\lambda_{k-1}[X,Y](n,m-1), \lambda_{k-1}[X,Y](n-1,m)), \hspace{2mm} q = x_m y_n,
\end{split}
\end{equation}
and $b_1 = \infty$. In addition, $\tilde{Q}^{n,m}_k[X,Y]$ is the quantile function of the distribution $F(\cdot | \tilde{a}_k,\tilde{b}_k,q)$ with $\tilde{a}_k, \tilde{b}_k$ as in (\ref{S34L1E2}) but with $\lambda$ replaced with $\tilde{\lambda}$. Finally, we set $\lambda^i[X,Y] = \lambda[X,Y](N,i)$ and $\tilde{\lambda}^i[X,Y] = \tilde{\lambda}[X,Y](N,i)$ for $i = 1, \dots, M$, and note that from Proposition \ref{DSP}, we have that under $\mathbb{P}$ the distribution of the sequences $(\lambda^1[X,Y], \dots, \lambda^{M}[X,Y])$ and $(\tilde{\lambda}^1[X,Y], \dots, \tilde{\lambda}^{M}[X,Y])$ is $\mathbb{P}_{X,Y}$.\\

What remains is to verify (\ref{MonIneq2}), for which it suffices to show that for $0 \leq m \leq M$, $0 \leq n \leq N$,
\begin{equation}\label{S34L2E3}
\lambda_{k + \max(A,B)}[X,Y](n,m)(\omega) \leq \tilde{\lambda}_{k}[\tilde{X}, \tilde{Y}](n,m)(\omega).
\end{equation}
Note that (\ref{S34L2E3}) trivially holds if $0 \leq m \leq B$ or $0\leq n \leq A$, as then the left side is equal to zero, cf. Remark \ref{S32Rem2}. Hence, we only need to show (\ref{S34L2E3}) when $B+1 \leq m \leq M$, and $A+1 \leq n \leq N$. Using that $\lambda[\tilde{X}, \tilde{Y}](n,m) \subseteq \lambda[\tilde{X}, \tilde{Y}](n',m')$, when $n \leq n'$ and $m \leq m'$, cf. Remark \ref{S32Rem3}, we see that it suffices to show  
\begin{equation}\label{S34L2E4}
 \lambda_{k + \max(A,B)}[X,Y](n + A,m +B)(\omega) \leq \tilde{\lambda}_k[\tilde{X}, \tilde{Y}](n,m)(\omega),
\end{equation}
for all $k \geq 1$, $m = 0, \dots, M-B$, $n = 0, \dots, N-A$. Note that (\ref{S34L2E4}) holds trivially when $m = 0$ or $n = 0$, as both sides are equal to zero. We may thus assume $m,n \geq 1$. 

Assuming that (\ref{S34L2E4}) holds when $(n,m)$ is replaced by $(n-1,m)$ or $(n,m-1)$, we see that 
\begin{equation}\label{S34L2E5}
a_{k + \max(A,B)} \leq \tilde{a}_k, \hspace{2mm} b_{k + \max(A,B)} \leq \tilde{b}_k, \hspace{2mm} q = x_{m + B}y_{n +A} \leq \tilde{x}_m \tilde{y}_n = \tilde{q},
\end{equation}
where $a_{k + \max(A,B)}, b_{k + \max(A,B)}$ are as in (\ref{S34L1E2}) with $m$ replaced with $m+B$ and $n$ replaced with $n +A$, and $\tilde{a}_k, \tilde{b}_k$ are as in (\ref{S34L1E2}) with $X,Y$ replaced with $\tilde{X}, \tilde{Y}$, and $\lambda$ replaced with $\tilde{\lambda}$. From Lemma \ref{S33Lem1} and (\ref{S34L2E1}) we conclude that for $k \geq 1$
\begin{equation*}
    \begin{split}
        &Q^{n +A,m +B}_{k + \max(A,B)}[X,Y](U(n + A,m +B,k +\max(A,B))) \\
        &= Q^{n +A,m +B}_{k + \max(A,B)}[X,Y](\tilde{U}(n,m,k))  \leq \tilde{Q}^{n,m}_k[\tilde{X},\tilde{Y}](\tilde{U}(n,m,k)),
    \end{split}
\end{equation*}
which together with (\ref{S34L2E2}) shows that (\ref{S34L2E4}) holds for $(n,m)$. The fact that (\ref{S34L2E4}) holds for all $0 \leq m \leq M -B$, $0 \leq n \leq N -A$ now follows by induction on $m+n$. 
\end{proof}

%
%
\subsection{Convergence of Schur processes}\label{Section3.5} In this section we state a weak convergence result about the Schur processes in Definition \ref{DSP}, which was shown in \cite{dimitrov2024tightness}. We begin by explaining how we scale our parameters in the following definition.

\begin{definition}\label{ParScale} We fix parameters $(a,b,c) \in \parP$ as in Definition \ref{DLP}. We fix $q \in (0,1)$ and set
\begin{equation}\label{SigmaQ}
\sigma_q = \frac{q^{1/3} (1 + q)^{1/3}}{1- q}, \hspace{2mm} p = \frac{q}{1-q}, \hspace{2mm} \sigma = \sqrt{p(1+p)}, \mbox{ and } f_q = \frac{q^{1/3}}{2 (1 + q)^{2/3}}.
\end{equation}
For $N \in \mathbb{N}$ we consider two numbers $A_N, B_N$ and sequences $\{x^N_i \}_{ i\geq 1}$ and $\{y^N_i\}_{i \geq 1}$ such that
\begin{equation}\label{ABSeq}
\begin{split}
&x_i^N =1 - \frac{1}{N^{1/3} b_i^+ \sigma_q } \mbox{ for $i = 1, \dots, B_N$, and } y_i^N =1 -\frac{1}{N^{1/3} a_i^+ \sigma_q } \mbox{ for $i = 1, \dots, A_N$},
\end{split}
\end{equation}
where $B_N \leq \min\left(\lfloor N^{1/12} \rfloor, J_b^+ \right)$ is the largest integer such that $x^N_{B_N} \geq q$, and $A_N \leq \min\left(\lfloor N^{1/12} \rfloor, J_a^+ \right)$ is the largest integer such that $y^N_{A_N} \geq q$. Here, we use the convention $x_0^N = y_0^N = 1$ so that $A_N =0$ and $B_N = 0$ are possible. We also have
\begin{equation}\label{RemSeq}
\begin{split}
\mbox{ $x_i^N = q$ for $i > B_N$ and $y_i^N = q$ for $i > A_N$.}
\end{split}
\end{equation}
Note that if $M,N \in \mathbb{N}$, we can define the ascending Schur process in Definition \ref{DSP} with parameters $X = (x_1^N, \dots, x_M^N)$ and $Y = (y_1^N, \dots, y_N^N)$ as above, since $x^N_i, y_i^N \in [q, 1)$ for all $i \in \mathbb{N}$.
\end{definition}

In the sequel we assume the same notation as in Definition \ref{ParScale}. We also assume that $M_N \geq N + N^{3/4} + 1$ and let $\mathbb{P}_N$ be the measure $\mathbb{P}_{X, Y}$ from Definition \ref{DSP} with $M = M_N$, $x_i = x_i^N$ for $i \in \{ 1, \dots, M_N \}$, $y_j = y_j^N$ for $j \in \{ 1, \dots, N \}$. If $\{\lambda^j_i: j = 1, \dots, M_N, i \geq 1 \}$ has distribution $\mathbb{P}_N$, we define the sequence of $\mathbb{N}$-indexed line ensembles $\mathfrak{L}^N = \{L_i^N\}_{i \geq 1}$ on $\mathbb{R}$ via 
\begin{equation}\label{S35E1}
L_i^N(s) = \begin{cases} \lambda_i^{N + s}  &\mbox{ if } s \in [ - N +1, M_N - N] \cap \mathbb{Z}, \\ \lambda_i^{1}  &\mbox{ if } s \leq - N, \\ \lambda_i^{M_N} &\mbox{ if } s > M_N - N,
\end{cases}
\end{equation}
extended by linear interpolation for non-integer $s$. We finally define $\mathcal{L}^N = \{\mathcal{L}_i^N\}_{i \geq 1}$ via
\begin{equation}\label{S35E2}
\mathcal{L}_i^N(t) = \sigma^{-1} N^{-1/3} \cdot \left( L_i^N(t N^{2/3}) - p t N^{2/3}  - 2p N\right).
\end{equation}

The following is a special case of \cite[Propostion 6.2]{dimitrov2024tightness}, corresponding to setting $c^+ = 0$.
\begin{proposition}\label{Prop.SchurToAiry} For any $(a,b,c) \in \parP$ the sequence $\mathcal{L}^N$ in (\ref{S35E2}) converges weakly to the line ensemble $\left\{f_q^{-1/2}\mathcal{L}_i^{a,b,c}(f_q t): i \geq 1, t \in \mathbb{R}\right\}$, where $\mathcal{L}^{a,b,c}$ is as in Proposition \ref{S12AWD}.
\end{proposition}

%
%
\subsection{Proof of Theorem \ref{Thm.MonCoupling}}\label{Section3.6}  Let $A$, $B$, $(a,b,c)$, and $(\tilde{a}, \tilde{b}, \tilde{c})$ be as in the statement of the theorem. Let $J_a^+, J_b^+$ and $\tilde{J}_a^+, \tilde{J}_b^+$ be as in Definition \ref{DLP} for the two sets of parameters $(a,b,c)$ and $(\tilde{a}, \tilde{b}, \tilde{c})$, respectively. For $N \in \mathbb{N}$, we let $M = M_N \geq N + N^{3/4} + 1$, $X^N = (x_1^N, \dots, x_M^N)$, $Y^N = (y_1^N, \dots, y_N^N)$ be as in Definition \ref{ParScale} for $(a,b,c)$. We also let $\tilde{X}^N = (\tilde{x}_1^N, \dots, \tilde{x}_M^N)$, $\tilde{Y}^N = (\tilde{y}_1^N, \dots, \tilde{y}_N^N)$ be as in Definition \ref{ParScale} with $a_i^+, b_i^+, J_a^+, J_b^+$ replaced with $\tilde{a}_i^+, \tilde{b}_i^+, \tilde{J}_a^+, \tilde{J}_b^+$. 

We first claim that 
\begin{equation}\label{Eq.ParOrdered}
x_{i+B}^N \leq \tilde{x}_i^N \mbox{ for } i = 1, \dots, M-B, \mbox{ and }y_{i+B}^N \leq \tilde{y}_i^N \mbox{ for } i = 1, \dots, N-A.
\end{equation}
We only establish the first set of inequalities in (\ref{Eq.ParOrdered}) --- the second one is established analogously.

Fix any $i \in \{1, \dots, M- B\}$. Since $\tilde{x}_i^N \in [q,1)$ by construction, we may assume $x_{i+B}^N > q$. The latter implies the following: 
\begin{equation}\label{Eq.Implications}
x_{i+B}^N = 1 - \frac{1}{N^{1/3} b_{i+B}^+ \sigma_q}, \hspace{2mm} i+B \leq \lfloor N^{1/12} \rfloor, \hspace{2mm} i+ B \leq J_b^+. 
\end{equation}
From the third part of (\ref{Eq.Implications}), we see $b_{i+B}^+ > 0$, and since $\tilde{b}_i^+ \geq b_{i+B}^+$, we conclude $\tilde{J}_b^+ \geq i$. The second part of (\ref{Eq.Implications}) shows $i \leq \lfloor N^{1/12} \rfloor$. The last two statements and Definition \ref{ParScale} show
$$ 1 - \frac{1}{N^{1/3} b_{i+B}^+ \sigma_q} \leq 1 - \frac{1}{N^{1/3} \tilde{b}_{i}^+ \sigma_q} = \tilde{x}_{i}^N.$$
The last equation, and the first part of (\ref{Eq.Implications}) show the first inequality in (\ref{Eq.ParOrdered}).\\

Combining (\ref{Eq.ParOrdered}) and Proposition \ref{MonCoup2}, we conclude that for each $N \in \mathbb{N}$, we can couple two random sequences 
$$(\lambda^1, \dots, \lambda^{M}) \mbox{ and } (\tilde{\lambda}^1, \dots, \tilde{\lambda}^{M})  \in \mathbb{Y}^{M},$$ 
whose distributions are $\mathbb{P}_{X^N,Y^N}$ and $\mathbb{P}_{\tilde{X}^N,\tilde{Y}^N}$ as in Definition \ref{SPD}, respectively, and which satisfy almost surely for all $j \in \{1, \dots, M\}$ and $k \geq 1$
\begin{equation}\label{Eq.MonIneqSpec}
 \lambda_{k+\max(A,B)}^j \leq \tilde{\lambda}_k^j.
\end{equation}
Let $\mathcal{L}^N$ be as in (\ref{S35E1}) and (\ref{S35E2}), and let $\tilde{\mathcal{L}}^N$ be as in (\ref{S35E1}) and (\ref{S35E2}) with $\lambda$ replaced with $\tilde{\lambda}$. From (\ref{Eq.MonIneqSpec}), we conclude that almost surely
\begin{equation}\label{Eq.MonIneqSpecLE}
\mathcal{L}^N_{k+\max(A,B)}(t) \leq \tilde{\mathcal{L}}^N_k(s) \mbox{ for } k \geq 1, t \in \mathbb{R}.
\end{equation}

From Proposition \ref{Prop.SchurToAiry}, we know that $\mathcal{L}^N \Rightarrow \mathcal{L}^{\infty}$, and $\tilde{\mathcal{L}}^N \Rightarrow \tilde{\mathcal{L}}^{\infty}$, where 
$$\mathcal{L}^\infty \overset{d}{=} \left\{f_q^{-1/2}\mathcal{L}_i^{a,b,c}(f_q t): i \geq 1, t \in \mathbb{R}\right\} \mbox{ and } \tilde{\mathcal{L}}^\infty \overset{d}{=} \left\{f_q^{-1/2}\mathcal{L}_i^{\tilde{a},\tilde{b},\tilde{c}}(f_q t): i \geq 1, t \in \mathbb{R}\right\}.$$
We conclude that $(\mathcal{L}^N, \tilde{\mathcal{L}}^N)$ is tight, and then by Skorohod's representation theorem, see \cite[Theorem 6.7]{Billing}, we may assume that $(\mathcal{L}^N, \tilde{\mathcal{L}}^N)$ and $(\mathcal{L}^\infty, \tilde{\mathcal{L}}^\infty)$, are defined on the same probability space, and
$$\mathcal{L}^N \overset{a.s.}{\rightarrow} \mathcal{L}^{\infty} \mbox{ and }\tilde{\mathcal{L}}^N \overset{a.s.}{\rightarrow} \tilde{\mathcal{L}}^{\infty}.$$
The last equation and (\ref{Eq.MonIneqSpecLE}) show that we can couple $\mathcal{L}^{a,b,c}$ and $\mathcal{L}^{\tilde{a},\tilde{b},\tilde{c}}$ on the same probability space, so that almost surely
$$f_q^{-1/2}\mathcal{L}^{a,b,c}_{k+\max(A,B)}(f_qt) \leq f_q^{-1/2}\mathcal{L}^{\tilde{a},\tilde{b},\tilde{c}}_k(f_q t) \mbox{ for } k \geq 1, t \in \mathbb{R}.$$
The last equation and (\ref{S1FDE}) imply the statement of the theorem.

%
%
\section{Limits of factorial moments}\label{Section4} For $t > 0$ and $(a,b,c)\in \parP$ as in Definition \ref{DLP}, we define the following random measures on $\mathbb{R}$:
\begin{equation}\label{Eq.MeasT}
M^{t}(A) = \sum_{i \geq 1} {\bf 1}\left\{ t^{-1} \cdot \left(\mathcal{A}^{a,b,c}_i(t) - t^2 \right)  \in A \right\}, \hspace{2mm} \tilde{M}^t(A) = \sum_{i \geq 1} {\bf 1}\left\{ \mathcal{A}^{a,b,c}_i(t)  \in A \right\},
\end{equation}
where $\mathcal{A}^{a,b,c}$ is as in Proposition \ref{S12AWD}. In Lemmas \ref{Lem.FirstMoment} and \ref{Lem.SecondMoment}, proved in Sections \ref{Section4.1} and \ref{Section4.2}, respectively, we analyze the limits of the first and second factorial moments of $M^t[\hat{\alpha}, \hat{\beta})$ for suitable $\hat{\alpha} \in \mathbb{R}$ and $\hat{\beta} \in [\hat{\alpha}, \infty]$. The latter will play a role in the proof of Theorem \ref{Thm.Slopes}(a,b). In Section \ref{Section4.3}, we also analyze the limit of the first factorial moment of $\tilde{M}^t[\hat{\alpha}, \infty)$, see Lemma \ref{Lem.FirstMomentFlat}, which will play a role in the proof of Theorem \ref{Thm.Slopes}(c,d). All three lemmas are established under the simplifying assumption that $a_i^+ > a_{i+1}^+$ for $i = 1, \dots, J_a^+$, which suffices for our purposes. In this section we freely use the terminology and notation for determinantal point processes from \cite[Section 2]{dimitrov2024airy}.

%
%
\subsection{First moments of $M^t$}\label{Section4.1} We begin by summarizing some notation in the following definition.

\begin{definition}\label{Def.SlopesPar} Assume the same notation as in Definition \ref{DLP}. Let $(a,b,c) \in \parP$ be such that $a^+_{i} > a^+_{i+1}$ for $i = 1, \dots, J_a^+$, and set $\hat{\alpha}_{i} = -2/a^+_{i}$ for $i = 1, \dots, J_{a}^+$. We also set $\hat{\alpha}_0 = 0$ and, if $J_{a}^+ < \infty$, we set $\hat{\alpha}_{J_a^+ + 1} = -\infty$. Since $c^{\pm} = b_i^- = a_i^- = 0$ for $i \geq 1$, we drop the plus sign from the notation and simply write $a_i, b_i, J_a, J_b$ in place of $a^+_i, b^+_i, J^+_a, J^+_b$. Lastly, we define $a_0 = \infty$, $1/a_0 = 0$, and, if $a_{i} = 0$, set $1/a_i = \infty$. 
\end{definition}

The goal of this section is to establish the following key lemma.
\begin{lemma}\label{Lem.FirstMoment} Assume the same notation as in Definition \ref{Def.SlopesPar}, and let $M^t$ be as in (\ref{Eq.MeasT}). For any $k \in \mathbb{Z}_{\geq 0}$ with $k \leq J_a$, and $\hat{\alpha} \in (\hat{\alpha}_{k+1}, \hat{\alpha}_{k})$, we have
\begin{equation}\label{Eq.FirstMoment}
\lim_{t \rightarrow \infty} \mathbb{E}\left[M^{t}[\hat{\alpha}, \infty) \right] = k.
\end{equation}
\end{lemma} 
\begin{proof} For clarity, we split the proof into four steps. In the first step, we derive a double-contour integral formula for $\mathbb{E}\left[M^{t}[\hat{\alpha}, \infty) \right]$, see (\ref{Eq.FME1}). In the second step, we deform the contours from (\ref{Eq.FME1}), and reduce the proof of the lemma to establishing that certain single or double integrals converge to zero as $t \rightarrow \infty$, see (\ref{Eq.FME3}) and (\ref{Eq.FME4}). In the third step we summarize various estimates for the integrands in (\ref{Eq.FME3}) and (\ref{Eq.FME4}), and in the fourth step we establish these two statements.\\

{\bf \raggedleft Step 1.} From Proposition \ref{S12AWD}, \cite[Lemma 2.17]{dimitrov2024airy}, and \cite[Proposition 2.13(5,6)]{dimitrov2024airy}, we know that $M^{t}$ is a determinantal point process on $\mathbb{R}$ with correlation kernel $K^t(x,y) = t K_{a,b,c}(t,tx + t^2;t,ty + t^2)$ and Lebesgue reference measure. From Definition \ref{3BPKernelDef}, we can write $K^t(x,y)$ as 
\begin{equation}\label{Eq.KernelT}
 K^t(x,y) = \frac{t}{(2\pi \im)^2} \int_{\Gamma_{\alpha }^+}  d z \int_{\Gamma_{0}^-}  dw \frac{e^{z^3/3 - w^3/3 - z^2t + w^2t - ztx+ wty +t^2x - t^2y}}{z - w } \prod_{i = 1}^{\infty} \frac{(1 + b_i z)(1 - a_i w)}{(1 - a_i z)(1 + b_i w)},
\end{equation}
where $\Gamma_{\alpha}^+, \Gamma_{0}^-$ are as in Definition \ref{DefContInf} with $0 < \alpha < 1/a_1$. We mention that in obtaining (\ref{Eq.KernelT}) from Definition \ref{3BPKernelDef}, we performed in $K^3_{a,b,c}$ the change of variables $z \mapsto z - t$ and $w \mapsto w - t$, and used the definition of $\Phi_{a,b,c}$ from (\ref{DefPhi}).

From \cite[(2.13) and (2.18)]{dimitrov2024airy}, 
\begin{equation*}
\begin{split}
 \mathbb{E}\left[M^{t}[\hat{\alpha}, \infty)  \right] = \int_{\hat{\alpha}}^{\infty}K^t(x,x)dx. 
 \end{split} 
\end{equation*}
Substituting $K^t(x,x)$ from (\ref{Eq.KernelT}), exchanging the order of the integrals, and performing the integral over $x$, we obtain
\begin{equation}\label{Eq.FME1}
\begin{split}
 \mathbb{E}\left[M^{t}[\hat{\alpha}, \infty)  \right] = \frac{1}{(2\pi \im)^2} \int_{\Gamma_{\alpha }^+}  d z \int_{\Gamma_{0}^-}  dw \frac{e^{z^3/3 - w^3/3 - z^2t + w^2t - (z-w)t\hat{\alpha}}}{(z - w)^2 } \prod_{i = 1}^{\infty} \frac{(1 + b_i z)(1 - a_i w)}{(1 - a_i z)(1 + b_i w)}.
 \end{split} 
\end{equation}
The exchange of the order of integration is justified by Fubini's theorem in view of the following bounds, which imply absolute integrability of the integrand on the product contour. For any $A \in [0, 1/a_{k+1})$, there exists $D_1 > 0$ (depending on $A$ alone) such that for $\alpha, \beta \in [0, A]$ and $z \in \Gamma_{\alpha}^+$, $w \in \Gamma_{\beta}^-$:
\begin{equation}\label{Eq.CubicEstimate}
\left|e^{z^3/3}\right| \leq D_1 \exp(-|z|^3/12), \hspace{2mm} \left|e^{w^3/3}\right| \leq D_1 \exp(-|w|^3/12).
\end{equation}
if we set $B = \sum_{i \geq 1} (a_i + b_i)$, then by (\ref{RatBound}) with $d = 1/2$, we further have
\begin{equation}\label{Eq.ProdEstimate1}
\prod_{i = 1}^{\infty} |1 + b_i z| \leq \exp\left(B|z|\right), \hspace{2mm} \prod_{i \geq 1} \frac{|1 - a_i w|}{|1 + b_i w|} \leq \exp\left(2B|w| \right).
\end{equation}
Lastly, for $d_1 = (1/2)(1 - A a_{k+1})$, we have $|1-a_iz| \geq d_1$ for $i \geq k+1$, and then from (\ref{RatBound})
\begin{equation}\label{Eq.ProdEstimate2}
\prod_{i \geq k+1}^{\infty} \frac{1}{|1 - a_i z|} \leq \exp\left(B|z|/d_1 \right).
\end{equation}
We mention that in deriving (\ref{Eq.FME1}) we also used $|z-w| \geq \Real(z-w) \geq \alpha > 0$ for $z \in \Gamma_{\alpha}^+, w \in \Gamma^-_{0}$, and also that $\Gamma_{\alpha}^+$ is bounded away from $a_1^{-1}, \dots, a_k^{-1}$, so that $\prod_{i =1}^{k} \frac{1}{(1 - a_i z)}$ is bounded on $\Gamma_{\alpha}^+$.\\

{\bf \raggedleft Step 2.} Since $(\hat{\alpha}_{k+1}, \hat{\alpha}_{k})$ is open and $\hat{\alpha} \in (\hat{\alpha}_{k+1}, \hat{\alpha}_{k})$, we can find $\epsilon_0 > 0$ such that 
$$ \left(-(1/2)\hat{\alpha} - \epsilon_0, -(1/2)\hat{\alpha} + \epsilon_0\right) \subset \left(1/a_{k}, 1/a_{k+1}\right).$$
Fix $\delta \in (0, \epsilon_0/4)$, set
\begin{equation}\label{Eq.FMDeltaEpsilon}
\begin{split}
u = -(1/2)\hat{\alpha} + \delta, \hspace{2mm} v = -(1/2)\hat{\alpha} + 2\delta, 
\end{split}
\end{equation}
and notice that by construction $u,v \in \left(1/a_{k}, 1/a_{k+1}\right)$.

We now proceed to deform $\Gamma_{\alpha}^+$ to $\Gamma_{v}^+$ in (\ref{Eq.FME1}). In the process of deformation, we cross the {\em simple} poles at $z = 1/a_j$ for $j = 1, \dots, k$. By the residue theorem, 
\begin{equation}\label{Eq.FME2}
\begin{split}
&\mathbb{E}\left[M^{t}[\hat{\alpha}, \infty)  \right] = \sum_{j = 1}^k \frac{1}{2\pi \im}  \int_{\Gamma_{0}^-} dw \frac{e^{a_j^{-3}/3 - w^3/3 - a_j^{-2}t + w^2t - (1/a_j-w)t\hat{\alpha}}}{(1/a_j - w)^2 } \\
&\times \prod_{i = 1, i \neq j}^{\infty} \frac{(1 + b_i /a_j)(1 - a_i w)}{(1 - a_i/a_j)(1 + b_i w)} \cdot \frac{(1 + b_j/a_j)(1 - a_j w)}{a_j(1 + b_j w)}\\
 & + \frac{1}{(2\pi \im)^2} \int_{\Gamma_{v}^+}  d z \int_{\Gamma_{0}^-}  dw \frac{e^{z^3/3 - w^3/3 - z^2t + w^2t - (z-w)t\hat{\alpha}}}{(z - w)^2 } \prod_{i = 1}^{\infty} \frac{(1 + b_i z)(1 - a_i w)}{(1 - a_i z)(1 + b_i w)}.
 \end{split} 
\end{equation}
We mention that the poles at $z = 1/a_j$ are simple due to our assumption that the positive $a_i$'s are distinct, as in Definition \ref{Def.SlopesPar}. In addition, we mention that the contour deformation near infinity is justified due to the decay ensured by the cubic terms in the exponential, see (\ref{Eq.CubicEstimate}), and the bounds in (\ref{Eq.ProdEstimate1}), (\ref{Eq.ProdEstimate2}).

We next deform $\Gamma_{0}^-$ to $\Gamma_{u}^-$ in (\ref{Eq.FME2}). In the process of deformation we cross the simple poles at $w = 1/a_j$ for $j = 1, \dots, k$ within the first two lines, and do not cross any poles for the third. In addition, all the residues from the first two lines are readily seen to equal $1$. By the residue theorem,
\begin{equation}\label{Eq.FMExpanded}
\mathbb{E}\left[M^{t}[\hat{\alpha}, \infty)  \right] = k + A^t + \sum_{j = 1}^k A^{j,t},
\end{equation}
where, setting $\Phi(z) = \prod_{i \geq 1}\frac{1+b_iz}{1- a_iz}$ and $\hat{\Phi}_j = \prod_{i = 1, i \neq j}^{\infty} \frac{(1 + b_i /a_j)}{(1 - a_i/a_j)} \cdot \frac{(1 + b_j/a_j)}{a_j}$, we have
\begin{equation}\label{Eq.FMAT}
A^t =  \frac{1}{(2\pi \im)^2} \int_{\Gamma_{v}^+}  d z \int_{\Gamma_{u}^-}  dw \frac{e^{z^3/3 - w^3/3 - z^2t + w^2t - (z-w)t\hat{\alpha}}}{(z - w)^2 } \cdot \frac{\Phi(z)}{\Phi(w)},
\end{equation}
and for $j = 1,\dots, k$
\begin{equation}\label{Eq.FMAJT}
A^{j,t} =  \frac{1}{2\pi \im}  \int_{\Gamma_{u}^-} dw \frac{e^{a_j^{-3}/3 - w^3/3 - a_j^{-2}t + w^2t - (1/a_j-w)t\hat{\alpha}}}{(1/a_j - w)^2 } \cdot \frac{\hat{\Phi}_j}{\Phi(w)}.
\end{equation}

From (\ref{Eq.FMExpanded}), we see that, to show (\ref{Eq.FirstMoment}), it suffices to prove
\begin{equation}\label{Eq.FME3}
\begin{split}
&\lim_{t \rightarrow \infty} A^t = 0,
 \end{split} 
\end{equation}
and for $j = 1, \dots, k$ that
\begin{equation}\label{Eq.FME4}
\begin{split}
&\lim_{t \rightarrow \infty}A^{j,t} = 0.
 \end{split} 
\end{equation}

{\bf \raggedleft Step 3.} In this step we derive various estimates for the integrands in (\ref{Eq.FMAT}) and (\ref{Eq.FMAJT}).

From the definition of the contours $\Gamma^+_v, \Gamma_u^-$ in Definition \ref{DefContInf}, we have for $z \in \Gamma_v^+$ and $w \in \Gamma_u^-$ that
$$z = v + e^{\im \phi}r_z \mbox{ and }w = u + e^{\im \psi}r_w,$$
where $\phi \in \{\pi/4, -\pi/4\}$, $\psi \in \{3\pi/4, -3\pi/4\}$, $r_z = |z-v|$, $r_w = |w - u|$. Using that $\cos(\phi) = \sqrt{2}/2$, $\cos(\psi) = -\sqrt{2}/2$, we get
\begin{equation*}
\Real[z] = v + (\sqrt{2}/2) r_z, \hspace{2mm} \Real[z^2] = v^2 + \sqrt{2}vr_z, \hspace{2mm} \Real[w] = u - (\sqrt{2}/2)r_w, \hspace{2mm} \Real[w^2] = u^2 - \sqrt{2}ur_w.
\end{equation*}
The latter shows for $z \in \Gamma_v^+$, and $w \in \Gamma_u^-$, that
\begin{equation}\label{Eq.FMQuadraticBound}
\left|e^{- z^2t} \right| = \exp\left( -tv^2 - \sqrt{2}tvr_z\right), \hspace{2mm} \left|e^{w^2t}\right| = \exp\left(tu^2 - \sqrt{2}tur_w \right).
\end{equation}
In addition, we have
\begin{equation}\label{Eq.FMLinearBound}
\left|e^{-zt\hat{\alpha}}\right| = \exp\left(- t v \hat{\alpha} - t \hat{\alpha}(\sqrt{2}/2)r_z \right) , \hspace{2mm} \left|e^{wt\hat{\alpha}}\right| = \exp \left( tu \hat{\alpha} - t\hat{\alpha}(\sqrt{2}/2)r_w \right).
\end{equation}
Combining (\ref{Eq.FMDeltaEpsilon}), (\ref{Eq.FMQuadraticBound}) and (\ref{Eq.FMLinearBound}), we conclude
\begin{equation}\label{Eq.FMBigBound}
\left|e^{- z^2t + w^2t- (z-w)t\hat{\alpha}}\right| = \exp \left( - 3t\delta^2  - \sqrt{2}t(2\delta r_z + \delta r_w) \right) \leq \exp \left(-3t\delta^2 \right).
\end{equation}
Combining (\ref{Eq.FMDeltaEpsilon}), (\ref{Eq.FMQuadraticBound}) and (\ref{Eq.FMLinearBound}), we also obtain
\begin{equation}\label{Eq.FMResBigBound}
\begin{split}
\left|e^{- a_j^{-2}t + w^2t - (1/a_j-w)t\hat{\alpha}}\right| = \exp \left( -t(a_j^{-1} + \hat{\alpha}/2)^2 + t(u + \hat{\alpha}/2)^2 - \sqrt{2}t\delta r_w  \right) \leq \exp(-15t\delta^2),
\end{split}
\end{equation}
where in the last inequality we used that from (\ref{Eq.FMDeltaEpsilon}) we have $u+\hat{\alpha}/2 = \delta$, and $|\hat{\alpha}/2 + a_i^{-1}| \geq \epsilon_0 > 4 \delta$ for all $i \geq 1$, as well as $t, \delta, r_w \geq 0$.

As $\Gamma_v^+$ is separated from $a_1^{-1}, \dots, a_k^{-1}$, and $\delta$-separated from $\Gamma_u^-$, we conclude that for some $D_2 > 0$, depending on $v, a_1, \dots, a_k$, and all $z \in \Gamma_v^+, w\in \Gamma_u^-$
\begin{equation}\label{Eq.FMRemTerms}
\frac{1}{|z-w|^2} \leq \delta^{-2}, \hspace{2mm} \prod_{i = 1}^k \frac{1}{|1-a_iz|} \leq D_2.
\end{equation}
In addition, from (\ref{Eq.FMDeltaEpsilon}) we have $|u - 1/a_i| > 3 \delta$ for all $i \geq 1$. Consequently, for all $w \in \Gamma_u^-$ and $j = 1,\dots, k$, we have
\begin{equation}\label{Eq.FMResSep}
\frac{1}{|1/a_j-w|^2} \leq \delta^{-2}.
\end{equation}

{\bf \raggedleft Step 4.} In this final step, we establish (\ref{Eq.FME3}) and (\ref{Eq.FME4}).

Combining (\ref{Eq.CubicEstimate}), (\ref{Eq.ProdEstimate1}), (\ref{Eq.ProdEstimate2}), (\ref{Eq.FMBigBound}) and (\ref{Eq.FMRemTerms}), we conclude 
\begin{equation*}
\begin{split}
&\left|A^t\right| \leq D_1D_2 \delta^{-2} \cdot e^{-3t \delta^2} \cdot \int_{\Gamma_{v}^+}  |d z| \int_{\Gamma_{u}^-} |dw| \exp\left(B|z| + B|z|/d_1 + 2B|w| - |z|^3/12 - |w|^3/12 \right),
\end{split}
\end{equation*}
where $|dz|, |dw|$ denote integration with respect to arc-length. The last inequality implies (\ref{Eq.FMAT}).

Similarly, for a fixed $j \in \{1,\dots, k\}$, combining (\ref{Eq.CubicEstimate}), (\ref{Eq.ProdEstimate1}), (\ref{Eq.FMResBigBound}) and (\ref{Eq.FMResSep}), we conclude 
\begin{equation*}
\begin{split}
&\left|A^{j,t}\right| \leq D_1 \delta^{-2} \cdot e^{-15t \delta^2} \cdot e^{a_j^{-3}/3}  |\hat{\Phi}_j| \cdot \int_{\Gamma_{u}^-} |dw| \exp\left(2B|w| - |w|^3/12 \right),
\end{split}
\end{equation*}
which implies (\ref{Eq.FMAJT}).
\end{proof}

%
%
\subsection{Second moments of $M^t$}\label{Section4.2} The goal of this section is to establish the following key lemma.
\begin{lemma}\label{Lem.SecondMoment} Assume the same notation as in Definition \ref{Def.SlopesPar}, and let $M^t$ be as in (\ref{Eq.MeasT}). For any $k \in \mathbb{Z}_{\geq 1}$ with $k \leq J_a$, $\hat{\alpha} \in (\hat{\alpha}_{k+1}, \hat{\alpha}_k)$, and $\hat{\beta} \in (\hat{\alpha}_k, \hat{\alpha}_{k-1})$, we have
\begin{equation}\label{Eq.SecondMoment}
\lim_{t \rightarrow \infty} \mathbb{E}\left[M^{t}[\hat{\alpha}, \hat{\beta}) \cdot \left(M^{t}[\hat{\alpha}, \hat{\beta}) - 1 \right) \right] = 0.
\end{equation}
\end{lemma} 
\begin{proof} For clarity, we split the proof into seven steps. In the first step, we express the second factorial moment of $M^{t}[\hat{\alpha}, \hat{\beta})$ as a sum of five terms. We introduce a common generalization of four of them by $B^t(\tilde{\alpha}, \tilde{\beta})$, and reduce the proof of the lemma to establishing its convergence, see (\ref{Eq.SME3}). In the second step, we decompose $B^t(\tilde{\alpha}, \tilde{\beta})$ into a finite sum of five types of terms, and reduce the proof of (\ref{Eq.SME3}) to establishing that each type converges, see (\ref{Eq.SMDecB}), (\ref{Eq.SMDecBi}), (\ref{Eq.SMDecBj}), (\ref{Eq.SMDecResij}) and (\ref{Eq.SMDecResii}). The latter five equations are established one at a time in the remaining five steps.\\

{\bf \raggedleft Step 1.} Using the determinantal structure of $M^t$ from Step 1 of the proof of Lemma \ref{Lem.FirstMoment} and \cite[(2.13) and (2.18)]{dimitrov2024airy}, we get
\begin{equation}\label{Eq.SME1}
\begin{split}
 &\mathbb{E}\left[M^{t}[\hat{\alpha}, \hat{\beta}) \cdot \left(M^{t}[\hat{\alpha}, \hat{\beta}) - 1 \right) \right] = A^t - B^t \mbox{, where } \\
 &A^t =\int_{\hat{\alpha}}^{\hat{\beta}}dx \int_{\hat{\alpha}}^{\hat{\beta}}dyK^t(x,x)K^t(y,y), \hspace{2mm} B^t = \int_{\hat{\alpha}}^{\hat{\beta}}dx \int_{\hat{\alpha}}^{\hat{\beta}}dy K^t(x,y)K^t(y,x).  
 \end{split} 
\end{equation} 
As in the proof of Lemma \ref{Lem.FirstMoment}, we substitute $K^t$ from (\ref{Eq.KernelT}), change the order of the integrals, and perform the integrals over $x,y$ to get
\begin{equation}\label{Eq.SME2}
\begin{split}
 &A^t =(A_1^t - A_2^t)^2, \mbox{ where } \\
 &A_1^t = \frac{1}{(2\pi \im)^2} \int_{\Gamma_{\alpha}^+}  d z \int_{\Gamma_{0}^-} dw \frac{e^{z^3/3 - w^3/3 - z^2t + w^2t- (z-w)t\hat{\alpha}}}{(z-w)^2} \cdot \frac{\Phi(z)}{\Phi(w)},\\
 & A_2^t = \frac{1}{(2\pi \im)^2} \int_{\Gamma_{\alpha}^+}  d z \int_{\Gamma_{0}^-} dw \frac{e^{z^3/3 - w^3/3 - z^2t + w^2t- (z-w)t\hat{\beta}}}{(z-w)^2} \cdot \frac{\Phi(z)}{\Phi(w)},
 \end{split} 
\end{equation} 
\begin{equation}\label{Eq.SME3}
\begin{split}
 &B^t = B^t(\hat{\alpha}, \hat{\alpha}) + B^t(\hat{\beta}, \hat{\beta})
- B^t(\hat{\alpha}, \hat{\beta}) - B^t(\hat{\beta}, \hat{\alpha}), \mbox{ where }\\
& B^t(\tilde{\alpha}, \tilde{\beta}) = \frac{1}{(2\pi \im)^4} \int_{\Gamma_{\alpha }^+}  d z_1 \int_{\Gamma_{0}^-} dw_1 \int_{\Gamma_{\alpha }^+}  d z_2 \int_{\Gamma_{0}^-} dw_2 \frac{\Phi(z_1)\Phi(z_2)}{\Phi(w_1)\Phi(w_2)} \\
 & \times \frac{e^{- (z_1-w_2)t\tilde{\alpha}- (z_2-w_1)t\tilde{\beta}} \cdot e^{z_1^3/3 - w_1^3/3 - z_1^2t + w_1^2t}  \cdot e^{z_2^3/3 - w_2^3/3 - z_2^2t + w_2^2t} }{(z_1 - w_1)(z_1 - w_2)(z_2 - w_1)(z_2 - w_2) }.
\end{split} 
\end{equation}
We recall from the proof of Lemma \ref{Lem.FirstMoment} that $\Phi(z) = \prod_{i \geq 1} \frac{1+b_iz}{1-a_iz}$ and $0 < \alpha < 1/a_1$. Also, the exchange of the order of integration is justified by the same estimates stated below (\ref{Eq.FME1}).\\

Suppose $p,q \in \mathbb{Z}_{\geq 0}$ satisfy $p,q \leq J_a$ and fix $\tilde{\alpha} \in (\hat{\alpha}_{p+1}, \hat{\alpha}_p)$, and $\tilde{\beta} \in (\hat{\alpha}_{q+1}, \hat{\alpha}_q)$. We claim 
\begin{equation}\label{Eq.SMPiece}
\lim_{t \rightarrow \infty} B^t(\tilde{\alpha}, \tilde{\beta}) = \min(p,q), 
\end{equation}
where
\begin{equation}\label{Eq.BTilde}
\begin{split}
&B^t(\tilde{\alpha}, \tilde{\beta}) = \frac{1}{(2\pi \im)^4} \int_{\Gamma_{\alpha }^+}  d z_1 \int_{\Gamma_{0}^-} dw_1 \int_{\Gamma_{\alpha }^+}  d z_2 \int_{\Gamma_{0}^-} dw_2 \frac{\Phi(z_1)\Phi(z_2)}{\Phi(w_1)\Phi(w_2)} \\
 & \frac{e^{- (z_1-w_2)t\tilde{\alpha}- (z_2-w_1)t\tilde{\beta}} \cdot e^{z_1^3/3 - w_1^3/3 - z_1^2t + w_1^2t}  \cdot e^{z_2^3/3 - w_2^3/3 - z_2^2t + w_2^2t} }{(z_1 - w_1)(z_1 - w_2)(z_2 - w_1)(z_2 - w_2) }. 
\end{split}
\end{equation}
We establish (\ref{Eq.SMPiece}) in the steps below. Here, we assume its validity and conclude the proof of the lemma.\\

Applying (\ref{Eq.SMPiece}) to each of the four terms in (\ref{Eq.SME3}), we get
$$\lim_{t \rightarrow \infty} B^t = \min(k,k) + \min(k-1,k-1) - \min(k, k-1) - \min(k-1,k) = 1.$$
From (\ref{Eq.FirstMoment}) and (\ref{Eq.FME1}), we have $A_1^t \rightarrow k$ and $A_2^t \rightarrow k-1$, and so $A^t \rightarrow 1$ as $t \rightarrow \infty$. The last two statements and (\ref{Eq.SME1}) give (\ref{Eq.SecondMoment}).\\

{\bf \raggedleft Step 2.} In the remainder of the proof we focus on establishing (\ref{Eq.SMPiece}), and without loss of generality assume $\tilde{\alpha} \leq \tilde{\beta}$ and hence $q \leq p$. As $\tilde{\beta} \in (\hat{\alpha}_{q+1}, \hat{\alpha}_q)$, $\tilde{\alpha} \in (\hat{\alpha}_{p+1}, \hat{\alpha}_p)$, we can find $\epsilon_0 > 0$ such that 
$$ (-(1/2)\tilde{\beta} - \epsilon_0, -(1/2)\tilde{\beta} + \epsilon_0) \subset \left(1/a_{q}, 1/a_{q+1}\right), \hspace{2mm} \left(-(1/2)\tilde{\alpha} - \epsilon_0, -(1/2)\tilde{\alpha} + \epsilon_0 \right) \subset \left(1/a_{p}, 1/a_{p+1}\right).$$
Fix $\delta \in (0, \epsilon_0/12)$, set
\begin{equation}\label{Eq.SMDeltaEpsilon}
\begin{split}
u_1 = -(1/2)\tilde{\beta} + \delta, \hspace{2mm} v_1 = -(1/2)\tilde{\alpha} + 6\delta, \hspace{2mm} u_2 = -(1/2)\tilde{\alpha} + 3\delta, \hspace{2mm} v_2 = -(1/2)\tilde{\beta} + 2\delta,
\end{split}
\end{equation}
and notice that by construction $u_1, v_2 \in \left(1/a_{q}, 1/a_{q+1}\right)$, while $u_2, v_1 \in \left(1/a_{p}, 1/a_{p+1}\right)$. In addition, one directly verifies that
\begin{equation}\label{Eq.SMWellSep}
\begin{split}
&\frac{1}{|z-w|} \leq \delta^{-1} \mbox{ for }z \in \Gamma_{v_1}^+ \cup \Gamma_{v_2}^+, w \in \Gamma_{u_1}^- \cup \Gamma_{u_2}^-, \\
&\frac{1}{|z-1/a_i|} \leq \delta^{-1} \mbox{ for }z \in \Gamma_{v_1}^+ \cup \Gamma_{v_2}^+ \cup \Gamma_{u_1}^- \cup \Gamma_{u_2}^- \mbox{ and }i\geq 1 , \\
& (\tilde{\alpha}/2 + 1/a_i)^2 \geq 144 \delta^2, \hspace{2mm} (\tilde{\beta}/2 + 1/a_i)^2 \geq 144 \delta^2 \mbox{ for } i \geq 1.
\end{split}
\end{equation}

We now proceed to deform the $z_i$ contours $\Gamma_{\alpha }^+$ in (\ref{Eq.SMPiece}) to $\Gamma_{v_i}^+$ for $i = 1,2$. In the process of deformation we cross the {\em simple} poles at $z_1 = 1/a_i$ for $i = 1, \dots, p$ and $z_2 = 1/a_j$ for $j = 1, \dots, q$. By the residue theorem, we get
\begin{equation}\label{Eq.SME2}
\begin{split}
&B^t(\tilde{\alpha}, \tilde{\beta}) = \tilde{B}^t(\tilde{\alpha}, \tilde{\beta}) + \sum_{i = 1}^p B_i^{1,t}(\tilde{\alpha}, \tilde{\beta}) + \sum_{j = 1}^q B_j^{2,t}(\tilde{\alpha}, \tilde{\beta})  \\
& + \sum_{i = 1}^{p} \sum_{j = 1, j\neq i}^q B_{ij}^t(\tilde{\alpha}, \tilde{\beta})+  \sum_{j = 1}^q B_{jj}^t(\tilde{\alpha}, \tilde{\beta}).
\end{split} 
\end{equation}
Here, $\tilde{B}^t(\tilde{\alpha}, \tilde{\beta})$ is as in (\ref{Eq.BTilde}) but with $\Gamma^+_{\alpha}$ replaced with $\Gamma_{v_i}^+$. For $i = 1, \dots, p$ we have
\begin{equation}\label{Eq.SMResi}
\begin{split}
&B_{i}^{1,t}(\tilde{\alpha}, \tilde{\beta}) = \frac{1}{(2\pi \im)^3} \int_{\Gamma_{0}^-} dw_1 \int_{\Gamma_{v_2}^+} dz_2 \int_{\Gamma_{0}^-} dw_2 \frac{\hat{\Phi}_i \Phi(z_2)}{\Phi(w_1) \Phi(w_2)}  \\
 & \times \frac{e^{- (1/a_i-w_2)t\tilde{\alpha}- (z_2-w_1)t\tilde{\beta}} \cdot e^{a_i^{-3}/3 - w_1^3/3-a_{i}^{-2}t + w_1^2 t} \cdot e^{z_2^3/3 - w_2^3/3-z_2^2t + w_2^2 t} }{(1/a_i - w_1)(1/a_i - w_2)(z_2 - w_1)(z_2 - w_2)},
\end{split} 
\end{equation}
where we recall $\hat{\Phi}_j =  \prod_{i = 1, i \neq j}^{\infty} \frac{(1 + b_i /a_j)}{(1 - a_i/a_j)} \cdot \frac{(1 + b_j/a_j)}{a_j}$. For $j = 1, \dots, q$, we have
\begin{equation}\label{Eq.SMResj}
\begin{split}
&B_{j}^{2,t}(\tilde{\alpha}, \tilde{\beta}) = \frac{1}{(2\pi \im)^3} \int_{\Gamma_{0}^-} dw_1 \int_{\Gamma_{v_1}^+} dz_1 \int_{\Gamma_{0}^-} dw_2 \frac{\hat{\Phi}_j \Phi(z_1)}{\Phi(w_1) \Phi(w_2)}  \\
 & \times \frac{e^{- (z_1-w_2)t\tilde{\alpha}- (1/a_j-w_1)t\tilde{\beta}} \cdot e^{z_1^3/3 - w_1^3/3-z_1^2t + w_1^2 t} \cdot e^{a_j^{-3}/3 - w_2^3/3-a_j^{-2}t + w_2^2 t} }{(1/a_j - w_1)(1/a_j - w_2)(z_1 - w_1)(z_1 - w_2)}.
\end{split} 
\end{equation}
In addition for $j = 1, \dots ,q$, we have
\begin{equation}\label{Eq.SMResii}
\begin{split}
&B_{jj}^t(\tilde{\alpha}, \tilde{\beta}) = \frac{1}{(2\pi \im)^2} \int_{\Gamma_{0}^-} dw_1  \int_{\Gamma_{0}^-} dw_2 \frac{\hat{\Phi}_j^2}{\Phi(w_1) \Phi(w_2)}  \\
 & \times \frac{e^{- (1/a_j-w_2)t\tilde{\alpha}- (1/a_j-w_1)t\tilde{\beta}} \cdot e^{a_j^{-3}/3 - w_1^3/3-a_{j}^{-2}t + w_1^2 t} \cdot e^{a_j^{-3}/3 - w_2^3/3-a_{j}^{-2}t + w_2^2 t} }{(1/a_j - w_1)^2(1/a_j - w_2)^2},
\end{split} 
\end{equation}
Lastly, for $ i \neq j$
\begin{equation}\label{Eq.SMResij}
\begin{split}
&B_{ij}^t(\tilde{\alpha}, \tilde{\beta}) = \frac{1}{(2\pi \im)^2} \int_{\Gamma_{0}^-} dw_1  \int_{\Gamma_{0}^-} dw_2 \frac{\hat{\Phi}_i\hat{\Phi}_j}{\Phi(w_1) \Phi(w_2)}  \\
 & \times \frac{e^{- (1/a_i-w_2)t\tilde{\alpha}- (1/a_j-w_1)t\tilde{\beta}} \cdot e^{a_i^{-3}/3 - w_1^3/3-a_{i}^{-2}t + w_1^2 t} \cdot e^{a_j^{-3}/3 - w_2^3/3-a_{j}^{-2}t + w_2^2 t} }{(1/a_i - w_1)(1/a_j - w_1)(1/a_i - w_2)(1/a_j - w_2)}.
\end{split} 
\end{equation}
As in the proof of Lemma \ref{Lem.FirstMoment}, the contour deformation near infinity is justified due to the decay ensured by the cubic terms in the exponential, see (\ref{Eq.CubicEstimate}), and the bounds in (\ref{Eq.ProdEstimate1}), (\ref{Eq.ProdEstimate2}). \\

{\raggedleft \em Heuristic explanation of (\ref{Eq.SME2}):} As we deform the contours, we can have that only $z_1$ picks up a pole (producing $B_i^{1,t}$), only $z_2$ picks up a pole (producing $B_j^{2,t}$), neither picks up a pole (producing $\tilde{B}^t$), both $z_1$ and $z_2$ pick up a pole and the two poles are different (producing $B_{ij}^t$), or both $z_1$ and $z_2$ pick up a pole and the two poles are the same (producing $B_{jj}^t$).\\

In view of (\ref{Eq.SME2}), to prove (\ref{Eq.SMPiece}), it suffices to show 
\begin{equation}\label{Eq.SMDecB}
\lim_{t \rightarrow \infty} \tilde{B}^t(\tilde{\alpha}, \tilde{\beta}) = 0;
\end{equation}
\begin{equation}\label{Eq.SMDecBi}
\lim_{t \rightarrow \infty} B^{1,t}_i(\tilde{\alpha}, \tilde{\beta}) = 0, \mbox{ for } i = 1, \dots, p;
\end{equation}
\begin{equation}\label{Eq.SMDecBj}
\lim_{t \rightarrow \infty} B^{2,t}_j(\tilde{\alpha}, \tilde{\beta}) = 0, \mbox{ for } j = 1, \dots, q;
\end{equation}
\begin{equation}\label{Eq.SMDecResij}
\lim_{t \rightarrow \infty} B_{ij}^t(\tilde{\alpha}, \tilde{\beta}) = 0, \mbox{ for } 1 \leq i \leq p, \hspace{2mm} 1\leq  j \leq q, \hspace{2mm} i \neq j;
\end{equation}
\begin{equation}\label{Eq.SMDecResii}
\lim_{t \rightarrow \infty} B_{jj}^t(\tilde{\alpha}, \tilde{\beta}) = 1, \mbox{ for } 1 \leq j \leq q.
\end{equation}
The above five equations are established in the steps below.\\

{\bf \raggedleft Step 3.} In this step, we prove (\ref{Eq.SMDecB}).

We start by deforming the $w_i$ contours $\Gamma_{0}^-$ in $\tilde{B}^t(\tilde{\alpha}, \tilde{\beta})$ to $\Gamma_{u_i}^-$ for $i = 1,2$. In the process of deformation, we do not cross any poles, and so by Cauchy's theorem
\begin{equation}\label{Eq.SMBSpec}
\begin{split}
&\tilde{B}^t(\tilde{\alpha}, \tilde{\beta}) = \frac{1}{(2\pi \im)^4} \int_{\Gamma_{v_1}^+}  d z_1 \int_{\Gamma_{u_1}^-} dw_1 \int_{\Gamma_{v_2}^+}  d z_2 \int_{\Gamma_{u_2}^-} dw_2 \frac{\Phi(z_1)\Phi(z_2)}{\Phi(w_1)\Phi(w_2)} \\
 & \times \frac{e^{- (z_1-w_2)t\tilde{\alpha}- (z_2-w_1)t\tilde{\beta}} \cdot e^{z_1^3/3 - w_1^3/3 - z_1^2t + w_1^2t}  \cdot e^{z_2^3/3 - w_2^3/3 - z_2^2t + w_2^2t} }{(z_1 - w_1)(z_1 - w_2)(z_2 - w_1)(z_2 - w_2) }.
\end{split}
\end{equation}
As in the proof of Lemma \ref{Lem.FirstMoment}, we have for $z_1,z_2 \in \Gamma_v^+$ and $w_1,w_2 \in \Gamma_u^-$ that
\begin{equation}\label{Eq.SMParametrization}
z_i = v_i + e^{\im \phi_i}r_{z_i} \mbox{ and }w_i = u_i + e^{\im \psi_i}r_{w_i},
\end{equation}
where $\phi_i \in \{\pi/4, -\pi/4\}$, $\psi_i \in \{3\pi/4, -3\pi/4\}$, $r_{z_i} = |z_i-v_i|$, $r_{w_i} = |w_i - u_i|$. Using (\ref{Eq.FMQuadraticBound}) and (\ref{Eq.FMLinearBound}), we obtain the following analogue of (\ref{Eq.FMBigBound}) for $z_i \in \Gamma_{v_i}^+$, $w_i \in \Gamma_{u_i}^-$:
\begin{equation}\label{Eq.SMQuadraticBound}
\begin{split}
&\left|e^{- (z_1-w_2)t\tilde{\alpha} -z_1^2t + w_2^2 t} \right| \leq \exp \left(-3t[3\delta]^2 \right), \hspace{2mm} \left|e^{- (z_2-w_1)t\tilde{\beta} -z_2^2t + w_1^2 t} \right| \leq \exp \left(-3t\delta^2 \right).
\end{split}
\end{equation}

We may now apply the bounds in (\ref{Eq.CubicEstimate}), (\ref{Eq.ProdEstimate1}), (\ref{Eq.ProdEstimate2}) and (\ref{Eq.FMRemTerms}), with $z,w$ replaced with $z_1, w_2$, $k = p$, and $A = v_1$ to conclude that for some $D_1, D_2, d_1 > 0$ (depending on $v_1, a_1, \dots, a_p$), and all $z_1 \in \Gamma_{v_1}^+, w_2 \in \Gamma_{u_2}^-$
\begin{equation}\label{Eq.SMBigBound1}
\begin{split}
&\left|e^{z_1^3/3 -w_2^3/3} \cdot \frac{\Phi(z_1)}{\Phi(w_2)} \right| \leq D_1^2D_2 \exp\left( B|z_1| + B|z_1|/d_1 + 2B|w_2| -|z_1|^3/12 - |w_2|^3/12 \right),
\end{split}
\end{equation}
where we recall that $B = \sum_{i \geq 1} (a_i + b_i)$.

Using the same bounds with $k = q$ and $A = v_2$, we can find $\hat{D}_1, \hat{D}_2, \hat{d}_1 > 0$ (depending on $v_2, a_1, \dots, a_q$), such that for all $z_2 \in \Gamma_{v_2}^+, w_1 \in \Gamma_{u_1}^-$
\begin{equation}\label{Eq.SMBigBound2}
\begin{split}
&\left|e^{z_2^3/3 -w_1^3/3} \cdot \frac{\Phi(z_2)}{\Phi(w_1)} \right| \leq \hat{D}_1^2\hat{D}_2 \exp\left( B|z_2| + B|z_2|/\hat{d}_1 + 2B|w_1| -|z_2|^3/12 - |w_1|^3/12 \right).
\end{split}
\end{equation}

Combining (\ref{Eq.SMBSpec}) with (\ref{Eq.SMWellSep}), (\ref{Eq.SMQuadraticBound}), (\ref{Eq.SMBigBound1}), and (\ref{Eq.SMBigBound2}), we conclude 
\begin{equation*}
\begin{split}
&\left|\tilde{B}^t(\tilde{\alpha}, \tilde{\beta})\right| \leq D_1^2D_2  \cdot \hat{D}_1^2\hat{D}_2 \cdot \delta^{-4} \cdot \exp \left(-30t\delta^2 \right) \int_{\Gamma_{v_1}^+}  |d z_1| \int_{\Gamma_{u_1}^-} |dw_1| \int_{\Gamma_{v_2}^+}  |d z_2| \int_{\Gamma_{u_2}^-} |dw_2|\\
& e^{B|z_1| + B|z_1|/d_1 + 2B|w_2| + B|z_2| + B|z_2|/\hat{d}_1 + 2B|w_1|-|z_1|^3/12 - |w_1|^3/12 -|z_2|^3/12 - |w_2|^3/12 }.
\end{split}
\end{equation*}
The last inequality implies (\ref{Eq.SMDecB}) since the integral is finite.\\

{\bf \raggedleft Step 4.} In this step, we fix $i \in \{1,\dots, p\}$, and prove (\ref{Eq.SMDecBi}).

We first observe that 
\begin{equation}\label{Eq.SMKillPole}
\frac{1}{(1/a_j - w)\Phi(w)} = a_j\prod_{m = 1, m \neq j}^{\infty} (1 -a_m w) \cdot \prod_{m = 1}^{\infty} \frac{1}{1 + b_m w},
\end{equation}
which by (\ref{RatBound}) with $d = 1/2$ implies for all $w \in \Gamma_{\beta}^-$ with $\beta \geq 0$ that
\begin{equation}\label{Eq.SMKillPoleBound}
\left|\frac{1}{(1/a_j - w)\Phi(w)}\right| \leq a_j \cdot \exp\left(2B|w|\right).
\end{equation}

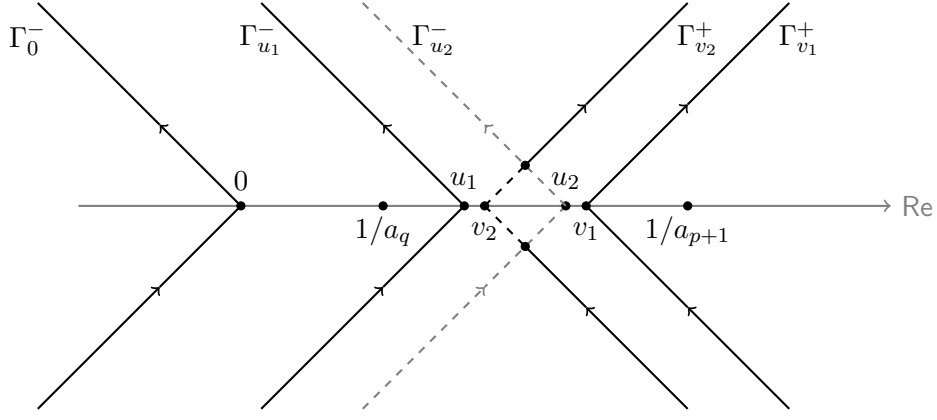
\begin{figure}[h]
    \centering
     \begin{tikzpicture}[scale=2.7]
		\begin{scope}[shift={(0,0)}]
        \def\tra{3} 
        \draw[->, thick, gray] (-2,0)--(2,0) node[right]{$\Real$};

        \draw[-,thick][black] (-1.2,0) -- (-1.6,-0.4);
        \draw[->,thick][black] (-2.2,-1) -- (-1.6,-0.4);
        \draw[black, fill = black] (-1.2,0) circle (0.02);
        \draw (-1.2, 0.125) node{$0$};
        \draw[->,thick][black] (-1.2,0) -- (-1.6,0.4);
        \draw[-,thick][black]  (-2.2,1) -- (-1.6,0.4);       

        \draw[-,thick][black] (-0.1,0) -- (-0.5,-0.4);
        \draw[->,thick][black] (-1.1,-1) -- (-0.5,-0.4);
        \draw[black, fill = black] (-0.1,0) circle (0.02);
        \draw (-0.1,0.125) node{$u_1$};
        \draw[->,thick][black] (-0.1,0) -- (-0.5,0.4);
        \draw[-,thick][black]  (-1.1,1) -- (-0.5,0.4);

        \draw[-, dashed, thick][gray] (0.4,0) -- (0,-0.4);
        \draw[->, dashed, thick][gray] (-0.6,-1) -- (0,-0.4);
        \draw[black, fill = black] (0.4,0) circle (0.02);
        \draw (0.4,0.125) node{$u_2$};
        \draw[->, dashed, thick][gray] (0.4,0) -- (0,0.4);
        \draw[-, dashed, thick][gray]  (-0.6,1) -- (0,0.4);

        \draw[-,thick][black] (0.5, 0) -- (1,-0.5);
        \draw[->,thick][black] (1.5, -1) -- (1, -0.5);
        \draw[black, fill = black] (0.5,0) circle (0.02);
        \draw (0.5,-0.125) node{$v_1$};
        \draw[->,thick][black] (0.5,0) -- (1,0.5);
        \draw[-,thick][black]  (1,0.5) -- (1.5,1);

        \draw[-, dashed, thick][black] (0, 0) -- (0.2,-0.2);
		\draw[-,thick][black]  (0.2,-0.2) -- (0.5,-0.5);

        \draw[->,thick][black] (1, -1) -- (0.5, -0.5);
        \draw[black, fill = black] (0,0) circle (0.02);
        \draw (0,-0.125) node{$v_2$};
        \draw[->,thick][black] (0.2,0.2) -- (0.5,0.5);
        \draw[-,thick][black]  (0.5,0.5) -- (1,1);
        \draw[-, dashed, thick][black]  (0,0) -- (0.2,0.2);

        \draw[black, fill = black] (1,0) circle (0.02);
        \draw (1,-0.125) node{$1/a_{p+1}$};	

        \draw[black, fill = black] (-0.5,0) circle (0.02);
        \draw (-0.5,-0.125) node{$1/a_{q}$};	

        \draw[black, fill = black] (0.2,0.2) circle (0.02);
        \draw[black, fill = black] (0.2,-0.2) circle (0.02);

        \draw (1.55,0.825) node{$\Gamma^+_{v_1}$};
        \draw (1.05,0.825) node{$\Gamma^+_{v_2}$};
        \draw (-1.1,0.825) node{$\Gamma^-_{u_1}$};
        \draw (-0.25,0.825) node{$\Gamma^-_{u_2}$};
        \draw (-2.25,0.825) node{$\Gamma^-_{0}$};

        \end{scope}

    \end{tikzpicture} 
    \caption{The figure depicts the contours $\Gamma^-_0$, $\Gamma^-_{u_1}$, $\Gamma_{u_2}^-$, $\Gamma_{v_1}^+$, $\Gamma_{v_2}^+$. The contours $\Gamma_{u_2}^-$ and $\Gamma_{v_2}^+$ intersect at two points, and $\Gamma_{v_2}^{+,0}$ is the portion of $\Gamma_{v_2}^+$ between them (drawn in dashed black). If $z_2 \in \Gamma_{v_2}^{+,0}$, then deforming $\Gamma_0^-$ to $\Gamma_{u_2}^-$ crosses $z_2$; otherwise, it does not.}
    \label{S41}
\end{figure}

We proceed to deform the $w_1,w_2$ contours $\Gamma_{0}^-$ in $B_i^{1,t}(\tilde{\alpha}, \tilde{\beta})$ to $\Gamma_{u_1}^-, \Gamma_{u_2}^-$, respectively. In the process of deforming $w_1$, in view of (\ref{Eq.SMKillPole}), we do not cross any poles. When deforming $w_2$, we cross the simple pole at $z_2$, provided that $z_2 \in \Gamma_{v_2}^{+,0}:= \{z \in \Gamma_{v_2}^{+}: |\Imag(z)| \leq (u_2-v_2)/2\}$, see Figure \ref{S41}. By the residue theorem, we conclude
\begin{equation}\label{Eq.SMResiV2}
\begin{split}
&B_{i}^{1,t}(\tilde{\alpha}, \tilde{\beta}) = B_{i}^{1,1,t}(\tilde{\alpha}, \tilde{\beta}) + B_{i}^{1,2,t}(\tilde{\alpha}, \tilde{\beta}) \mbox{, where }\\
&B_{i}^{1,1,t}(\tilde{\alpha}, \tilde{\beta}) = \frac{1}{(2\pi \im)^3} \int_{\Gamma_{u_1}^-} dw_1 \int_{\Gamma_{v_2}^+} dz_2 \int_{\Gamma_{u_2}^-} dw_2 \frac{\hat{\Phi}_i \Phi(z_2)}{\Phi(w_1) \Phi(w_2)}  \\
 & \times \frac{e^{- (1/a_i-w_2)t\tilde{\alpha}- (z_2-w_1)t\tilde{\beta}} \cdot e^{a_i^{-3}/3 - w_1^3/3-a_{i}^{-2}t + w_1^2 t} \cdot e^{z_2^3/3 - w_2^3/3-z_2^2t + w_2^2 t} }{(1/a_i - w_1)(1/a_i - w_2)(z_2 - w_1)(z_2 - w_2)}, \\
 &B_{i}^{1,2,t}(\tilde{\alpha}, \tilde{\beta}) = \hspace{-0.5mm} \frac{1}{(2\pi \im)^2} \int_{\Gamma_{u_1}^-} dw_1 \int_{\Gamma_{v_2}^{+,0}} \hspace{-0.5mm} dz_2\frac{\hat{\Phi}_i }{\Phi(w_1) } \cdot \frac{e^{- (1/a_i-z_2)t\tilde{\alpha}- (z_2-w_1)t\tilde{\beta} + a_i^{-3}/3 - w_1^3/3-a_{i}^{-2}t + w_1^2 t} }{(1/a_i - w_1)(1/a_i - z_2)(z_2 - w_1)}.
\end{split} 
\end{equation}
We mention that the contours $\Gamma_{v_2}^+$ and $\Gamma_{u_2}^-$ intersect each other at the two points: $(v_2 + u_2)/2 \pm \im (u_2 - v_2)$. The integrand is singular near each of these intersection points, due to the term $z_2 - w_2$ in the denominator. In a local parametrization around these points, the singularities are like that of the function $(x^2 + y^2)^{-1/2}$ near $(0,0)$ in $\mathbb{R}^2$, which makes them {\em integrable}. In particular, all integrals in (\ref{Eq.SMResiV2}) are well-defined and finite. \\

Setting $w_2 = u_2 + e^{\im \psi_2}r_{w_2}$ as in (\ref{Eq.SMParametrization}), we obtain
\begin{equation}\label{Eq.SMBiBoundW2}
\begin{split}
&\left| e^{-(1/a_i - w_2) t \tilde{\alpha} - a_i^{-2}t + w_2^2t} \right| = \exp\left(t (u_2 + \tilde{\alpha}/2)^2 - t (\tilde{\alpha}/2 + 1/a_i)^2 - \sqrt{2} t r_{w_2} (u_2 + \tilde{\alpha}/2) \right) \\
& = \exp\left( 9t\delta^2 - t (\tilde{\alpha}/2 + 1/a_i)^2 - 3\sqrt{2} t \delta r_{w_2} \right) \leq \exp\left( - 135 t \delta^2 \right),
\end{split}
\end{equation}
where in going from the first to the second line we used $u_2 + \tilde{\alpha}/2 = 3\delta$ from (\ref{Eq.SMDeltaEpsilon}), and in the last inequality we used (\ref{Eq.SMWellSep}) and $r_w,\delta,t\geq 0$.

Combining the second inequality in (\ref{Eq.CubicEstimate}) with $w$ replaced with $w_2$, (\ref{Eq.SMWellSep}), the second inequality in (\ref{Eq.SMQuadraticBound}), (\ref{Eq.SMBigBound2}), (\ref{Eq.SMKillPoleBound}) with $w$ replaced with $w_2$, and (\ref{Eq.SMBiBoundW2}), we conclude
\begin{equation*}
\begin{split}
&\left|B_{i}^{1,1,t}(\tilde{\alpha}, \tilde{\beta})\right| \leq D_1 \cdot \delta^{-2} \cdot \exp(-3t \delta^2) \cdot \hat{D}_1^2 \hat{D}_2    \cdot  a_i \cdot  \exp\left(-135t \delta^2 \right) \cdot |\hat{\Phi}_i| \cdot e^{a_i^{-3}/3} \\
&  \times \int_{\Gamma_{u_1}^-} |dw_1| \int_{\Gamma_{v_2}^+} |dz_2| \int_{\Gamma_{u_2}^-} |dw_2|  \frac{e^{2B|w_2| + B|z_2| + B|z_2|/\hat{d}_1 + 2B|w_1| - |z_2|^3/12 - |w_1|^3/12 - |w_2|^3/12}}{ |z_2 - w_2|}.
\end{split}
\end{equation*}
The last integral is finite, since $|z_2-w_2|^{-1}$ is locally integrable by our discussion below (\ref{Eq.SMResiV2}), while the cubic terms in the exponential provide sufficient decay near infinity. Consequently, 
\begin{equation}\label{Eq.SMB11iVanish}
\lim_{t \rightarrow \infty}B_{i}^{1,1,t}(\tilde{\alpha}, \tilde{\beta}) = 0.
\end{equation}

To analyze $B_{i}^{1,2,t}(\tilde{\alpha}, \tilde{\beta})$, we proceed to deform $\Gamma_{v_2}^{+,0}$ to $\Gamma_{u_2}^{-,0} := \{z \in \Gamma_{u_2}^-: |\Imag(z)| \leq (u_2-v_2)/2\}$. In the process of deformation, we may cross the simple pole at $1/a_i$, which happens precisely when $q < i \leq p$, see Figure \ref{S42}. By the residue theorem,
\begin{equation}\label{Eq.SMB12iV2}
\begin{split}
&B_{i}^{1,2,t}(\tilde{\alpha}, \tilde{\beta}) = \frac{1}{(2\pi \im)^2} \int_{\Gamma_{u_1}^-} dw_1 \int_{\Gamma_{u_2}^{-,0}} dz_2\frac{\hat{\Phi}_i }{\Phi(w_1) } \cdot  \frac{e^{- (1/a_i-z_2)t\tilde{\alpha}- (z_2-w_1)t\tilde{\beta}+ a_i^{-3}/3 - w_1^3/3-a_{i}^{-2}t + w_1^2 t} }{(1/a_i - w_1)(1/a_i - z_2)(z_2 - w_1)} \\
 & + \frac{{\bf 1}\{q < i \leq p\}}{2\pi \im} \cdot \int_{\Gamma_{u_1}^-} dw_1 \frac{\hat{\Phi}_i }{\Phi(w_1) } \cdot \frac{e^{- (1/a_i-w_1)t\tilde{\beta}} \cdot e^{a_i^{-3}/3 - w_1^3/3-a_{i}^{-2}t + w_1^2 t} }{(1/a_i - w_1)^2}.
\end{split}
\end{equation}

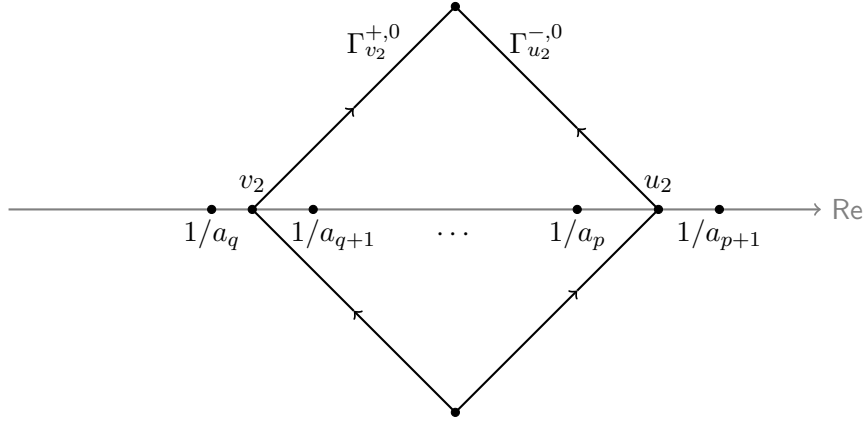
\begin{figure}[h]
    \centering
     \begin{tikzpicture}[scale=2.7]
		\begin{scope}[shift={(0,0)}]
        \def\tra{3} 
        \draw[->, thick, gray] (-2,0)--(2,0) node[right]{$\Real$};

        \draw[-, thick][black] (1.2,0) -- (0.8,-0.4);
        \draw[->, thick][black] (0.2,-1) -- (0.8,-0.4);
        \draw[black, fill = black] (1.2,0) circle (0.02);
        \draw (1.2,0.125) node{$u_2$};
        \draw[->,  thick][black] (1.2,0) -- (0.8,0.4);
        \draw[-, thick][black]  (0.2,1) -- (0.8,0.4);

        \draw[-, thick][black] (-0.8, 0) -- (-0.3,-0.5);
        \draw[->,thick][black] (0.2, -1) -- (-0.3, -0.5);
        \draw[black, fill = black] (-0.8,0) circle (0.02);
        \draw (-0.8,0.125) node{$v_2$};
        \draw[->,thick][black] (-0.8,0) -- (-0.3,0.5);
        \draw[-,thick][black]  (-0.3,0.5) -- (0.2,1);

        \draw[black, fill = black] (1.5,0) circle (0.02);
        \draw (1.5,-0.125) node{$1/a_{p+1}$};	

        \draw (0.2,-0.125) node{$\cdots$};

        \draw[black, fill = black] (-0.5,0) circle (0.02);
        \draw (-0.4,-0.125) node{$1/a_{q+1}$};	

        \draw[black, fill = black] (0.8,0) circle (0.02);
        \draw (0.8,-0.125) node{$1/a_{p}$};	

        \draw[black, fill = black] (-1,0) circle (0.02);
        \draw (-1,-0.125) node{$1/a_{q}$};

        \draw[black, fill = black] (0.2,1) circle (0.02);
        \draw[black, fill = black] (0.2,-1) circle (0.02);

        \draw (-0.2,0.825) node{$\Gamma^{+,0}_{v_2}$};
        \draw (0.6,0.825) node{$\Gamma^{-,0}_{u_2}$};

        \end{scope}

    \end{tikzpicture} 
    \caption{The figure depicts the contours $\Gamma_{u_2}^{-,0}$ and $\Gamma_{v_2}^{+,0}$. The points in $\{1/a_i\}_{m = 0}^{J_a}$  enclosed by the two contours are precisely $1/a_{q+1}, 1/a_{q+2}, \dots, 1/a_p$.}
    \label{S42}
\end{figure}

We observe that for $z_2 \in \Gamma_{u_2}^{-,0}$, we have $\Real(z_2) \geq (v_2 + u_2)/2 = -\tilde{\alpha}/4 - \tilde{\beta}/4 + 5 \delta/2$, where the last equality used (\ref{Eq.SMDeltaEpsilon}). This gives for $z_2 \in \Gamma_{u_2}^{-,0}$
\begin{equation}\label{Eq.SMB12iZ2Bound}
\left|e^{z_2 t \tilde{\alpha} - z_2 t \tilde{\beta}}\right| \leq \exp \left( - t \tilde{\alpha}^2/4 + t\tilde{\beta}^2/4 \right),
\end{equation}
where we implicitly used $t, \delta > 0$, and $\tilde{\beta} \geq \tilde{\alpha}$. In addition, from the parametrization (\ref{Eq.SMParametrization}), we have
\begin{equation}\label{Eq.SMB12iW2Bound}
\left|e^{w_1t\tilde{\beta} + w_1^2 t}  \right|  = \exp\left(9t\delta^2 - t\tilde{\beta}^2/4-3\sqrt{2}t \delta r_{w_1} \right) \leq \exp\left(9t \delta^2 - t \tilde{\beta}^2/4 \right),
\end{equation}
where we used $u_2 + \tilde{\beta}/2 = 3\delta$ from (\ref{Eq.SMDeltaEpsilon}), and in the last inequality again we used $t, \delta, r_{w_1} \geq 0$.
Combining (\ref{Eq.SMB12iZ2Bound}) with (\ref{Eq.SMB12iW2Bound}), we conclude 
\begin{equation}\label{Eq.SMB12iZW2Bound}
\begin{split}
&\left|e^{- (1/a_i-z_2)t\tilde{\alpha}- (z_2-w_1)t\tilde{\beta}} \cdot e^{-a_{i}^{-2}t + w_1^2 t} \right| \leq \exp\left( 9t\delta^2 - t(\tilde{\alpha}/2 +1/a_i)^2 \right) \leq \exp\left( - 135t\delta^2\right), \\
&\left|e^{- (1/a_i-w_1)t\tilde{\beta}} \cdot e^{-a_{i}^{-2}t + w_1^2 t} \right| \leq \exp(9t\delta^2 - t(\tilde{\beta}/2 + 1/a_i)^2) \leq \exp(-135 t \delta^2),
\end{split}
\end{equation}
where in the right inequalities we used the third line in (\ref{Eq.SMWellSep}).

Note that from (\ref{Eq.SMDeltaEpsilon}), we have $u_2 -u_1 \geq 2\delta$. This means that $\Gamma_{u_1}^-, \Gamma_{u_2}^-$ are at least $\delta$-separated, so that for $z_2 \in \Gamma_{u_2}^{-,0}$ and $w_1 \in \Gamma_{u_1}^-$ we have 
\begin{equation}\label{Eq.SMB12Sep}
\begin{split}
&\frac{1}{|z_2 - w_1|} \leq \delta^{-1}. 
\end{split}
\end{equation}

Combining the second inequality in (\ref{Eq.CubicEstimate}) with $w$ replaced with $w_1$, (\ref{Eq.SMWellSep}), (\ref{Eq.SMKillPoleBound}) with $w$ replaced with $w_1$, (\ref{Eq.SMB12iZW2Bound}), and (\ref{Eq.SMB12Sep}), we conclude
\begin{equation*}
\begin{split}
&\left|B_{i}^{1,2,t}(\tilde{\alpha}, \tilde{\beta})\right| \leq D_1 \cdot \delta^{-1} \cdot a_i \cdot \exp(-135t \delta^2) \cdot \delta^{-1} \cdot |\hat{\Phi}_i|e^{a_i^{-3}/3} \cdot \int_{\Gamma_{u_1}^-} |dw_1| \int_{\Gamma_{u_2}^{-,0}}  |dz_2| e^{2B|w_1| -|w_1|^3/12}  \\
& + D_1 \cdot \delta^{-1} \cdot \exp(-135t \delta^2) \cdot |\hat{\Phi}_i| e^{a_i^{-3}/3} \cdot \int_{\Gamma_{u_1}^-} |dw_1|  e^{2B|w_1| - |w_1|^3/12}.
\end{split}
\end{equation*}
As the last integral is finite, we conclude
\begin{equation}\label{Eq.SMB12iVanish}
\lim_{t \rightarrow \infty}B_{i}^{1,2,t}(\tilde{\alpha}, \tilde{\beta}) = 0.
\end{equation}
Combining the first line in (\ref{Eq.SMResiV2}), with (\ref{Eq.SMB11iVanish}) and (\ref{Eq.SMB12iVanish}), we conclude (\ref{Eq.SMDecBi}).\\

{\bf \raggedleft Step 5.} In this step we fix $j \in \{1, \dots, q\}$ and prove (\ref{Eq.SMDecBj}). As the proof is quite similar to that in Step 3, we will be brief. 

As in Step 3, we deform the $w_1,w_2$ contours $\Gamma_0^-$ in $B_{j}^{2,t}(\tilde{\alpha}, \tilde{\beta})$ to $\Gamma_{u_1}^-, \Gamma_{u_2}^-$, respectively. By (\ref{Eq.SMDeltaEpsilon}) and (\ref{Eq.SMKillPole}), we do not cross any poles in the process. By Cauchy's theorem,
\begin{equation}\label{Eq.SMBjV2}
\begin{split}
&B_{j}^{2,t}(\tilde{\alpha}, \tilde{\beta}) = \frac{1}{(2\pi \im)^3} \int_{\Gamma_{u_1}^-} dw_1 \int_{\Gamma_{v_1}^+} dz_1 \int_{\Gamma_{u_2}^-} dw_2 \frac{\hat{\Phi}_j \Phi(z_1)}{\Phi(w_1) \Phi(w_2)}  \\
 & \times \frac{e^{- (z_1-w_2)t\tilde{\alpha}- (1/a_j-w_1)t\tilde{\beta}} \cdot e^{z_1^3/3 - w_1^3/3-z_1^2t + w_1^2 t} \cdot e^{a_j^{-3}/3 - w_2^3/3-a_j^{-2}t + w_2^2 t} }{(1/a_j - w_1)(1/a_j - w_2)(z_1 - w_1)(z_1 - w_2)}.
\end{split} 
\end{equation}

Setting $w_1 = u_1 + e^{\im \psi_1}r_{w_1}$ as in (\ref{Eq.SMParametrization}), we obtain the following analogue of (\ref{Eq.SMBiBoundW2}):
\begin{equation}\label{Eq.SMBjBoundW1}
\begin{split}
&\left| e^{-(1/a_j - w_1) t \tilde{\beta} - a_j^{-2}t + w_1^2t} \right| = \exp\left(t (u_1 + \tilde{\beta}/2)^2 - t (\tilde{\beta}/2 + 1/a_j)^2 - \sqrt{2} t r_{w_1} (u_1 + \tilde{\beta}/2) \right) \\
& \leq \exp\left( t\delta^2 - t (\tilde{\beta}/2 + 1/a_j)^2 \right) \leq \exp\left( - 143 t \delta^2 \right),
\end{split}
\end{equation}
where we used $u_1 + \tilde{\beta}/2 = \delta$ from (\ref{Eq.SMDeltaEpsilon}), and (\ref{Eq.SMWellSep}).

Combining (\ref{Eq.SMBjV2}) with the second inequality in (\ref{Eq.CubicEstimate}) with $w$ replaced with $w_1$, (\ref{Eq.SMWellSep}), the first inequality in (\ref{Eq.SMQuadraticBound}), (\ref{Eq.SMBigBound1}), (\ref{Eq.SMKillPoleBound}) with $w$ replaced with $w_1$, and (\ref{Eq.SMBjBoundW1}), we conclude
\begin{equation*}
\begin{split}
&\left|B_{j}^{2,t}(\tilde{\alpha}, \tilde{\beta})\right| \leq D_1 \cdot \delta^{-3} \cdot \exp\left(-27t\delta^2 \right) \cdot D_1^2 D_2 \cdot a_j \cdot \exp\left( - 143 t \delta^2 \right) \cdot |\hat{\Phi}_j| e^{a_j^{-3}/3} \cdot \\
&\times \int_{\Gamma_{u_1}^-} |dw_1| \int_{\Gamma_{v_1}^+} |dz_1| \int_{\Gamma_{u_2}^-} |dw_2| e^{B |z_1| + B|z_1|/d_1 + 2B|w_2| + 2B|w_1| - |w_1|^3/12 - |w_2|^3/12 - |z_1|^3/12 }.
\end{split}
\end{equation*}
The last inequality implies (\ref{Eq.SMDecBj}).\\

{\bf \raggedleft Step 6.} In this step we fix $i \in \{1, \dots, p\}$ and $j \in \{1, \dots, q\}$ with $i \neq j$, and prove (\ref{Eq.SMDecResij}). 

As in Step 3, we deform the $w_1,w_2$ contours $\Gamma_0^-$ in $B_{ij}^{t}(\tilde{\alpha}, \tilde{\beta})$ to $\Gamma_{u_1}^-, \Gamma_{u_2}^-$, respectively. By (\ref{Eq.SMKillPole}), we do not cross any poles in the process, and so, by Cauchy's theorem,
\begin{equation}\label{Eq.SMBijV2}
\begin{split}
&B_{ij}^t(\tilde{\alpha}, \tilde{\beta}) = \frac{1}{(2\pi \im)^2} \int_{\Gamma_{u_1}^-} dw_1  \int_{\Gamma_{u_2}^-} dw_2 \frac{\hat{\Phi}_i\hat{\Phi}_j}{\Phi(w_1) \Phi(w_2)}  \\
 & \times \frac{e^{- (1/a_i-w_2)t\tilde{\alpha}- (1/a_j-w_1)t\tilde{\beta}} \cdot e^{a_i^{-3}/3 - w_1^3/3-a_{i}^{-2}t + w_1^2 t} \cdot e^{a_j^{-3}/3 - w_2^3/3-a_{j}^{-2}t + w_2^2 t} }{(1/a_i - w_1)(1/a_j - w_1)(1/a_i - w_2)(1/a_j - w_2)}.
\end{split} 
\end{equation} 

Combining (\ref{Eq.SMBijV2}) with the second inequality in (\ref{Eq.CubicEstimate}) with $w$ replaced with $w_1$ or $w_2$, (\ref{Eq.SMWellSep}), (\ref{Eq.SMKillPoleBound}) with $w$ replaced with $w_1$ or $w_2$, (\ref{Eq.SMBiBoundW2}) and (\ref{Eq.SMBjBoundW1}), we conclude
 \begin{equation*}
\begin{split}
&\left|B_{ij}^t(\tilde{\alpha}, \tilde{\beta})\right| \leq D_1^2 \cdot \delta^{-2} \cdot a_ia_j \cdot \exp\left(-135t\delta^2\right) \cdot \exp\left(-143 t \delta^2\right)\\
&\times \int_{\Gamma_{u_1}^-} |dw_1|  \int_{\Gamma_{u_2}^-} |dw_2| e^{2B|w_1| + 2B|w_2| - |w_1|^3/12 - |w_2|^3/12}.
\end{split} 
\end{equation*} 
The last inequality implies (\ref{Eq.SMDecResij}).\\

{\bf \raggedleft Step 7.} In this step we fix $j \in \{1, \dots, q\}$, and prove (\ref{Eq.SMDecResii}). 

As in Step 3, we deform the $w_1,w_2$ contours $\Gamma_0^-$ in $B_{jj}^{t}(\tilde{\alpha}, \tilde{\beta})$ to $\Gamma_{u_1}^-, \Gamma_{u_2}^-$, respectively. By (\ref{Eq.SMKillPole}), in the process of deformation we cross the simple poles at $w_1 = 1/a_j$ and $w_2 = 1/a_j$. By the residue theorem, we conclude
\begin{equation}\label{Eq.SMBjjV2}
B_{jj}^t(\tilde{\alpha}, \tilde{\beta}) = 1 + B_{jj}^{1,t}(\tilde{\alpha}, \tilde{\beta}) + B_{jj}^{2,t}(\tilde{\alpha}, \tilde{\beta}) + B_{jj}^{3,t}(\tilde{\alpha}, \tilde{\beta}),
\end{equation}
where 
\begin{equation*}
\begin{split}
&B_{jj}^{1,t}(\tilde{\alpha}, \tilde{\beta}) = \frac{1}{2\pi \im}  \int_{\Gamma_{u_2}^-} dw_2 \frac{\hat{\Phi}_j}{ \Phi(w_2)} \cdot \frac{e^{- (1/a_j-w_2)t\tilde{\alpha}} \cdot e^{a_j^{-3}/3 - w_2^3/3-a_{j}^{-2}t + w_2^2 t} }{(1/a_j - w_2)^2},
\end{split}
\end{equation*}
\begin{equation*}
\begin{split}
&B_{jj}^{2,t}(\tilde{\alpha}, \tilde{\beta}) = \frac{1}{2\pi \im}  \int_{\Gamma_{u_1}^-} dw_1 \frac{\hat{\Phi}_j}{ \Phi(w_1)} \cdot \frac{e^{- (1/a_j-w_1)t\tilde{\beta}} \cdot e^{a_j^{-3}/3 - w_1^3/3-a_{j}^{-2}t + w_1^2 t} }{(1/a_j - w_1)^2},
\end{split}
\end{equation*}
\begin{equation*}
\begin{split}
&B_{jj}^{3,t}(\tilde{\alpha}, \tilde{\beta}) = \frac{1}{(2\pi \im)^2} \int_{\Gamma_{u_1}^-} dw_1  \int_{\Gamma_{u_2}^-} dw_2 \frac{\hat{\Phi}_j^2}{\Phi(w_1) \Phi(w_2)}  \\
 & \times \frac{e^{- (1/a_j-w_2)t\tilde{\alpha}- (1/a_j-w_1)t\tilde{\beta}} \cdot e^{a_j^{-3}/3 - w_1^3/3-a_{j}^{-2}t + w_1^2 t} \cdot e^{a_j^{-3}/3 - w_2^3/3-a_{j}^{-2}t + w_2^2 t} }{(1/a_j - w_1)^2(1/a_j - w_2)^2}.
\end{split} 
\end{equation*}

Combining the last three displayed equations with the second inequality in (\ref{Eq.CubicEstimate}) with $w$ replaced with $w_1$ or $w_2$, (\ref{Eq.SMWellSep}), (\ref{Eq.SMKillPoleBound}) with $w$ replaced with $w_1$ or $w_2$, (\ref{Eq.SMBiBoundW2}) and (\ref{Eq.SMBjBoundW1}), we conclude
\begin{equation*}
\begin{split}
&\left|B_{jj}^{1,t}(\tilde{\alpha}, \tilde{\beta}) \right| \leq D_1 \cdot \delta^{-1} \cdot a_j \cdot \exp\left(-135t\delta^2 \right) \cdot |\hat{\Phi}_j| e^{a_j^{-3}/3} \cdot \int_{\Gamma_{u_2}^-} |dw_2| e^{2B|w_2| - |w_2|^3/12},
\end{split}
\end{equation*}
\begin{equation*}
\begin{split}
&\left|B_{jj}^{2,t}(\tilde{\alpha}, \tilde{\beta}) \right| \leq D_1 \cdot \delta^{-1} \cdot a_j \cdot \exp\left(-143t\delta^2 \right) \cdot |\hat{\Phi}_j| e^{a_j^{-3}/3} \cdot \int_{\Gamma_{u_1}^-} |dw_1| e^{2B|w_1| - |w_1|^3/12},
\end{split}
\end{equation*}
\begin{equation*}
\begin{split}
&\left|B_{jj}^{3,t}(\tilde{\alpha}, \tilde{\beta})\right| \leq D_1^2 \cdot \delta^{-2} \cdot a_j^2 \cdot \exp\left(-278t\delta^2 \right) \cdot |\hat{\Phi}_j|^2 e^{2a_j^{-3}/3} \\
& \times \int_{\Gamma_{u_1}^-} |dw_1|\int_{\Gamma_{u_2}^-} |dw_2| e^{2B|w_1| + 2B|w_2| - |w_1|^3/12 - |w_2|^3/12}.
\end{split} 
\end{equation*}
The last three displayed equations and (\ref{Eq.SMBjjV2}) imply (\ref{Eq.SMDecResii}).
\end{proof}

%
%
\subsection{First moments of $\tilde{M}^t$}\label{Section4.3} The goal of this section is to establish the following key lemma.
\begin{lemma}\label{Lem.FirstMomentFlat} Assume the same notation as in Definition \ref{Def.SlopesPar}, and suppose $J_a < \infty$. For $\tilde{M}^t$ as in (\ref{Eq.MeasT}), we have
\begin{equation}\label{Eq.FirstMomentFlat}
\lim_{\hat{\alpha} \rightarrow \infty} \limsup_{t \rightarrow \infty} \left|\mathbb{E}\left[\tilde{M}^t[\hat{\alpha}, \infty)  \right] - J_a\right| = 0.
\end{equation}
\end{lemma}
\begin{proof} The proof we present is similar to that of Lemma \ref{Lem.FirstMoment}, and so we will be brief. For clarity, we split the proof into three steps. In the first step, we derive a double-contour integral formula for $\mathbb{E}\left[\tilde{M}^t[\hat{\alpha}, \infty)  \right]$, see (\ref{Eq.FMFE1}). In the second step, we deform the contours from (\ref{Eq.FMFE1}), and reduce the proof of the lemma to establishing that certain single or double integrals converge to zero as $t \rightarrow \infty$, see (\ref{Eq.FMFE3}) and (\ref{Eq.FMFE4}), which we establish in the third step.\\

{\bf \raggedleft Step 1.} From Proposition \ref{S12AWD} and \cite[Lemma 2.17]{dimitrov2024airy}, we know that $\tilde{M}^{t}$ is a determinantal point process on $\mathbb{R}$ with correlation kernel $\tilde{K}^t(x,y) = K_{a,b,c}(t,x;t,y)$ and Lebesgue reference measure. Using Definition \ref{3BPKernelDef}, and changing variables $z \mapsto z - t$ and $w \mapsto w - t$, we can write $\tilde{K}^t(x,y)$ as 
\begin{equation}\label{Eq.KernelTF}
 \tilde{K}^t(x,y) = \frac{1}{(2\pi \im)^2} \int_{\Gamma_{\alpha }^+}  d z \int_{\Gamma_{0}^-}  dw \frac{e^{(z-t)^3/3 - (w-t)^3/3 - (z-t)x+ (w-t)y}}{z - w } \cdot \frac{\Phi(z)}{\Phi(w)},
\end{equation}
where $\Gamma_{\alpha}^+, \Gamma_{0}^-$ are as in Definition \ref{DefContInf} with $0 < \alpha < 1/a_1$, and $\Phi(z) = \prod_{i \geq 1}^{\infty} \frac{1 + b_i z}{1 - a_i z}$.

From \cite[(2.13) and (2.18)]{dimitrov2024airy} 
\begin{equation*}
\begin{split}
 \mathbb{E}\left[\tilde{M}^{t}[\hat{\alpha}, \infty)  \right] = \int_{\hat{\alpha}}^{\infty}\tilde{K}^t(x,x)dx. 
 \end{split} 
\end{equation*}
Substituting $\tilde{K}^t(x,x)$ from (\ref{Eq.KernelTF}), exchanging the order of the integrals, and performing the integral over $x$, we obtain
\begin{equation}\label{Eq.FMFE1}
\begin{split}
 \mathbb{E}\left[\tilde{M}^{t}[\hat{\alpha}, \infty)  \right] = \frac{1}{(2\pi \im)^2} \int_{\Gamma_{\alpha }^+}  d z \int_{\Gamma_{0}^-}  dw \frac{e^{(z-t)^3/3 - (w-t)^3/3 - (z-w)\hat{\alpha}}}{(z - w)^2 } \cdot \frac{\Phi(z)}{\Phi(w)}.
 \end{split} 
\end{equation}
The exchange of the order of integration is justified by the same estimates stated below (\ref{Eq.FME1}).\\

{\bf \raggedleft Step 2.} Suppose $t > 1 + 1/a_{J_a}$, and put $ u = t$, $v = t + 1$.

We now proceed to deform $\Gamma_{\alpha}^+$ to $\Gamma_{v}^+$ in (\ref{Eq.FMFE1}). In the process of deformation we cross the {\em simple} poles at $z = 1/a_j$ for $j = 1, \dots, J_a$ (coming from $\Phi(z)$). By the residue theorem, 
\begin{equation}\label{Eq.FMFE2}
\begin{split}
&\mathbb{E}\left[\tilde{M}^{t}[\hat{\alpha}, \infty)  \right] = \sum_{j = 1}^{J_a} \frac{1}{2\pi \im}  \int_{\Gamma_{0}^-} dw \frac{e^{(1/a_j-t)^3/3 - (w- t)^3/3 - (1/a_j-w)\hat{\alpha}}}{(1/a_j - w)^2 } \cdot \frac{\hat{\Phi}_j}{\Phi(w)} \\
 & + \frac{1}{(2\pi \im)^2} \int_{\Gamma_{v}^+}  d z \int_{\Gamma_{0}^-}  dw \frac{e^{(z-t)^3/3 - (w-t)^3/3 - (z-w)\hat{\alpha}}}{(z - w)^2 } \cdot \frac{\Phi(z)}{\Phi(w)},
 \end{split} 
\end{equation}
where $\hat{\Phi}_j = \prod_{i = 1, i \neq j}^{\infty} \frac{(1 + b_i /a_j)}{(1 - a_i/a_j)} \cdot \frac{(1 + b_j/a_j)}{a_j}$.

We next deform $\Gamma_{0}^-$ to $\Gamma_{u}^-$ in (\ref{Eq.FMFE2}). In the process of deformation we cross the simple poles at $w = 1/a_j$ for $j = 1, \dots, J_a$ within the first line, and do not cross any poles for the second. Applying the residue theorem, and the change of variables $z \mapsto z+t$, $w \mapsto w+t$, we conclude
\begin{equation}\label{Eq.FMFExpanded}
\mathbb{E}\left[\tilde{M}^{t}[\hat{\alpha}, \infty)  \right] = J_a + \tilde{A}^t + \sum_{j = 1}^{J_a} \tilde{A}^{j,t},
\end{equation}
where
\begin{equation}\label{Eq.FMFAT}
\tilde{A}^t = \frac{1}{(2\pi \im)^2} \int_{\Gamma_{1}^+}  d z \int_{\Gamma_{0}^-}  dw \frac{e^{z^3/3 - w^3/3 - (z-w)\hat{\alpha}}}{(z - w)^2 } \cdot \frac{\Phi(z+t)}{\Phi(w+t)},
\end{equation}
and for $j = 1,\dots, J_a$
\begin{equation}\label{Eq.FMFAJT}
\tilde{A}^{j,t} =  \frac{1}{2\pi \im}  \int_{\Gamma_{0}^-} dw \frac{e^{(1/a_j - t)^3/3  - w^3/3 - (1/a_j-w-t)\hat{\alpha}}}{(1/a_j - w - t)^2 } \cdot \frac{\hat{\Phi}_j}{\Phi(w+t)}.
\end{equation}

From (\ref{Eq.FMFExpanded}), we see that to show (\ref{Eq.FirstMomentFlat}), it suffices to prove
\begin{equation}\label{Eq.FMFE3}
\begin{split}
&\lim_{\hat{\alpha} \rightarrow \infty} \limsup_{t \rightarrow \infty} |\tilde{A}^t| = 0,
 \end{split} 
\end{equation}
and for $j = 1, \dots, J_a$ that
\begin{equation}\label{Eq.FMFE4}
\begin{split}
&\lim_{\hat{\alpha} \rightarrow \infty} \limsup_{t \rightarrow \infty} |\tilde{A}^{j,t}| = 0.
 \end{split} 
\end{equation}

{\bf \raggedleft Step 3.} In this step, we prove (\ref{Eq.FMFE3}) and (\ref{Eq.FMFE4}).

Fix $j \in \{1, \dots, J_a\}$. From (\ref{RatBound}) with $d = 1/2$, we have for $w \in \Gamma_0^-$ and $t \geq 0$
\begin{equation}\label{Eq.FMFB1}
|\Phi(w+t)^{-1}|\leq e^{2B|w| + 2Bt},
\end{equation}
where $B = \sum_{i \geq 1} (a_i + b_i)$. Also, if $t > 2 + 1/a_{J_a}$ and $\hat{\alpha} \geq 0$, we have the trivial bounds
\begin{equation}\label{Eq.FMFB2}
|1/a_j -w -t|^{-1} \leq 1 \mbox{, and } |e^{w\hat{\alpha}}| \leq 1 \mbox{ for }w \in \Gamma_{0}^-.
\end{equation}

Combining (\ref{Eq.FMFAJT}) with (\ref{Eq.CubicEstimate}), (\ref{Eq.FMFB1}) and (\ref{Eq.FMFB2}), we get for $t > 2 + 1/a_{J_a}$ and $\hat{\alpha} \geq 0$
\begin{equation*}
|\tilde{A}^{j,t} | \leq  D_1 \cdot |\hat{\Phi}_j| \cdot e^{(1/a_j - t)^3/3+ (t-1/a_j)\hat{\alpha} + 2Bt} \int_{\Gamma_{0}^-}|dw| e^{2B|w|- |w|^3/12 },
\end{equation*}
which implies (\ref{Eq.FMFE4}).\\

We next observe for $t, b > 0$, $z \in \Gamma_1^+$, $w \in \Gamma_0^-$, that
$$\frac{|1+b(z+t)|}{|1+b(w+t)|} = \left|\frac{1 + \frac{bz}{1+bt}}{1 + \frac{bw}{1 + bt}}\right| \leq \exp\left(\frac{b|z|}{1+bt} + \frac{2b|w|}{1+bt} \right) \leq \exp(2b|z|+2b|w|),$$
where in the next to last inequality we used (\ref{RatBound}) with $d = 1/2$. In addition, if $J_a > 0$, $a \in [a_{J_a}, a_1]$, $t > 2/a_{J_a}$, $z \in \Gamma_1^+$, $w \in \Gamma_0^-$
$$\frac{|1-a(w+t)|}{|1-a(z+t)|} = \left|\frac{1 + \frac{aw}{at-1}}{1 + \frac{az}{at-1}} \right| \leq \left|1 + \frac{aw}{at-1}\right| \leq \exp\left(\frac{a|w|}{at-1} \right) \leq \exp(a|w|).$$

The last two equations give for $t > 2/a_{J_a}$, $z \in \Gamma_1^+$, $w \in \Gamma_0^-$ 
\begin{equation}\label{Eq.FMFBound1}
\frac{|\Phi(z+t)|}{|\Phi(w+t)|} \leq \exp\left(2B|z| + 2B|w| \right).
\end{equation}
In addition, using that $\Real(z-w) \geq 1$ for $z \in \Gamma_1^+$, $w \in \Gamma_0^-$, we conclude for $\hat{\alpha} \geq 0$
\begin{equation}\label{Eq.FMFBound2}
\left|e^{-(z-w)\hat{\alpha}}\right| \leq e^{-\hat{\alpha}}, \hspace{2mm} |z-w|^{-1} \leq 1.
\end{equation}

Combining (\ref{Eq.FMFAT}) with (\ref{Eq.CubicEstimate}), (\ref{Eq.FMFBound1}) and (\ref{Eq.FMFBound2}), we get for $\hat{\alpha} \geq 0$, and $t > 2/a_{J_a}$
\begin{equation*}
|\tilde{A}^{t} | \leq  D_1^2 \cdot e^{-\hat{\alpha}}\int_{\Gamma_{1}^+}  |d z| \int_{\Gamma_{0}^-}  |dw| e^{2B|z| + 2B|w| -|z|^3/12 - |w|^3/12 },
\end{equation*}
which implies (\ref{Eq.FMFE3}).
\end{proof}

%
%
\section{Asymptotic slopes}\label{Section5} The goal of this section is to prove Theorem \ref{Thm.Slopes}. Throughout this section we continue with the same notation as in the statement of the theorem.

%
%
\subsection{Proof of Theorem \ref{Thm.Slopes}(a,b)}\label{Section5.1} Observe that part (b) follows from part (a) and Proposition \ref{S13P1}(b). Consequently, it suffices to prove part (a). We split the proof into two steps.\\

{\bf \raggedleft Step 1.} In this step, we suppose $a_i^+ > a^{+}_{i+1}$ for $i \in \{1, \dots, J_a^+\}$.

Recall $M^{t_N}$ from (\ref{Eq.MeasT}):
\begin{equation*}
M^{t_N}(A) = \sum_{i \geq 1} {\bf 1}\left\{ t_N^{-1} \cdot \left(\mathcal{A}^{a,b,c}_i(t_N) - t_N^2 \right)  \in A \right\}.
\end{equation*}
Fix $k \in \mathbb{N}$ with $k \leq J_a^+$, and $\varepsilon > 0$. We can find $\rho \in (0, \varepsilon)$ sufficiently small so that if we set
$$\tilde{\alpha}_i = -2/a_i^+ -\rho \mbox{ and } \tilde{\beta}_i = -2/a_i^+ + \rho, \mbox{ for }i = 1, \dots, k,$$
we have 
$$-2/a^+_{k+1} < \tilde{\alpha}_k < -2/a_k^+ < \tilde{\beta}_k < \tilde{\alpha}_{k-1} < -2/a_{k-1}^+ < \tilde{\beta}_{k-1} < \cdots < \tilde{\alpha}_{1} < -2/a_{1}^+ < \tilde{\beta}_{1} < 0,$$
where, as usual, $-2/0 = -\infty$. We further set for $i = 1, \dots, k$ 
$$X^N_i = M^{t_N}[\tilde{\alpha}_i, \infty), \hspace{2mm} Y^N_i = M^{t_N}[\tilde{\beta}_i, \infty), \hspace{2mm} Z^N_i = M^{t_N}[\tilde{\alpha}_i, \tilde{\beta}_i).$$

In (\ref{Eq.FME1}) we showed that $X_i^N, Y_i^N$ have finite means, which implies $X_i^N, Y_i^N, Z_i^N \in \mathbb{Z}_{\geq 0}$ almost surely and $Z_i^N = X_i^N - Y_i^N$. From Lemma \ref{Lem.FirstMoment}, we have for $i = 1, \dots, k$ 
\begin{equation}\label{Eq.MeansXConv}
\lim_{N \rightarrow \infty} \mathbb{E}\left[X_i^N \right] = i, \hspace{2mm} \lim_{N \rightarrow \infty} \mathbb{E}\left[Y_i^N \right] = i-1.
\end{equation}
The last equation and Lemma \ref{Lem.SecondMoment} give
\begin{equation}\label{Eq.SecondZConv}
\lim_{N \rightarrow \infty} \mathbb{E}\left[Z_i^N \right] = 1, \hspace{2mm} \lim_{N \rightarrow \infty} \mathbb{E}\left[Z_i^N (Z_i^N - 1) \right] = 0.
\end{equation}
The last equation implies $Z_i^N \Rightarrow 1$ for $i = 1, \dots, k$. 

Define $W_1^N = Y_1^N$, and for $i = 2, \dots, k$ set $W_i^N = M^{t_N}[\tilde{\beta}_i, \tilde{\alpha}_{i-1}) = Y_i^{N} - X_{i-1}^N$. From (\ref{Eq.MeansXConv}), we conclude $W_i^N \Rightarrow 0$ for $i = 1, \dots, k$. Since $Z_i^N, W_i^N \in \mathbb{Z}_{\geq 0}$ (which is a discrete set), we conclude that 
\begin{equation}\label{Eq.LimProbabSlopes}
\lim_{N \rightarrow \infty} \mathbb{P}(Z_i^N = 1, \hspace{2mm} W_i^N = 0 \mbox{ for } i =1 ,\dots, k) = 1.
\end{equation} 

We now observe that we have the following equality of events
$$\{ Z_i^N = 1, \hspace{2mm} W_i^N = 0 \mbox{ for } i =1 ,\dots, k \} = \left\{\tilde{\alpha}_i \leq t_N^{-1} \cdot (\mathcal{A}^{a,b,c}_i(t_N) - t_N^2) < \tilde{\beta}_i \mbox{ for } i =1 ,\dots, k \right\}.$$
Combining the last displayed equation with (\ref{Eq.LimProbabSlopes}), and the definitions of $\tilde{\alpha}_i, \tilde{\beta}_i$, we conclude 
$$\lim_{N \rightarrow \infty} \mathbb{P} \left( \left|t_N^{-1} \cdot (\mathcal{A}^{a,b,c}_i(t_N) - t_N^2) + 2/a_i^+ \right| > \rho \mbox{ for some $i\in \{1, \dots, k\}$}\right) = 0,$$
which implies the statement of the theorem as $\rho \in (0, \varepsilon)$ and $\varepsilon > 0$ was arbitrary.\\

{\bf \raggedleft Step 2.} In this step we fix an arbitrary $(a,b,c) \in \parP$ and $k \in \mathbb{N}$ with $k \leq J_a^+$. Note that in particular $J_a^+ \geq 1$. We seek to show that for any fixed $\varepsilon > 0$, we have
\begin{equation}\label{Eq.ThmSlopesRed}
\begin{split}
&\lim_{N \rightarrow \infty} \mathbb{P}\left( t_N^{-1} \cdot (\mathcal{A}^{a,b,c}_k(t_N) - t_N^2) + 2/a_k^+ > \varepsilon \right) = 0, \mbox{ and }\\
& \lim_{N \rightarrow \infty} \mathbb{P}\left( t_N^{-1} \cdot (\mathcal{A}^{a,b,c}_k(t_N) - t_N^2) + 2/a_k^+ < -\varepsilon \right) = 0.
\end{split}
\end{equation} 

Fix $\delta > 0$ sufficiently small so that 
$$2(e^{\delta k} - 1)/a_k^+ < \varepsilon/2, \mbox{ and }2/a_k^+ - 2/(2^{-k} \delta +a_k^+) < \varepsilon/2.$$ 
We define the auxiliary sequences
$$\hat{a}^+_i = e^{-\delta i} \cdot a_i^+, \hspace{2mm} \check{a}^+_i = 2^{-i}\delta + a_i^+ \mbox{ for } i \geq 1,$$
and observe that $J_{\hat{a}}^+ = J_a^+$, $J_{\check{a}}^+ = \infty$, and
$$\check{a}^+_i > \check{a}^+_{i+1} \mbox{ for } i \geq 1, \mbox{ and } \hat{a}^+_i > \hat{a}^+_{i+1} \mbox{ for } i = 1, \dots, J_{\hat{a}}^+.$$

From our work in Step 1, we conclude 
\begin{equation}\label{Eq.ThmSlopesAux}
\begin{split}
&\lim_{N \rightarrow \infty} \mathbb{P}\left( t_N^{-1} \cdot (\mathcal{A}^{\check{a},b,c}_k(t_N) - t_N^2) + 2/\check{a}_k^+ > \varepsilon/2 \right) = 0, \mbox{ and }\\
& \lim_{N \rightarrow \infty} \mathbb{P}\left( t_N^{-1} \cdot (\mathcal{A}^{\hat{a},b,c}_k(t_N) - t_N^2) + 2/\hat{a}_k^+ < -\varepsilon/2 \right) = 0.
\end{split}
\end{equation} 

Using that $\check{a}_i > a_i^+$ for $i \geq 1$, and Theorem \ref{Thm.MonCoupling} with $A = B = 0$, we conclude
\begin{equation*}
\begin{split}
&\limsup_{N \rightarrow \infty} \mathbb{P}\left( t_N^{-1} \cdot (\mathcal{A}^{a,b,c}_k(t_N) - t_N^2) + 2/\check{a}_k^+ > \varepsilon/2 \right) \\
& \leq \limsup_{N \rightarrow \infty} \mathbb{P}\left( t_N^{-1} \cdot (\mathcal{A}^{\check{a},b,c}_k(t_N) - t_N^2) + 2/\check{a}_k^+ > \varepsilon/2 \right) = 0,
\end{split}
\end{equation*}
where in the last equality we used the first line in (\ref{Eq.ThmSlopesAux}). The last equation implies the first line in (\ref{Eq.ThmSlopesRed}), once we note that $2/\check{a}_k^+ > 2/a_k^+ - \varepsilon/2$ by construction.

Similarly, using that $\hat{a}_i \leq a_i^+$ for $i \geq 1$, and Theorem \ref{Thm.MonCoupling} with $A = B = 0$, we conclude
\begin{equation*}
\begin{split}
&\limsup_{N \rightarrow \infty} \mathbb{P}\left( t_N^{-1} \cdot (\mathcal{A}^{a,b,c}_k(t_N) - t_N^2) + 2/\hat{a}_k^+ < - \varepsilon/2 \right) \\
& \leq \limsup_{N \rightarrow \infty} \mathbb{P}\left( t_N^{-1} \cdot (\mathcal{A}^{\hat{a},b,c}_k(t_N) - t_N^2) + 2/\hat{a}_k^+ <- \varepsilon/2 \right) = 0,
\end{split}
\end{equation*}
where in the last equality we used the second line in (\ref{Eq.ThmSlopesAux}). The last equation implies the second line in (\ref{Eq.ThmSlopesRed}), once we note that $2/\hat{a}_k^+ < 2/a_k^+ + \varepsilon/2$ by construction.

%
%
\subsection{Proof of Theorem \ref{Thm.Slopes}(c,d)}\label{Section5.2} Observe that part (d) follows from part (c) and Proposition \ref{S13P1}(b). Consequently, it suffices to prove part (c). We split the proof into two steps.\\

{\bf \raggedleft Step 1.} In this step we show the following tightness-from-above statement 
\begin{equation}\label{Eq.ThmSlopeCDRed1}
\lim_{h \rightarrow \infty }\limsup_{N \rightarrow \infty} \mathbb{P}\left( \mathcal{A}^{a,b,c}_k(t_N) \geq h  \right) = 0.
\end{equation}

Set $\check{a}_i^+ = 2^{-i} + a_i^+$ for $i = 1, \dots, J_a^+$ and $\check{a}_i^+ = 0$ for $i \geq J_a^+ + 1$. We observe that $J^+_{\check{a}} = J^+_a$, $\check{a}_i^+ > \check{a}_{i+1}^+$ for $i = 1, \dots, J_{\check{a}}^+$, and $\check{a}_i^+ \geq a_i^+$ for $i \geq 1$. From Theorem \ref{Thm.MonCoupling}, we see that to prove (\ref{Eq.ThmSlopeCDRed1}), it suffices to show 
\begin{equation}\label{Eq.ThmSlopeCDRed2}
\lim_{h \rightarrow \infty }\limsup_{N \rightarrow \infty} \mathbb{P}\left(E_{h,N} \right) = 0, \mbox{ where } E_{h,N} =  \{\mathcal{A}^{\check{a},b,c}_k(t_N) \geq h \}.
\end{equation}

Fix $\varepsilon > 0$. Let $F_{h,N} = \{\mathcal{A}^{\check{a},b,c}_i(t_N) \geq h \mbox{ for } i =1, \dots, J_{\check{a}}^+\}$, and observe that from Theorem \ref{Thm.Slopes}(a), which we proved in Section \ref{Section5.1} above, we have for each $h \in \mathbb{R}$
\begin{equation}\label{Eq.ThmSlopeCD1}
\lim_{N \rightarrow \infty}\mathbb{P}(F_{h,N}) = 1.
\end{equation}

Recall $\tilde{M}^{t_N}$ from (\ref{Eq.MeasT}):
\begin{equation*}
 \tilde{M}^{t_N}(A) = \sum_{i \geq 1} {\bf 1}\left\{ \mathcal{A}^{\check{a},b,c}_i(t_N)  \in A \right\}.
\end{equation*}
From Lemma \ref{Lem.FirstMomentFlat}, we can find $H > 0$, such that for $h \geq H$, we have 
\begin{equation*}
\limsup_{N \rightarrow \infty} \mathbb{E}\left[\tilde{M}^{t_N}[h, \infty)  \right] \leq J_{\check{a}}^+ + \varepsilon.
\end{equation*}
Using that $k \geq J_{\check{a}}^+ + 1$ by assumption, we conclude
$$\tilde{M}^{t_N}[h, \infty) \geq (J_{\check{a}}^+ + 1) \cdot {\bf 1}_{E_{h,N}} + J_{\check{a}}^+ \cdot {\bf 1}_{F_{h,N} \cap E_{h,N}^c} \geq (J_{\check{a}}^+ + 1)  \cdot {\bf 1}_{E_{h,N}} + J_{\check{a}}^+ \cdot {\bf 1}_{E_{h,N}^c} - J_{\check{a}}^+ \cdot {\bf 1}_{F_{h,N}^c}.$$
The last two inequalities and (\ref{Eq.ThmSlopeCD1}) show for $h \geq H$
$$\limsup_{N \rightarrow \infty} \mathbb{P}(E_{h,N}) \leq \varepsilon,$$
which establishes (\ref{Eq.ThmSlopeCDRed2}).\\

{\bf \raggedleft Step 2.} In this step we show the following tightness-from-below statement 
\begin{equation}\label{Eq.ThmSlopeCDRed3}
\lim_{h \rightarrow \infty }\limsup_{N \rightarrow \infty} \mathbb{P}\left( \mathcal{A}^{a,b,c}_k(t_N) \leq - h  \right) = 0,
\end{equation}
which together with (\ref{Eq.ThmSlopeCDRed1}) completes the proof of the theorem.

From Theorem \ref{Thm.MonCoupling} with $A = B= 0$, we have
\begin{equation*}
\mathbb{P}\left( \mathcal{A}^{a,b,c}_k(t_N) \leq - h  \right) \leq \mathbb{P}\left( \mathcal{A}^{0,0,0}_k(t_N) \leq - h  \right)  = \mathbb{P}\left( \mathcal{A}^{0,0,0}_k(0) \leq - h  \right),
\end{equation*}
where $\mathcal{A}^{0,0,0}$ is the usual Airy line ensemble, cf. Remark \ref{Rem.AWLE2}, and in the last equality we used its stationarity. The last inequality immediately implies (\ref{Eq.ThmSlopeCDRed3}).

%
%
\section{Extremality}\label{Section6} The goal of this section is to prove Theorem \ref{Thm.Extreme}. In Section \ref{Section6.1} we formally introduce the Brownian Gibbs property, and recall some known results from the literature. In Section \ref{Section6.2} we perform a sequence of steps that reduce the proof of Theorem \ref{Thm.Extreme} to establishing the existence of {\em approximate} monotone couplings between certain ensembles related to $\mathcal{L}^{a,b,c}$ and ensembles $\mathcal{L}^{\hat{a},\hat{b},0}$ with finitely many non-zero parameters. In the same section we give a general overview of how this approximate monotone coupling is established, with the full argument presented in Section \ref{Section6.3}.

%
%
\subsection{Preliminary results}\label{Section6.1} The goal of this section is to define the Brownian Gibbs property, originally introduced in \cite{CorHamA}. As fairly detailed expositions of this property have already been given, we will be brief. We refer to \cite[Section 2]{DEA21} and \cite[Section 2]{DimMat} for additional details.

If $W_t$ is a standard one-dimensional Brownian motion, then the process
$$\tilde{B}(t) =  W_t - t W_1, \hspace{5mm} 0 \leq t \leq 1,$$
is called a {\em standard Brownian bridge}. Given $a,b,x,y \in \mathbb{R}$ with $a < b$, we define
\begin{equation}\label{BBDef}
B(t) = (b-a)^{1/2} \cdot \tilde{B} \left( \frac{t - a}{b-a} \right) + \left(\frac{b-t}{b-a} \right) \cdot x + \left( \frac{t- a}{b-a}\right) \cdot y, 
\end{equation}
and refer to the law of this random continuous function in $C([a,b])$ as a {\em Brownian bridge from $B(a) = x$ to $B(b) = y$ (with diffusion parameter $1$)}. Given $k \in \mathbb{N}$ and $\vec{x}, \vec{y} \in \mathbb{R}^k$, we let $\mathbb{P}^{a,b, \vec{x},\vec{y}}_{\mathrm{free}}$ denote the law of $k$ independent Brownian bridges $\{B_i: [a,b] \mapsto \mathbb{R} \}_{i = 1}^k$ from $B_i(a) = x_i$ to $B_i(b) = y_i$.

The following definition introduces the notion of an $(f,g)$-avoiding Brownian line ensemble.
\begin{definition}\label{DefAvoidingLaw}
Let $k \in \mathbb{N}$ and $\weyl_k$ denote the open Weyl chamber in $\mathbb{R}^{k}$, i.e.
$$\weyl_k = \{ \vec{x} = (x_1, \dots, x_k) \in \mathbb{R}^k\colon x_1 > x_2 > \cdots > x_k \}.$$
Let $\vec{x}, \vec{y} \in \weyl_k$, $a,b \in \mathbb{R}$ with $a < b$, and $f: [a,b] \rightarrow (-\infty, \infty]$ and $g: [a,b] \rightarrow [-\infty, \infty)$ be two continuous functions. We also assume that $f(t) > g(t)$ for all $t \in[a,b]$, $f(a) > x_1$, $f(b) > y_1$ and $g(a) < x_k$, $g(b) < y_k$.

We define the {\em $(f,g)$-avoiding Brownian line ensemble on the interval $[a,b]$ with entrance data $\vec{x}$ and exit data $\vec{y}$} to be the $\{ 1, \dots, k\}$-indexed line ensemble $\mathcal{Q}$ on $\Lambda = [a,b]$ and with the law of $\mathcal{Q}$ equal to the law of $k$ independent Brownian bridges $\{B_i: [a,b] \mapsto \mathbb{R} \}_{i = 1}^k$ from $B_i(a) = x_i$ to $B_i(b) = y_i$, conditioned on the event 
$$E  = \left\{ f(t) > B_1(t) > B_2(t) > \cdots > B_k(t) > g(t) \mbox{ for all $t \in[a,b]$} \right\}.$$ 
We denote the probability distribution of $\mathcal{Q}$ by $\mathbb{P}_{\mathrm{avoid}}^{a,b, \vec{x}, \vec{y}, f, g}$ and write $\mathbb{E}_{\mathrm{avoid}}^{a,b, \vec{x}, \vec{y}, f, g}$ for the expectation with respect to this measure. 
\end{definition}

With the above notation in place, we can formulate the Brownian Gibbs property for $\mathbb{N}$-indexed line ensembles on $\mathbb{R}$.
\begin{definition}\label{Def.BGP}
An $\mathbb{N}$-indexed line ensemble $\mathcal{L} = \{\mathcal{L}_i\}_{i \geq 1}$ on $\mathbb{R}$ is said to have the {\em Brownian Gibbs property}, if it is non-intersecting, meaning that almost surely
$$\mathcal{L}_i(t) > \mathcal{L}_{i+1}(t) \mbox{ for all } i \geq 1, \mbox{ and } t \in \mathbb{R},$$
and the following holds for all $ a < b$ and $1 \leq k_1 \leq k_2$. If we set $K = \{k_1, k_1 + 1, \dots, k_2\}$, then for any bounded Borel-measurable function $F: C(K \times [a,b]) \mapsto \mathbb{R}$, we have $\mathbb{P}$-almost surely
\begin{equation}\label{BGPTower}
\mathbb{E} \left[ F\left(\mathcal{L}|_{K \times [a,b]} \right)  {\big \vert} \mathcal{F}_{\mathrm{ext}} (K \times (a,b))  \right] =\mathbb{E}_{\mathrm{avoid}}^{a,b, \vec{x}, \vec{y}, f, g} \bigl[ F(\tilde{\mathcal{Q}}) \bigr].
\end{equation}
On the left side of (\ref{BGPTower}), we have that
$$\mathcal{F}_{\mathrm{ext}} (K \times (a,b)) = \sigma \left \{ \mathcal{L}_i(s): (i,s) \in (\mathbb{N} \times \mathbb{R}) \setminus (K \times (a,b)) \right\},$$
is the $\sigma$-algebra generated by the variables in the brackets above and $ \mathcal{L}|_{K \times [a,b]}$ is the restriction of $\mathcal{L}$ to the set $K \times [a,b]$. On the right side of (\ref{BGPTower}), we have $\vec{x} = (\mathcal{L}_{k_1}(a), \dots, \mathcal{L}_{k_2}(a))$, $\vec{y} = (\mathcal{L}_{k_1}(b), \dots, \mathcal{L}_{k_2}(b))$, $f = \mathcal{L}_{k_1 - 1}[a,b]$ (the restriction of $\mathcal{L}$ to the set $\{k_1 - 1 \} \times [a,b]$) with the convention that $f = \infty$ if $k_1 = 1$, and $g = \mathcal{L}_{k_2 +1}[a,b]$. 
In addition, $\mathcal{Q} = \{\mathcal{Q}_i\}_{i = 1}^{k_2 - k_1 + 1}$ has law $\mathbb{P}_{\mathrm{avoid}}^{a,b, \vec{x}, \vec{y}, f, g}$, and $\tilde{\mathcal{Q}} = \{\tilde{\mathcal{Q}}_{i}\}_{i = k_1}^{k_2}$ satisfies $\tilde{\mathcal{Q}}_i = \mathcal{Q}_{i - k_1 + 1}$.
\end{definition}

We end this section with two results we require for our arguments below. The first provides a monotone coupling of the measures $\mathbb{P}_{\mathrm{avoid}}^{a,b, \vec{x}, \vec{y}, f, g}$ from Definition \ref{DefAvoidingLaw} in their boundary data. This statement can be deduced by combining \cite[Lemma 2.6 and Lemma 2.7]{CorHamA}, although the formulation below appears in \cite[Lemma A.6]{DimMat}. 
\begin{lemma}\label{Lem.MonCoup} Assume the same notation as in Definition \ref{DefAvoidingLaw}. Fix $k \in \mathbb{N}$, $a < b$ and two continuous functions $g^{\mathsf{t}}, g^{\mathsf{b}}: [a,b] \mapsto [-\infty, \infty)$ such that $g^{\mathsf{t}}(t) \geq g^{\mathsf{b}}(t)$ for all $t \in [a,b]$. We also fix $\vec{x}, \vec{y}, \vec{x}', \vec{y}' \in \weyl_k$ such that $g^{\mathsf{b}}(a) < x_k$, $g^{\mathsf{b}}(b) < y_k$, $g^{\mathsf{t}}(a) < x_k'$, $g^{\mathsf{t}}(b) < y_k'$, and $x_i \leq x_i'$, $y_i \leq y_i'$ for $i = 1,\dots, k$. Then, there exists a probability space $(\Omega, \mathcal{F}, \mathbb{P})$, which supports two $\{1,\dots, k\}$-indexed line ensembles $\mathcal{L}^{\mathsf{t}}$ and $\mathcal{L}^{\mathsf{b}}$ on $[a,b]$, such that the law of $\mathcal{L}^{\mathsf{t}}$ {\big (}resp. $\mathcal{L}^{\mathsf{b}}${\big )} under $\mathbb{P}$ is $\mathbb{P}_{\mathrm{avoid}}^{a,b, \vec{x}', \vec{y}', \infty, g^{\mathsf{t}}}$ {\big (}resp. $\mathbb{P}_{\mathrm{avoid}}^{a,b, \vec{x}, \vec{y}, \infty, g^{\mathsf{b}}}${\big )}, and such that $\mathbb{P}$-almost surely we have $\mathcal{L}_i^{\mathsf{t}}(t) \geq \mathcal{L}^{\mathsf{b}}_i(t)$ for all $i = 1,\dots, k$ and $t \in [a,b]$.
\end{lemma}

The second result we require shows that the Brownian Gibbs property is preserved under simultaneous vertical/horizontal shifts and Brownian scaling.
\begin{lemma}\label{Lem.Affine} Let $\mathcal{L}$ be an $\mathbb{N}$-indexed line ensemble on $\mathbb{R}$ that satisfies the Brownian Gibbs property. Then, for any $\lambda > 0$ and $h, v \in \mathbb{R}$ the line ensemble $\tilde{\mathcal{L}} = \{\tilde{\mathcal{L}}_i\}_{i \geq 1}$, defined by
$$\tilde{\mathcal{L}}_i(t) = \lambda^{-1}\mathcal{L}_i(\lambda^2 t + h) + v, \mbox{ for } i \geq 1, t \in \mathbb{R},$$
also satisfies the Brownian Gibbs property.
\end{lemma}
\begin{proof} This follows from \cite[Lemma 2.1]{DS25} with $N = \infty$, $A_i = 0$ for $i \geq 1$.
\end{proof}

%
%
\subsection{Proof of Theorem \ref{Thm.Extreme}: Part I}\label{Section6.2} Denote the law of $\mathcal{L}^{a,b,c}$ on $C(\mathbb{N} \times \mathbb{R})$ by $\mu$. Suppose that $\mathcal{L}^1, \mathcal{L}^2$ are $\mathbb{N}$-indexed line ensembles on $\mathbb{R}$, which have the Brownian Gibbs property, and whose laws $\mu_1, \mu_2$ satisfy 
\begin{equation}\label{Eq.Convex}
\mu = \alpha \mu_1 + (1-\alpha) \mu_2,
\end{equation} 
for some $\alpha \in (0,1)$. To prove the theorem, we need to show
\begin{equation}\label{Eq.ThmExtreme}
\mu = \mu_1 \mbox{ and } \mu = \mu_2. 
\end{equation}

In the following steps, we reduce the proof of (\ref{Eq.ThmExtreme}) to a sequence of intermediate statements that are easier to establish.\\

{\bf \raggedleft Step 1.} Fix $L, R \in \mathbb{Z}_{\geq 0}$, such that $L \leq J_b^+$ and $R \leq J_a^+$. Let $\mathcal{L}^{L,R} = \{\mathcal{L}^{L,R}_{i}\}_{i \geq 1}$ be the line ensemble whose law agrees with $\mathcal{L}^{\tilde{a}, \tilde{b}, 0}$, where $(\tilde{a}, \tilde{b}, 0) \in \parP$ are given by
$$\tilde{a}_i^+ = a_i^+ \mbox{, for $i = 1, \dots, R$, }\tilde{b}_i^+ = b_i^+ \mbox{, for $i = 1, \dots, L$, and } \tilde{a}_{i+R}^+ = \tilde{b}_{i+L}^+ = 0 \mbox{ for }i \geq 1.$$ 

We claim that for any $j \in \{1,2\}$, $m \in \mathbb{N}$, $t_1, \dots, t_m, u_1, \dots, u_m \in \mathbb{R}$ and $n_1, \dots, n_m \in \mathbb{N}$, we have
\begin{equation}\label{Eq.ExtRed1}
\mathbb{P}(\mathcal{L}^j_{n_i}(t_i) \leq u_i \mbox{ for } i =1, \dots, m ) \leq \mathbb{P}(\mathcal{L}^{L,R}_{n_i}(t_i) \leq u_i \mbox{ for } i =1, \dots, m ).
\end{equation}

Let us assume (\ref{Eq.ExtRed1}), and deduce (\ref{Eq.ThmExtreme}). Define $L_n = \min(J_b^+, n)$ and $R_n = \min(J_a^+, n)$, and note from (\ref{Eq.ExtRed1}) that
\begin{equation*}
\mathbb{P}(\mathcal{L}^j_{n_i}(t_i) \leq u_i \mbox{ for } i =1, \dots, m ) \leq \mathbb{P}(\mathcal{L}^{L_n,R_n}_{n_i}(t_i) \leq u_i \mbox{ for } i =1, \dots, m ).
\end{equation*} 
From Proposition \ref{S13P2}, we know $\mathcal{L}^{L_n,R_n} \Rightarrow \mathcal{L}^{a,b,c}$, and by the continuous mapping theorem $\mathcal{L}^{L_n,R_n} \overset{f.d.}{\rightarrow}\mathcal{L}^{a,b,c}$. By the Brownian Gibbs property satisfied by $\mathcal{L}^{a,b,c}$, we know that the distributions of $\mathcal{L}^{a,b,c}_{n_i}(t_i)$ have no atoms, and so we conclude by taking the $n \rightarrow \infty$ limit of the last equation that
\begin{equation*}
\mathbb{P}(\mathcal{L}^j_{n_i}(t_i) \leq u_i \mbox{ for } i =1, \dots, m ) \leq \mathbb{P}(\mathcal{L}^{a,b,c}_{n_i}(t_i) \leq u_i \mbox{ for } i =1, \dots, m ).
\end{equation*} 

Since $\mu = \alpha \mu_1 + (1-\alpha) \mu_2$ and $\alpha \in (0,1)$, we conclude for $j \in \{1,2\}$
\begin{equation*}
\begin{split}
 \mathbb{P}(\mathcal{L}^j_{n_i}(t_i) \leq u_i \mbox{ for } i =1, \dots, m ) = \mathbb{P}(\mathcal{L}^{a,b,c}_{n_i}(t_i) \leq u_i \mbox{ for } i =1, \dots, m ) . 
\end{split}
\end{equation*}
As finite-dimensional sets uniquely determine the law of a line ensemble, see \cite[Lemma 3.1]{DimMat}, we conclude (\ref{Eq.ThmExtreme}).\\

{\bf \raggedleft Step 2.} Fix $\delta > 0$, and define
\begin{equation}\label{Eq.DeltaPerturbation}
\begin{split}
&\hat{a}_i^+ = e^{-\delta i} \cdot a_i^+ \mbox{, for } i = 1, \dots, R, \mbox{ and } \hat{a}_i^+ = 0 \mbox{, for } i \geq R+1,\\
&\hat{b}_i^+ = e^{-\delta i} \cdot b_i^+ \mbox{, for } i = 1, \dots, L, \mbox{ and } \hat{b}_i^+ = 0 \mbox{, for } i \geq L+1.
\end{split}
\end{equation}
We also let $\hat{\mathcal{L}}^{\delta}$ denote the ensemble with the same law as $\mathcal{L}^{\hat{a}, \hat{b}, 0}$, with $(\hat{a}, \hat{b}, 0)$ as above. 

We claim that for any $\delta > 0$, we have 
\begin{equation}\label{Eq.ExtRed2}
\mathbb{P}(\mathcal{L}^1_{n_i}(t_i) \leq u_i \mbox{ for } i =1, \dots, m ) \leq \mathbb{P}(\hat{\mathcal{L}}^{\delta}_{n_i}(t_i) \leq u_i \mbox{ for } i =1, \dots, m ).
\end{equation}
Notice that by Proposition \ref{S13P2} and the continuous mapping theorem, we have $\hat{\mathcal{L}}^{\delta} \overset{f.d.}{\rightarrow} \mathcal{L}^{L,R}$ as $\delta \rightarrow 0+$. Similarly to Step 1, we conclude that (\ref{Eq.ExtRed1}) with $j = 1$ holds by taking the $\delta \rightarrow 0+$ limit of (\ref{Eq.ExtRed2}). Upon swapping $\mu_1$ with $\mu_2$, and $\alpha$ with $1-\alpha$, we see that (\ref{Eq.ExtRed1}) holds for $j = 2$ as well.\\

{\bf \raggedleft Step 3.} Fix $\varepsilon > 0$. We claim that we can find sequences $\lambda_N > 1$, $r_N > 0$, $x^N = (x^N_1, \dots, x^N_{N-1}) \in \weyl_{N-1}$, $y^N = (y^N_1, \dots, y^N_{N-1}) \in \weyl_{N-1}$, and $g_N \in C([-N,N])$, depending on $\varepsilon$, $\alpha$, $\delta$, $\{a_i^+\}_{i \geq 1}$, $\{b_i^+\}_{i \geq 1}$, $L$ and $R$, that satisfy the following statements for all sufficiently large $N$, depending on the same set of parameters.

\begin{enumerate} 
	\item[(a)] The following limits both hold
\begin{equation}\label{Eq.LambdaRNLimit}
\lim_{N \rightarrow \infty} \lambda_N = 1 \mbox{ and } \lim_{N \rightarrow \infty} r_N = 0.
\end{equation}
\item[(b)] For all large $N$
\begin{equation}\label{Eq.XYLim}
x_{N-1}^N > g_N(-N), \hspace{2mm} y_{N-1}^N > g_N(N).
\end{equation}
\item[(c)] For all large $N$ 
\begin{equation}\label{Eq.SideLRightLB}
\mathbb{P}\big( E_N^{\mathrm{right}} \big) < (R+1)\varepsilon/\alpha, \mbox{ where } E_N^{\mathrm{right}} = \left\{\mathcal{L}_i^{1}(N) \leq y_i^N \mbox{ for some } i \in\{1,\dots, N-1\} \right\},
\end{equation}
\begin{equation}\label{Eq.SideLLeftLB}
\mathbb{P}\big( E_N^{\mathrm{left}} \big) < (L+1)\varepsilon/\alpha,\mbox{ where }E_N^{\mathrm{left}} = \left\{\mathcal{L}_i^{1}(-N) \leq x_i^N \mbox{ for some } i \in\{1,\dots, N-1\} \right\}.
\end{equation}
\item[(d)] For all large $N$ 
\begin{equation}\label{Eq.BotLLB}
\mathbb{P}\big( E_N^{\mathrm{bot}} \big) < \varepsilon/\alpha, \mbox{ where } E_N^{\mathrm{bot}} = \left\{\mathcal{L}_N^{1}(s) \leq g_N(s) \mbox{ for some $s \in [-N,N]$} \right\}.
\end{equation}
\end{enumerate} 

In addition, if we define the ensembles $\hat{\mathcal{L}}^{\delta, N} = \{\hat{\mathcal{L}}^{\delta, N}_i\}_{i \geq 1}$ via
\begin{equation}\label{Eq.ScaledEnsembleDef}
\hat{\mathcal{L}}^{\delta, N}_i(t) = \lambda_N^{-1} \hat{\mathcal{L}}^{\delta}_i (\lambda_N^2 t) - r_N,
\end{equation}
then the following statements also hold.
\begin{enumerate} 
\item[(e)] For all large $N$ 
\begin{equation}\label{Eq.SideLHatRightUB}
\mathbb{P}\big( \hat{E}_N^{\mathrm{right}} \big) < (R+1)\varepsilon, \mbox{ where } \hat{E}_N^{\mathrm{right}} = \left\{ \hat{\mathcal{L}}_i^{\delta,N}(N) \geq y_i^N \mbox{ for some } i \in\{1,\dots, N-1\} \right\},
\end{equation}
\begin{equation}\label{Eq.SideLHatLeftUB}
\mathbb{P}\big( \hat{E}_N^{\mathrm{left}} \big) < (L+1)\varepsilon, \mbox{ where }\hat{E}_N^{\mathrm{left}} = \left\{\hat{\mathcal{L}}^{\delta, N}_i(-N) \geq x_i^N \mbox{ for some } i \in\{1,\dots, N-1\} \right\}.
\end{equation}
\item[(f)] For all large $N$ 
\begin{equation}\label{Eq.BotLHatUB}
\mathbb{P}\big( \hat{E}_N^{\mathrm{bot}} \big) < \varepsilon, \mbox{ where } \hat{E}_N^{\mathrm{bot}} = \left\{\hat{\mathcal{L}}_N^{\delta,N}(s) \geq g_N(s) \mbox{ for some $s \in [-N,N]$}\right\}.
\end{equation}
\end{enumerate} 

{\em \raggedleft Explanation.} Condition (c) ensures that $\mathcal{L}^1$ at time $N$ is likely {\em higher} than $y^N$, and at time $-N$ is likely {\em higher} than $x^N$. On the other hand, condition (e) ensures that $\hat{\mathcal{L}}^{\delta,N}$ at time $N$ is likely {\em lower} than $y^N$, and at time $-N$ is likely {\em lower} than $x^N$. Condition (d) ensures that the $N$-th curve $\mathcal{L}^1_N$ is likely {\em higher} than $g_N$ on the whole interval $[-N,N]$, while condition (f) ensures that $\hat{\mathcal{L}}^{\delta,N}_N$ is likely {\em lower} than $g_N$ on $[-N,N]$.\\

In what follows we suppose that we can find $\lambda_N$, $r_N$, $x^N$, $y^N$ and $g_N$ as above, and proceed to prove (\ref{Eq.ExtRed2}). For convenience we set 
$$E_N^{\mathrm{bdry}} = E_N^{\mathrm{left}} \cup E_N^{\mathrm{right}} \cup E_N^{\mathrm{bot}} \mbox{ and } \hat{E}_N^{\mathrm{bdry}} = \hat{E}_N^{\mathrm{left}} \cup \hat{E}_N^{\mathrm{right}} \cup \hat{E}_N^{\mathrm{bot}},$$
and note from (c-f) above that for all large $N$
\begin{equation}\label{Eq.LikelyBoundary}
\mathbb{P}\big( E_N^{\mathrm{bdry}} \big) < (L + R + 3)\varepsilon/\alpha \mbox{ and } \mathbb{P}\big( \hat{E}_N^{\mathrm{bdry}} \big) < (L + R + 3)\varepsilon.
\end{equation}

Suppose $N$ is large enough so that (\ref{Eq.XYLim}) and (\ref{Eq.LikelyBoundary}) hold, and $n_i, |t_i| \leq N-1$ for $i = 1, \dots, m$. By the Brownian Gibbs property of $\mathcal{L}^{1}$ and the monotone coupling Lemma \ref{Lem.MonCoup}, we have
\begin{equation*}
\begin{split}
&{\bf 1}_{(E_N^{\mathrm{bdry}})^c} \cdot \mathbb{E}\left[ {\bf 1} \left\{\mathcal{L}^1_{n_i}(t_i) \leq u_i \mbox{ for } i =1, \dots, m  \right\} \big{\vert} \mathcal{F}^1_{\mathrm{ext}} (\{1, \dots, N-1\} \times (-N,N))\right] \\
& \leq {\bf 1}_{(E_N^{\mathrm{bdry}})^c} \cdot \mathbb{E}_{\mathrm{avoid}}^{-N, N, x^N, y^N, \infty, g_N}\left[ {\bf 1} \left\{\mathcal{Q}_{n_i}(t_i) \leq u_i \mbox{ for } i =1, \dots, m  \right\} \right]. 
\end{split}
\end{equation*}
Taking expectations on both sides, using the tower property for conditional expectation, and utilizing the first inequality in (\ref{Eq.LikelyBoundary}), we conclude
\begin{equation}\label{Eq.LProbabilityLB}
\begin{split}
&\mathbb{P}\left(\mathcal{L}^1_{n_i}(t_i) \leq u_i \mbox{ for } i =1, \dots, m  \right) \\
& \leq  \mathbb{P}_{\mathrm{avoid}}^{-N, N, x^N, y^N, \infty, g_N}\left( \mathcal{Q}_{n_i}(t_i) \leq u_i \mbox{ for } i =1, \dots, m   \right) + (L + R + 3)\varepsilon/\alpha. 
\end{split}
\end{equation}

An analogous argument applied to $\hat{\mathcal{L}}_N^{\delta,N}$ yields for all large $N$
\begin{equation}\label{Eq.LHatProbabilityUB}
\begin{split}
&\mathbb{P}\left(\hat{\mathcal{L}}^{\delta, N}_{n_i}(t_i) \leq u_i \mbox{ for } i =1, \dots, m  \right) \\
& \geq  \mathbb{P}_{\mathrm{avoid}}^{-N, N, x^N, y^N, \infty, g_N}\left( \mathcal{Q}_{n_i}(t_i) \leq u_i \mbox{ for } i =1, \dots, m   \right) - (L + R + 3)\varepsilon. 
\end{split}
\end{equation}
We mention that $\hat{\mathcal{L}}^{\delta,N}$ satisfies the Brownian Gibbs property in view of Lemma \ref{Lem.Affine}.\\

Combining (\ref{Eq.LProbabilityLB}) and (\ref{Eq.LHatProbabilityUB}), we conclude for all large $N$
$$\mathbb{P}\left(\mathcal{L}^1_{n_i}(t_i) \leq u_i \mbox{ for } i =1, \dots, m  \right) - \mathbb{P}\left(\hat{\mathcal{L}}^{\delta, N}_{n_i}(t_i) \leq u_i \mbox{ for } i =1, \dots, m  \right) \leq (1+1/\alpha)(L+R+3)\varepsilon.$$

Letting $N \rightarrow \infty$ in the last inequality and using the fact that $\hat{\mathcal{L}}^{\delta, N} \Rightarrow \hat{\mathcal{L}}^{\delta}$ (which follows from (\ref{Eq.LambdaRNLimit}), (\ref{Eq.ScaledEnsembleDef}) and the continuous mapping theorem), we conclude  
$$\mathbb{P}\left(\mathcal{L}^1_{n_i}(t_i) \leq u_i \mbox{ for } i =1, \dots, m  \right) - \mathbb{P}\left(\hat{\mathcal{L}}^{\delta}_{n_i}(t_i) \leq u_i \mbox{ for } i =1, \dots, m  \right) \leq (1+1/\alpha)(L+R+3)\varepsilon.$$
Letting $\varepsilon \rightarrow 0+$ in the last inequality, we get (\ref{Eq.ExtRed2}).\\

%
%
{\bf \raggedleft Outline of Part II.} In the remainder of this section we discuss at a high level the way we construct the sequences $\lambda_N$, $r_N$, $x^N$, $y^N$ and $g_N$, satisfying conditions (a-f) above. See also Figure \ref{S61}.\\
\begin{figure}[h]
\centering
\includegraphics[width=0.8\textwidth]{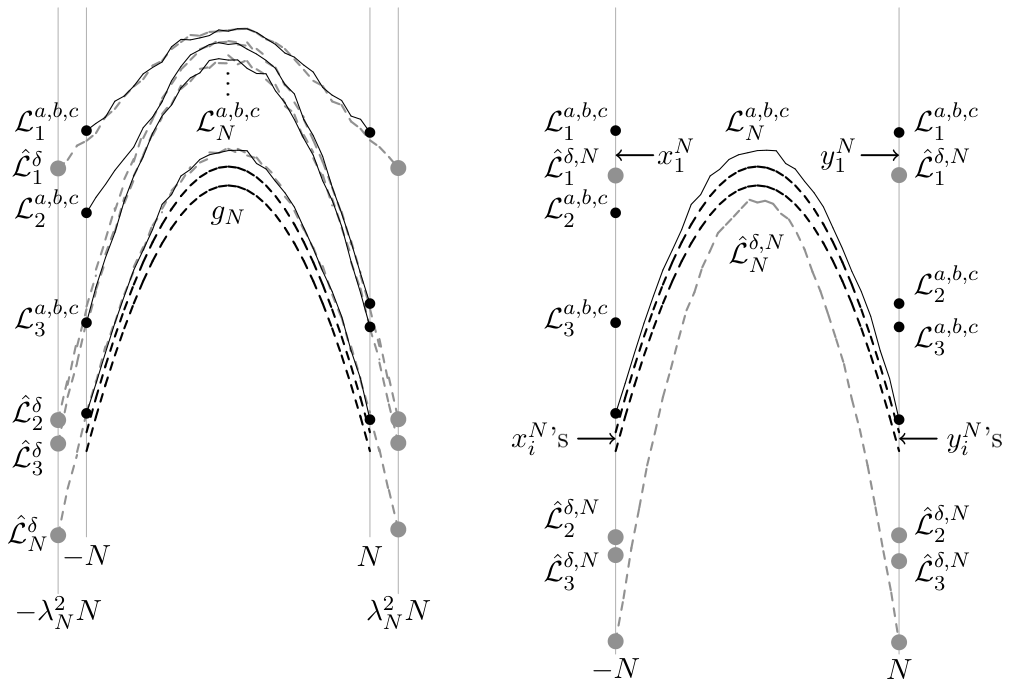}
\caption{The left figure depicts the top $N$ curves in $\mathcal{L}^{a,b,c}$ (in black) on $[-N,N]$ and $\hat{\mathcal{L}}^{\delta}$ (in dashed gray) on $[-\lambda_N^2 N, \lambda_N^2 N]$, for $J_a^+ = L = R = 1$ and $J_b^+ = 2$. The two dashed black parabolas are $g_N$ and $g_N + N^{-1/6}$. \\
\vspace{0mm}
\hspace{5mm}The right figure depicts the curve endpoints $\mathcal{L}_i^{a,b,c}(\pm N)$ and $\hat{\mathcal{L}}^{\delta,N}_i(\pm N)$ for $i = 1, \dots, N-1$, as well as the bottom curves $\mathcal{L}_N^{a,b,c}$ and $\hat{\mathcal{L}}^{\delta,N}_N$ on $[-N,N]$. The figure also depicts the locations of $x_i^N$ and $y_i^N$, which, for $i = 2, \dots, N-1$, are all squeezed between $g_N(\pm N)$ and $g_N(\pm N) + N^{-1/6}$. Distances in the figure have been exaggerated for clarity.}
\label{S61}
\end{figure}

Firstly, we will take $\lambda_N = 1 + \Theta(N^{-1})$ and $r_N = \Theta(N^{-1/12})$. Both exponents here are not very important, and can be modified, but these will suffice for our purposes. 

One major difficulty is finding a function $g_N$, satisfying conditions (d) and (f). The crucial observation here is that by known bulk rigidity estimates for the parabolic Airy line ensemble $\mathcal{L}^{0,0,0}$ (as in Remark \ref{Rem.AWLE2}), we know that the $N$-th curve $\mathcal{L}^{0,0,0}_N$ {\em concentrates} near the parabola $f_N(t) = -2^{-1/2}t^2 - (3\pi)^{2/3} 2^{-7/6} N^{2/3}$, see Lemma \ref{Lem.Parabolicity} for a precise statement. We take $g_N$ to be {\em slightly} below this parabola, specifically $g_N = f_N - 2N^{-1/6}$.

By our choice of $g_N$, we know that $g_N$ is a likely lower bound for $\mathcal{L}^{0,0,0}_N$, and since $\mathcal{L}^{a,b,c}_N$ stochastically dominates $\mathcal{L}^{0,0,0}_N$, it is a likely lower bound for it as well. Using (\ref{Eq.Convex}), we can transfer this property to $\mathcal{L}^{1}_N$, establishing the lower bound estimate in (d). 

On the other hand, we know from Theorem \ref{Thm.MonCoupling} that we can stochastically ``squeeze'' $\hat{\mathcal{L}}^{\delta}_N$ between $\mathcal{L}^{0,0,0}_{N-K}$ and $\mathcal{L}^{0,0,0}_{N}$. As $f_{N-K}$ and $f_N$ are both very close to $g_N$, we see that $\hat{\mathcal{L}}^{\delta}_N$ is likely very close to $g_N$. To make it likely {\em below} $g_N$, we need to shift it down slightly -- this is the origin of $r_N$ in the scaled ensemble $\hat{\mathcal{L}}^{\delta, N}$ in (\ref{Eq.ScaledEnsembleDef}). This is how we accomplish the upper bound estimate in (f). At this point we do not need the Brownian scaling by $\lambda_N$ in (\ref{Eq.ScaledEnsembleDef}), and in fact implementing it only improves the upper bound estimate in (\ref{Eq.BotLHatUB}). We will shortly explain the role of this scaling.\\

Focusing on (c), we can use the fact that by Theorem \ref{Thm.Slopes}(a), we know that $\mathcal{L}_i^{a,b,c}$ (and hence $\mathcal{L}_i^1$ by (\ref{Eq.Convex})) have asymptotic slopes $-\sqrt{2}/a_i^+$ near $\infty$ for $i = 1,\dots,R$. Focusing on (e), by our $\delta$-perturbation of the parameters in (\ref{Eq.DeltaPerturbation}), we know that $\hat{\mathcal{L}}^{\delta}_i$ have asymptotic slopes $-\sqrt{2}/\hat{a}_i^+ < -\sqrt{2}/a_i^+$ for $i = 1, \dots, R$. This enables us to find $y_1^N < \dots < y^N_{R}$, so that with high probability
$$  -N\sqrt{2}/\hat{a}_i^+ < y^N_i < -N\sqrt{2}/a_i^+ \mbox{, and hence } \hat{\mathcal{L}}^{\delta}_i(N)< y^N_i < \mathcal{L}^1_i(N).$$
This takes care of the bounds in (\ref{Eq.SideLRightLB}) and (\ref{Eq.SideLHatRightUB}) for $i = 1, \dots, R$. Here, the Brownian scaling in (\ref{Eq.ScaledEnsembleDef}) and the shift by $r_N$ do not play a substantial role.

In order to handle the bounds in (\ref{Eq.SideLRightLB}) and (\ref{Eq.SideLHatRightUB}) for $i \geq R+1$, we finally need to utilize the Brownian scaling in (\ref{Eq.ScaledEnsembleDef}). For $\mathcal{L}^1_{R+i}(N)$, we have the likely lower bound
$$\mathcal{L}^1_{R+i}(N) \geq \mathcal{L}_N^1(N) \geq f_N(N) - N^{-1/6} = g_N(N) + N^{-1/6}.$$ 
On the other hand, by Theorem \ref{Thm.Slopes}(c), we have the likely upper bound for all $i \geq 1$
$$\hat{\mathcal{L}}^{\delta, N}_{R+i}(N) \leq \hat{\mathcal{L}}^{\delta, N}_{R+1}(N) \approx -2^{-1/2} \lambda_N^3 N^2 + O(1) \leq -2^{-1/2}N^2 - N^{3/4},$$
where the last inequality used $\lambda_N = 1 + \Theta(N^{-1})$, and $N^{3/4}$ could be replaced by a small constant times $N$. Since $N^{3/4}$ dominates $(3\pi)^{2/3} 2^{-7/6} N^{2/3}$, we can find $y^N_{R+1}> \dots > y^N_{N-1}$, so that 
$$ g_N(N) + N^{-1/6} > y^N_{R+1}> \dots > y^N_{N-1} > g_N(N) >  -2^{-1/2}N^2 - N^{3/4}.$$

Putting together the last three displayed inequalities, we deduce the remaining bounds $i = R+1, \dots, N-1$ in (\ref{Eq.SideLRightLB}) and (\ref{Eq.SideLHatRightUB}). The bounds (\ref{Eq.SideLLeftLB}) and (\ref{Eq.SideLHatLeftUB}) are handled analogously, by replacing $R$ with $L$, $a_i^+, \hat{a}_i^+$ with $b_i^+, \hat{b}_i^+$, and $y^N$ with $x^N$. This completes (c) and (e).

%
%
\subsection{Proof of Theorem \ref{Thm.Extreme}: Part II}\label{Section6.3} In this section we provide the details behind the outline from the previous section, and construct the sequences $\lambda_N$, $r_N$, $x^N$, $y^N$ and $g_N$, satisfying conditions (a-f) there. In the rest of this section we fix $\varepsilon$, $\alpha$, $\delta$, $\{a_i^+\}_{i \geq 1}$, $\{b_i^+\}_{i \geq 1}$, $L$ and $R$ as in Section \ref{Section6.2}. All inequalities below hold provided that $N$ is sufficiently large depending on these parameters --- we do not mention this further.\\

{\bf \raggedleft Specifying the sequences.} We define $\lambda_N$ and $r_N$ through
\begin{equation}\label{Eq.DefLambdaRN}
\lambda_N = e^{1/N}, \hspace{2mm} r_N = N^{-1/12},
\end{equation}
and note that (\ref{Eq.LambdaRNLimit}) holds, verifying condition (a). In addition, we define 
\begin{equation}\label{Eq.DefGN} 
g_N(t) = -2^{-1/2}t^2 - (3\pi)^{2/3} 2^{-7/6} N^{2/3} - 2N^{-1/6}.
\end{equation}

We next turn to specifying $x^N, y^N$. From (\ref{Eq.DeltaPerturbation}), we have 
$$ a_i^+ > \hat{a}_i^+ \mbox{ for } i = 1, \dots, R;\hspace{2mm} b_i^+ > \hat{b}_i^+, \mbox{ for } i = 1, \dots, L.$$
From the last set of inequalities, we conclude that we can find $(\alpha_1, \dots, \alpha_{R}) \in \weyl_R$ and $(\beta_1, \dots, \beta_L) \in \weyl_L$, such that  
\begin{equation}\label{Eq.ParSqueeze}
-\sqrt{2}/\hat{a}_i^+ < \alpha_i < - \sqrt{2}/a_i^+ \mbox{ for } i =1, \dots, R, \mbox{ and }-\sqrt{2}/\hat{b}_i^+ < \beta_i < - \sqrt{2}/b_i^+ \mbox{ for } i = 1, \dots, L.
\end{equation}
For example, we can set $\alpha_i = - \sqrt{2}/a_i^+ - i \rho$, $\beta_j = - \sqrt{2}/b_j^+ - j \rho$ with $\rho > 0$ sufficiently close to zero.

We now define 
\begin{equation}\label{Eq.DefXY}
\begin{split}
&x_i^N = \beta_i N, \mbox{ for } i = 1, \dots, L, \mbox{ and } x_{L+i}^N = g_N(-N) +  \frac{N-i}{N^2} \mbox{, for } i =1, \dots, N -L - 1; \\
&y_i^N = \alpha_i N, \mbox{ for } i = 1, \dots, R, \mbox{ and } y_{R+i}^N = g_N(N) + \frac{N-i}{N^2} \mbox{, for } i =1, \dots, N -R - 1. 
\end{split}
\end{equation}
Using that $(\alpha_1, \dots, \alpha_{R}) \in \weyl_R$, $(\beta_1, \dots, \beta_L) \in \weyl_L$, and the quadratic decay of $g_N(\pm N)$ in $N$, we conclude from (\ref{Eq.DefXY}) that $x^N, y^N \in \weyl_{N-1}$ for all large $N$. In addition, the same equation implies (\ref{Eq.XYLim}), verifying condition (b).\\

In our work so far, we specified the five sequences $\lambda_N$, $r_N$, $x^N$, $y^N$ and $g_N$, and showed that they satisfy conditions (a) and (b). We next proceed to verify the other four conditions.\\

{\bf \raggedleft Verifying conditions (d) and (f).} We require the following lemma.
\begin{lemma}\label{Lem.Parabolicity} Let $\mathcal{L}^{0,0,0}$ be as in Proposition \ref{S12AWD} when all parameters $a_i^{\pm}, b_i^{\pm}$ and $c^{\pm}$ are equal to zero. Then, for $f_N(s) = -2^{-1/2}s^2 - (3\pi)^{2/3} 2^{-7/6} N^{2/3}$, we have
\begin{equation}\label{Eq.Parabolicity}
\lim_{N \rightarrow \infty} \sup_{t \in [-N^2, N^2]} N^{1/6} \left| \mathcal{L}^{0,0,0}_{N}(t) -f_N(t) \right| = 0 \mbox{ almost surely.}
\end{equation}
\end{lemma}
\begin{proof} The statement follows from \cite[Theorem 1.6]{DV21} upon setting $k = N$, $\ell_k = N^{1/6}$ and $t_k = N^2$. We mention that the line ensemble $\mathcal{L}$ from \cite{DV21} is related to $\mathcal{L}^{0,0,0}$ via $\sqrt{2} \cdot \mathcal{L}^{0,0,0} =  \mathcal{L}$.
\end{proof}

From Theorem \ref{Thm.MonCoupling} with $A = B = 0$, there is a coupling such that almost surely
$\mathcal{L}^{a,b,c}_N(t) \geq \mathcal{L}^{0,0,0}_N(t)$ for all $t \in \mathbb{R}$. Combining the latter with Lemma \ref{Lem.Parabolicity}, we conclude
\begin{equation*}
\mathbb{P}\left(   \mathcal{L}^{a,b,c}_{N}(s) \leq f_N(s) - N^{-1/6} \mbox{ for some $s \in [-N, N]$} \right) < \varepsilon.
\end{equation*}
The last inequality, the fact that $g_N(s) = f_N(s) - 2N^{-1/6}$ from (\ref{Eq.DefGN}), and (\ref{Eq.Convex}) imply 
\begin{equation}\label{Eq.S63LowerBoundaryL}
\mathbb{P}\left(   \mathcal{L}^{1}_{N}(s) \leq g_N(s) + N^{-1/6} \mbox{ for some $s \in [-N, N]$} \right) < \varepsilon/\alpha,
\end{equation}
which implies (\ref{Eq.BotLLB}), verifying condition (d).\\

From Theorem \ref{Thm.MonCoupling}, applied to $(\hat{a},\hat{b},0)$ and $(0,0,0)$ with $A = R$, $B = L$, there is a coupling such that almost surely
$$\hat{\mathcal{L}}^{\delta}_{N}(t) \leq \mathcal{L}^{0,0,0}_{N-K}(t),$$ 
for all $t \in \mathbb{R}$, where $K = \max(L,R)$. The latter and Lemma \ref{Lem.Parabolicity} show that
\begin{equation*}
\mathbb{P}\left( \hat{\mathcal{L}}^{\delta}_{N}(s) \geq f_{N-K}(s)  +  (N-K)^{-1/6} \mbox{ for some $s \in [-N^2, N^2]$} \right) < \varepsilon.
\end{equation*}
Using that $\lambda_N^2 N \leq N^2$, and the definition of $\hat{\mathcal{L}}^{\delta,N}$ from (\ref{Eq.ScaledEnsembleDef}), we conclude
\begin{equation*}
\mathbb{P}\left( \hat{\mathcal{L}}^{\delta,N}_{N}(s)  \geq  \lambda_N^{-1}f_{N-K}(\lambda_N^2 s) + (N-K)^{-1/6} - r_N \mbox{ for some $s \in [-N, N]$} \right) < \varepsilon.
\end{equation*}

We also note that
\begin{equation*}
\begin{split}
&\lambda_N^{-1}f_{N-K}(\lambda_N^2 s) + (N-K)^{-1/6} -r_N \leq -2^{-1/2} s^2 - \lambda_N^{-1} (3\pi)^{2/3} 2^{-7/6} (N-K)^{2/3} \\
&+ (N-K)^{-1/6} -r_N  \leq f_N(s) + N^{-1/4} + (N-K)^{-1/6} -r_N \leq g_N(s),
\end{split}
\end{equation*}
where in the first two inequalities we used $\lambda_N = e^{1/N}$, and in the last inequality we used that $g_N(s) = f_N(s) - 2N^{-1/6}$ and $r_N = N^{-1/12}$. 

The last two displayed inequalities imply (\ref{Eq.BotLHatUB}), and hence condition (f). \\

{\bf \raggedleft Verifying condition (c).} From Theorem \ref{Thm.Slopes}(a), we know $N^{-1}\mathcal{L}_i^{a,b,c}(N) \Rightarrow -\sqrt{2}/a_i^+$ for $i = 1, \dots, R$. Using that $\alpha_i < - \sqrt{2}/a_i^+$, see (\ref{Eq.ParSqueeze}), and that $y_i^N = \alpha_i N$, see (\ref{Eq.DefXY}), we conclude 
\begin{equation*}
\mathbb{P}\left( \mathcal{L}^{a,b,c}_{i}(N) \leq  y_i^N \right) < \varepsilon \mbox{, for $i = 1, \dots, R$.}
\end{equation*}
The latter and (\ref{Eq.Convex}) imply
\begin{equation}\label{Eq.SideLTopBoundR}
\mathbb{P}\left( \mathcal{L}^{1}_{i}(N) \leq  y_i^N \right) < \varepsilon/\alpha \mbox{, for $i = 1, \dots, R$}.
\end{equation}
An analogous argument using Theorem \ref{Thm.Slopes}(b) yields
\begin{equation}\label{Eq.SideLTopBoundL}
\mathbb{P}\left( \mathcal{L}^{1}_{i}(-N) \leq  x_i^N \right) < \varepsilon/\alpha \mbox{, for $i = 1, \dots, L$}.
\end{equation}

Using that $\mathcal{L}^1_i(t) \geq \mathcal{L}^1_N(t)$ for $1 \leq i \leq N$ and $t \in \mathbb{R}$, and (\ref{Eq.S63LowerBoundaryL}), we conclude 
\begin{equation*}
\mathbb{P}\left(\mathcal{L}^{1}_{i}(N) \leq g_N(N) + N^{-1/6} \mbox{ for some }i \in \{1, \dots, N\}  \right) < \varepsilon/\alpha.
\end{equation*}
Using that $y^N_{R+i} < g_N(N) + N^{-1/6}$, see (\ref{Eq.DefXY}), we conclude
\begin{equation}\label{Eq.SideLBotBoundR}
\mathbb{P}\left(\mathcal{L}^{1}_{i}(N) \leq y_i^N \mbox{ for some }i \in \{R+1, \dots, N\}  \right) < \varepsilon/\alpha.
\end{equation}
An analogous argument yields
\begin{equation}\label{Eq.SideLBotBoundL}
\mathbb{P}\left(\mathcal{L}^{1}_{i}(-N) \leq x_i^N \mbox{ for some }i \in \{L+1, \dots, N\}  \right) < \varepsilon/\alpha.
\end{equation}

By a union bound (\ref{Eq.SideLTopBoundR}) and (\ref{Eq.SideLBotBoundR}) imply (\ref{Eq.SideLRightLB}), while (\ref{Eq.SideLTopBoundL}) and (\ref{Eq.SideLBotBoundL}) imply (\ref{Eq.SideLLeftLB}). This verifies condition (c).\\

{\bf \raggedleft Verifying condition (e).} From Theorem \ref{Thm.Slopes}(a), we know $N^{-1}\hat{\mathcal{L}}_i^{\delta}(N) \Rightarrow -\sqrt{2}/\hat{a}_i^+$ for $i = 1, \dots, R$. Using the definition of $\hat{\mathcal{L}}^{\delta,N}$ from (\ref{Eq.ScaledEnsembleDef}), and that $\lambda_N = e^{1/N}$, we conclude $N^{-1}\hat{\mathcal{L}}_i^{\delta,N}(N) \Rightarrow -\sqrt{2}/\hat{a}_i^+$. Using that $\alpha_i > - \sqrt{2}/\hat{a}_i^+$, see (\ref{Eq.ParSqueeze}), and that $y_i^N = \alpha_i N$, see (\ref{Eq.DefXY}), we conclude 
\begin{equation}\label{Eq.SideLHatTopBoundR}
\mathbb{P}\left( \hat{\mathcal{L}}_i^{\delta,N}(N) \geq  y_i^N \right) < \varepsilon \mbox{, for $i = 1, \dots, R$.}
\end{equation}
An analogous argument using Theorem \ref{Thm.Slopes}(b) yields
\begin{equation}\label{Eq.SideLHatTopBoundL}
\mathbb{P}\left( \hat{\mathcal{L}}_i^{\delta,N}(-N) \geq  x_i^N \right) < \varepsilon \mbox{, for $i = 1, \dots, L$}.
\end{equation}

From Theorem \ref{Thm.Slopes}(c) for $k = R+ 1$, we know that $\hat{\mathcal{L}}_{R+1}^{\delta}(\lambda_N^2 N) + 2^{-1/2}\lambda_N^4 N^2$ is tight. We also observe from (\ref{Eq.DefLambdaRN}) that $\lambda_N^3 = e^{3/N} \geq 1 + 3/N$, and so
$$ - 2^{-1/2}\lambda_N^3 N^2 \leq -2^{-1/2} N^2 - N \leq g_N(N) - N^{3/4} - r_N,$$  
where in the last inequality we used that $r_N = N^{-1/12}$ from (\ref{Eq.DefLambdaRN}) and $g_N(N) = -2^{-1/2}N^2 +O(N^{2/3})$ from (\ref{Eq.DefGN}). Combining the last few statements with the definition of $\hat{\mathcal{L}}^{\delta,N}$ from (\ref{Eq.ScaledEnsembleDef}), we conclude 
\begin{equation*}
\mathbb{P}\left( \hat{\mathcal{L}}_{R+1}^{\delta,N}(N) \geq  g_N(N) \right) \leq \mathbb{P}\left( \lambda_N^{-1}\hat{\mathcal{L}}_{R+1}^{\delta}(\lambda_N^2 N) + 2^{-1/2}\lambda_N^3 N^2 - r_N \geq  N^{3/4} \right) < \varepsilon.
\end{equation*}

Since $y_i^N \geq g_N(N)$, see (\ref{Eq.DefXY}), and $\hat{\mathcal{L}}_{i}^{\delta,N}(N) \leq \hat{\mathcal{L}}_{R+1}^{\delta,N}(N)$ for $i \in \{R+1, \dots, N\} $ we conclude from the last equation that 
\begin{equation}\label{Eq.SideLHatBotBoundR}
\mathbb{P}\left( \hat{\mathcal{L}}_{i}^{\delta,N}(N) \geq  y_i^N \mbox{ for some }i \in \{R+1, \dots, N\} \right) < \varepsilon.
\end{equation}
An analogous argument using Theorem \ref{Thm.Slopes}(d) yields
\begin{equation}\label{Eq.SideLHatBotBoundL}
\mathbb{P}\left( \hat{\mathcal{L}}_{i}^{\delta,N}(-N) \geq  x_i^N \mbox{ for some }i \in \{L+1, \dots, N\} \right) < \varepsilon.
\end{equation}

By a union bound (\ref{Eq.SideLHatTopBoundR}) and (\ref{Eq.SideLHatBotBoundR}) imply (\ref{Eq.SideLHatRightUB}), while (\ref{Eq.SideLHatTopBoundL}) and (\ref{Eq.SideLHatBotBoundL}) imply (\ref{Eq.SideLHatLeftUB}). This verifies condition (e), and concludes the proof of Theorem \ref{Thm.Extreme}.

\bibliographystyle{amsplain} 
\bibliography{PD}

\end{document}